\documentclass{amsart}
\usepackage[margin=1in]{geometry}  
\usepackage{amsmath}               
\usepackage{amsfonts}              
\usepackage{amsthm}      
\usepackage{mathrsfs}                          
\usepackage{amssymb}
\usepackage{bbm}
\usepackage{enumitem}
\usepackage[utf8]{inputenc}
\usepackage{mathtools}
\usepackage[T1]{fontenc}
\usepackage{bbding}
\usepackage{dsfont} 
\usepackage{svg}

\usepackage{xcolor}   
\usepackage{hyperref}
\hypersetup{
	colorlinks,
	linkcolor={solred},
	citecolor={solblue},
	urlcolor={solblue}
}

\definecolor{solyellow}{HTML}{b58900}
\definecolor{solorange}{HTML}{cb4b16}
\definecolor{solred}{HTML}{dc322f}
\definecolor{solmagenta}{HTML}{d33682}
\definecolor{solviolet}{HTML}{6c71c4}
\definecolor{solblue}{HTML}{268bd2}
\definecolor{solcyan}{HTML}{2aa198}
\definecolor{solgreen}{HTML}{859900}

\def\acts{\curvearrowright}
\DeclarePairedDelimiter\abs{\lvert}{\rvert}%
\DeclarePairedDelimiter\norm{\lVert}{\rVert}%

\makeatletter
\let\oldabs\abs
\def\abs{\@ifstar{\oldabs}{\oldabs*}}
\let\oldnorm\norm
\def\norm{\@ifstar{\oldnorm}{\oldnorm*}}

\makeatother

\theoremstyle{definition}

\newtheorem{theorem}{Theorem}[section]

\newtheorem{lemma}[theorem]{Lemma}

\newtheorem{proposition}[theorem]{Proposition}
\newtheorem*{claim}{Claim}

\newtheorem{corollary}[theorem]{Corollary}

\newtheorem{definition}[theorem]{Definition}
\newtheorem{example}[theorem]{Example}

\newtheorem{remark}[theorem]{Remark}

\newtheorem{question}{Question}

\DeclareMathOperator{\cost}{cost}

\DeclareMathOperator{\SL}{SL}

\DeclareMathOperator{\Aut}{Aut}
\DeclareMathOperator{\Isom}{Isom}

\DeclareMathOperator{\bs}{\backslash}

\let\phi\varphi

\newcommand{\EE}{\mathbb{E}}      

\newcommand{\RR}{\mathbb{R}}      

\newcommand{\PP}{\mathbb{P}}      
\newcommand{\MM}{\mathbb{M}}      

\newcommand{\NN}{\mathbb{N}}      
\newcommand{\HH}{\mathbb{H}}      

\newcommand{\1}{\mathbbm{1}}

\newcommand\restr[2]{{
  \left.\kern-\nulldelimiterspace 
  #1 
  \vphantom{\big|} 
  \right|_{#2} 
  }}

\newcommand{\Rel}{\mathcal{O}} 
\newcommand{\MMo}{{\mathbb{M}_\circ}}

\DeclareMathOperator{\covol}{covol}
\DeclareMathOperator{\intensity}{int}

\DeclareMathOperator{\Stab}{Stab}

\newlist{cenum}{enumerate}{1}
\setlist[cenum]{label=(C-\arabic*)}

\usepackage{cleveref}

\DeclareMathOperator{\Vor}{Vor}
\DeclareMathOperator{\vol}{vol}
\DeclareMathOperator{\SO}{SO}
\DeclareMathOperator{\F}{\mathbb{F}}
\DeclareMathOperator{\R}{\mathbb{R}}
\DeclareMathOperator{\Hom}{Hom}
\DeclareMathOperator{\ad}{\mathbf{ad}}
\DeclareMathOperator{\tr}{tr}

\DeclareMathOperator{\rk}{rk}
\DeclareMathOperator{\sgn}{sgn}
\DeclareMathOperator{\supp}{supp}
\DeclareMathOperator{\Poisson}{Poi}

\newcommand{\plaw}{\mathcal{L}}

\newtheorem{theoremA}{Theorem}

\begin{document}
 
\title{Poisson-Voronoi tessellations and fixed price in higher rank}

\author{Miko\l{}aj Fr\k{a}czyk}
\thanks{}
\address{Department of Mathematics, University of Chicago, 5734 S University Ave, Chicago, Illinois 60637}
\email{mfraczyk@math.uchicago.edu}
\urladdr{https://sites.google.com/view/mikolaj-fraczyk/}

\author{Sam Mellick}
\thanks{}
\address{Department of Mathematics and Statistics, 805 Sherbrooke Street West, Montreal, Quebec H3A 0B9}
\email{samuel.mellick@mcgill.ca}
\urladdr{https://sammellick.github.io/}

\author{Amanda Wilkens}
\thanks{AW was supported in part by NSF grant DMS-1937215.}
\address{Department of Mathematics, University of Texas at Austin, 2515 Speedway, PMA 8.100, Austin, Texas 78712}
\email{amanda.wilkens@math.utexas.edu}
\urladdr{http://web.ma.utexas.edu/users/amandawilkens}

\begin{abstract}
Let $G$ be a higher rank semisimple real Lie group or the product of at least two automorphism groups of regular trees. We prove all probability measure preserving actions of lattices in such groups have cost one, answering Gaboriau's fixed price question for this class of groups. We prove the minimal number generators of a torsion-free lattice in $G$ is sublinear in the co-volume of $\Gamma$, settling a conjecture of Ab\'{e}rt-Gelander-Nikolov. As a consequence, we derive new estimates on the growth of first mod-$p$ homology groups of higher rank locally symmetric spaces. Our method of proof is novel, using low intensity Poisson point processes on higher rank symmetric spaces and the geometry of their associated Voronoi tessellations. We prove as the intensities limit to zero, these tessellations partition the space 
into ``horoball-like'' cells 
so that any two share an unbounded border. We use this new phenomenon to construct low cost graphings for orbit equivalence relations of higher rank lattices.
\end{abstract}



\maketitle

\section{Introduction}\label{sec-intro}

Let $G$ be a semisimple real Lie group or a product of automorphism groups of trees. 
The group $G$ acts on its symmetric space $X$, where $X$ in the latter case is the product of the trees themselves. If $\Gamma$ is a torsion-free lattice in $G$, then $\Gamma\backslash X$ is a finite volume manifold or finite CW complex with fundamental group $\Gamma$. In the setting of compact Riemannian manifolds, any word metric on the quotient of a fundamental group by a finite index normal subgroup coarsely approximates the metric on the corresponding manifold cover.
To get some measure of this algebraic-geometric relationship, Lackenby introduced the rank gradient of a finitely presented group, in \cite{Lackenby}.
The definition of the rank gradient has expanded to finitely generated groups as well as Lie groups which contain some sequence of lattices, with varying requirements on the sequence. We consider the rank gradient of higher rank $G$ with minimal conditions on the lattice sequence. In this context, we prove the rank gradient is $0$. 

Let $\{\Gamma_n\}\subset G$ be a sequence of lattices, not necessarily contained in a common lattice. The sequence of quotients $\Gamma_n\backslash X$ \textit{Benjamini-Schramm converges}, or \textit{BS converges}, to $X$ if the injectivity radius around a typical point in $\Gamma_n\backslash X$ goes to infinity as $n\to\infty$. See \cite{samurai} for a precise definition. If $G$ is a simple higher rank Lie group and $\vol(\Gamma_n\backslash G)\to\infty$, then $\Gamma_n\backslash X$ BS converges to $X$ automatically \cite{samurai}. 
In general, a sequence $\Gamma_n$ of normal finite index subgroups of a lattice $\Gamma$ such that $\bigcap_{n=1}^\infty \Gamma_n=\{1\}$ provides a simple example of a BS convergent sequence of quotients. 

We denote the minimal number of generators of a finitely generated group $\Gamma$ as $d(\Gamma)$ and refer to it as the \textit{rank} of $\Gamma$. All our trees are regular with bounded degree at least $3$, and products are finite with at least $2$ factors. Trees in a product need not share the same degree.

\begin{theoremA}\label{rankgradient}
    Let $G$ be a higher rank semisimple real Lie group or a product of automorphism groups of trees, let $X$ be its symmetric space, and let $\Gamma_n$ be any sequence of torsion-free lattices such that $\Gamma_n\backslash X$ BS converges to $X$. Then 
    \begin{align}\label{rgdef}
        \lim_{n\to\infty}\frac{d(\Gamma_n)-1}{\vol(\Gamma_n\backslash X)}=0.
    \end{align}
\end{theoremA}

The limit in (\ref{rgdef}) is the \textit{rank gradient} of $G$. We say $G$ has \textit{vanishing rank gradient} since the limit is $0$. We conclude the rank of $\Gamma_n$ grows sublinearly in the volume of the manifold or size of the CW complex $\Gamma_n\backslash X$. 

Gelander proved the rank of an irreducible lattice $\Lambda$ in a connected semisimple Lie group $H$ without compact factors is linearly bounded by $\vol(H/\Lambda)$ \cite{Gelander2011}, extending prior results from \cite{Gelander2004}, \cite{BGLS2010}. In particular, Gelander's result implies the limit superior of the relevant expression in (\ref{rgdef}) exists and is finite (even in rank $1$ cases). Theorem \ref{rankgradient} improves this bound in the higher rank case. Lubotzky and Slutsky proved the rank of a congruence subgroup in a nonuniform lattice in a higher rank simple Lie group is logarithmically bounded by its co-volume in the group \cite{LS2022}; they further provide a sharper, asymptotically optimal bound for the subclass of ``2-generic'' groups. 
Ab\'{e}rt, Gelander, and Nikolov proved vanishing rank gradient for sequences of subgroups of a ``right-angled lattice'' in a higher rank simple Lie group, and they conjectured Theorem \ref{rankgradient} holds for simple 
$G$ \cite{AGN}. Theorem \ref{rankgradient} extends their result and positively answers their conjecture, and confirms it for products of trees. We note our result is new even for sequences of congruence subgroups of a single non right-angled uniform arithmetic lattice. Typical uniform higher rank lattices are not right-angled (see for example \cite[Exercise 9.2.\#8]{Morris}).

\begin{theoremA}\label{modphom}
    Let $G, X,$ and $\Gamma_n$ be as in Theorem \ref{rankgradient}, and let $p$ be prime. Then
    $$\lim_{n\to\infty} \frac{\dim_{\F_p}H_1(\Gamma_n, \F_p)}{\vol(\Gamma_n\backslash X)}=0.$$ 
\end{theoremA}

Theorem \ref{modphom} is a direct consequence of Theorem \ref{rankgradient} since $d(\Gamma_n)$ bounds the dimension of the first mod-$p$ homology group of $\Gamma_n$. The first author proved a version of Theorem \ref{modphom} for semisimple algebraic groups with rank at least two over a local field with $p=2$ in \cite{Fraczyk}. Determining the asymptotic growth of mod-$p$ Betti numbers is known to be a difficult problem; analytic methods do not apply as easily here as for 
real Betti numbers. 
For recent progress in the setting of non-uniform lattices in higher rank groups, see \cite{ABFG}, and in the setting of Artin groups, see  \cite{AOS21,OS21,AOS23}. The rate of mod-$p$ homology growth in certain sequences of subgroups relates to completed cohomology groups (see \cite{CE09,CompCohomo}). Theorem \ref{modphom} is new for general uniform higher rank lattices. 

We do not prove Theorem \ref{rankgradient} directly. Rather, Theorem \ref{rankgradient} is a immediate corollary of Theorem \ref{fixedpriceG} and recent results in \cite{SamMiklos}, \cite{Carderi} (proved independently via different methods). The proof of Theorem \ref{fixedpriceG} is geometric and mostly constructive, and comprises most of the paper. It relies on techniques from measured group theory, dynamics, and probability.

\begin{theoremA}\label{fixedpriceG}
    Higher rank semisimple real Lie groups and products of automorphism groups of trees have fixed price $1$.
\end{theoremA}

Theorem \ref{rankgradient} holds whenever $G$ (as in its statement) has fixed price $1$ by \cite[Theorems 1.4 and 7.5]{SamMiklos} and \cite[Theorems 40 and 44]{Carderi}. We proceed with an overview of fixed price and lightly motivate its connection to the rank gradient. In order to define fixed price, we first discuss cost.

The notion of cost, introduced by Levitt \cite{Levitt}, was developed into a full fledged theory by Gaboriau in \cite{Gaboriau} to solve the orbit equivalence problem for free probability measure preserving actions of free groups. The cost of a measure preserving and essentially free action of a countable group on a standard probability space measures the ``average number of maps'' per point needed to generate the orbit equivalence relation; in this way it is a measure-theoretic analogue of the rank of a group.
We give some motivating examples; for a formal definition, see \cite[Definition 1.5.4]{Gaboriau}. Gaboriau proved, in his first of several papers developing the theory \cite{Gab1998}, a free group $\mathbb F_k$ has cost $d(\mathbb F_k)=k$ (meaning it has an action with cost $k$). On the other hand, countably infinite amenable groups have cost $1$ by results of Ornstein and Weiss \cite{OW1980}, and countably infinite property (T) groups have cost $1$ by a recent result of Hutchcroft and Pete \cite{HP}. Hence relations between generators of a group, or the geometry of its classifying space (a manifold or CW complex), may drive down the cost.

A countable group has \textit{fixed price} if each of its measure preserving and essentially free actions share the same cost. Gaboriau \cite{Gaboriau} asked whether all countable groups have fixed price. Higher rank lattices admit certain actions of cost one \cite[Corollary VI.30]{Gaboriau}, so the fixed price conjecture in that case amounts to showing that every action has cost one. Our Theorem \ref{fixedpricelattice} confirms that. Cost is defined more generally for unimodular locally compact second countable  groups (see Section \ref{sec-cost}), and the fixed price question extends. Currently there are no known examples of groups, countable or not, without fixed price. Previous positive results include the following. Free groups have fixed price equal to their rank \cite[Corollary 1]{Gaboriau}, and higher rank nonuniform irreducible lattices in real Lie groups  \cite[VI.28.(a)]{Gaboriau} and countable groups containing an infinite amenable normal subgroup have fixed price $1$ \cite[VI.24.(2)]{Gaboriau}.
Surface groups with genus at least $2$ \cite[VI.9]{Gaboriau} and $\SL_2(\mathbb R)$ have fixed price strictly greater than $1$ \cite{CGMTD}.

The above examples sample a larger family of results (see for example \cite[Section VII]{Gaboriau}). In the countable setting, fixed price $1$ is usually established via some form of the third criterion in \cite[Criteria VI.24]{Gaboriau}, which roughly speaking allows one to induce fixed price $1$ from a ``large enough'' fixed price $1$ subgroup. A new idea using point processes appeared in \cite{SamMiklos} and was used to show products $G\times \mathbb Z$ of a locally compact second countable group $G$ with $\mathbb Z$ have fixed price $1$.  
With this paper, we develop a new method, also based on point processes (presented in Section \ref{introIPVT}). Although the method is new, structural similarities to Gaboriau's third criterion remain. More specifically, Gaboriau's third criterion asserts that if $\Gamma$ is discrete, $\Lambda <\Gamma$ is amenable, and $\Lambda\cap \gamma\Lambda\gamma^{-1}$ is infinite for every $\gamma\in \Gamma$, then $\Gamma$ has fixed price $1$. A higher rank semisimple Lie group $G$ contains a closed amenable subgroup $U$ such that $U\cap gUg^{-1}$ is infinite for almost every $g\in G$ (see Lemma \ref{lem-UIntersections}). In Section \ref{sec-GaboriauComp}, we compare the two criteria.

We deduce Theorem \ref{fixedpricelattice} from Theorem \ref{fixedpriceG}. Theorem \ref{fixedpricelattice} extends Gaboriau's result for nonuniform lattices \cite[VI.28.(a)]{Gaboriau} referenced above. 

\begin{theoremA}\label{fixedpricelattice}
    Lattices in higher rank semisimple real Lie groups and products of automorphism groups of trees have fixed price $1$.
\end{theoremA}

The first result relating rank gradient and fixed price is \cite[Theorem 1]{AN}. Ab\'{e}rt and Nikolov proved the rank gradient of a Farber sequence $\Gamma_n < \Gamma$ is equal to the cost of a certain profinite pmp action of $\Gamma$. Theorem \ref{fixedpricelattice} confirms Conjecture 17 in the same paper.

\begin{corollary}
    Burger-Mozes groups \cite{BM00} have fixed price 1.
\end{corollary}

The corollary immediately follows from Theorem \ref{fixedpricelattice}, as such groups are lattices in products of automorphism groups of trees.

In particular, there exist free amalgams of finitely generated free groups with fixed price $1$.

\subsection{Poisson-Voronoi tessellations}\label{introIPVT}

Poisson point processes and their Voronoi tessellations are essential to our proof of Theorem \ref{fixedpriceG}, particularly the new object named the ideal Poisson-Voronoi tessellation (IPVT), introduced in \cite{BCP} for the hyperbolic plane and further studied in \cite{ACELU23} for all real hyperbolic spaces (for the interested reader, we note \cite{ACELU23} provides many illustrations). Our treatment differs significantly, which we comment on after describing our construction (see Remark \ref{rmk-ACELU23}). We give an example of the IPVT for the special linear group in Example \ref{slnr}. The construction of the IPVT is based on corona actions, which we define and introduce in Section \ref{sec-coronas}. These actions capture the dynamics of $G$ acting on its IPVT for quite general isometric actions $G\curvearrowright X$. We identify corona actions more explicitly for semisimple real Lie groups and products of trees in Sections \ref{sec-LieGpcorona} and \ref{sec-Treecoronas}, respectively.  

For simplicity, in this section we restrict to the action of $G$ on its symmetric space $X$, where $G$ is a semisimple real Lie group (not necessarily higher rank). The space $X$ carries a canonical $G$-invariant Riemannian metric that induces its distance function $d$ and $G$-invariant volume $\vol$. 

We formally define a Poisson point process in Example \ref{ppp}. Here it is enough to think of it as a random scattering of points in a space, for example $X$. Then it is governed by a probability law on the space of all countable sums of Dirac point measures on $X$ (or equivalently, locally finite subsets of $X$). Let $\eta>0$. The number of points (that is, Dirac point measures with value $1$) belonging to a Poisson point process $\Pi_\eta$ restricted to a Borel subset $B\subseteq X$ is a Poisson random variable with mean $\eta\vol(B)$, independent of the Poisson random variable which returns the number of points in any Borel $A\subseteq X$ disjoint from $B$. As a consequence, points of $\Pi_\eta$ in $B$ are independently and uniformly distributed in $B$. The constant $\eta$ is the \textit{intensity} of $\Pi_\eta$.




Let $S\subset X$ be a discrete set of points; we do not require $S$ to be a Poisson point process. Let $s\in S$. The \textit{Voronoi cell} of $s$ is the set of points in $X$ nearer to $s$ than any other $s'\in S$:
$$
    C_s^S:=\{x\in X \mid d(x,s)\leq d(x,s') \text{ for all } s'\in S\}.
$$
We say each $s\in S$ is a \textit{site}. The \textit{Voronoi tessellation} of $S$ is the family of cells $\Vor(S):=\{C_s^S \mid s\in S\}$ and partitions $X$ (in the tree case, a partition may fail, but we address this with a tie-breaking map on cells).
The IPVT is a special Voronoi tessellation on $X$ that arises from a Poisson point process on a ``boundary-like'' space for $X$. Alternatively, one may consider the IPVT as a limit of the Voronoi tessellations of Poisson point processes on $X$ with intensities decreasing to $0$ (this is the perspective in \cite{ACELU23}). We do not prove these perspectives are equivalent, although we suspect they are. The idea behind the construction of the IPVT is that sites of Poisson-Voronoi tessellations of decreasing intensity should escape to some version of the boundary, and the resulting limit point process on this version of the boundary should contain all information needed to reconstruct the limit of tessellations.

In the construction of the IPVT, we use the following type of tessellation, which we define with subsets of functions on $X$ rather than subsets of $X$ itself. 

\begin{definition}[Generalized Voronoi tessellations]\label{def-GeneralizedVoronoi}  Let $F\subset \mathcal C(X)$ be a countable subset. We say $F$ is \emph{admissible} if for any $x\in X$, the set of values $\{f(x)\mid f\in F\}$ is discrete and bounded from below. The Voronoi cell of $f\in F$, for an admissible $F$, is the set 
$$
    C_f^F:=\{x\in X \mid f(x) \leq f'(x) \text{ for all } f'\in F\}
$$ 
and the \emph{generalized Voronoi tessellation} is the family $\Vor(F):=\{ C_f^F\mid f\in F\}.$ 
\end{definition}

For a discrete subset $S\subset X$, the Voronoi tessellation of $S$ in terms of the generalized Voronoi tessellation is $\Vor(S)=\Vor(\{ d(\cdot,s)\mid s\in S\}),$ where $\{ d(\cdot,s)\mid s\in S\}$ takes the place of $F$ in Definition \ref{def-GeneralizedVoronoi}.
In Section \ref{sec-coronas}, we prove we may shift the sequence of $\mathcal C(X)$-valued point processes $\{ d(\cdot,s)\mid s\in \Pi_\eta\}$ up by constants $t_\eta$ in such a way so that $\{ d(\cdot,s)-t_\eta\mid s\in \Pi_\eta\}$ converges to a limiting admissible point process $\Upsilon$ on $\mathcal C(X)$ equipped with the topology of uniform convergence on compact sets. Shifting a process by a constant does not change its Voronoi tessellation, so $\Vor(\{ d(\cdot,s)-t_\eta\mid s\in \Pi_\eta\})=\Vor(\Pi_\eta).$ The ideal Poisson-Voronoi tessellation is then defined as $\Vor(\Upsilon)$, the generalized Voronoi tessellation for the limit process. 

The process $\Upsilon$ is always a Poisson point process on $\mathcal C(X)$ (see Lemma \ref{lem-PoissonContinuity} and Section \ref{sec-coronas}), with a certain $G$-invariant mean measure $\mu$. The action $G\curvearrowright (\mathcal C(X),\mu)$ is the aforementioned corona action, which is responsible for the dynamics of the IPVT (see Definition \ref{def-corona}). We compute the corona action explicitly for semisimple real Lie groups, leading to the following definition, which is a special case of \ref{def-AbstIPVT}.

\begin{definition}[Ideal Poisson-Voronoi tessellation for semisimple real Lie groups]\label{IPVTmodel}
    Let $G$ be a semisimple real Lie group with a minimal parabolic subgroup $P$ and Iwasawa decomposition $G=PK,$ where $K$ is a maximal compact subgroup. Let $U$ be the kernel of the modular character $\chi_P$ of $P$. For each coset $gU\in G/U$ we define the function 
    $$
        \beta_{gU}(s):=-\frac{1}{2\|\rho\|}\log \chi_P(p_0),
    $$ where $g^{-1}s=p_0k_0, p_0\in P,k_0\in K,$ and $\|\rho\|>0$ is a certain constant\footnote{The constant does not affect the resulting tessellation; it is only included so that $\beta_{gU}$ is a Busemann function, rather than a multiple of a Busemann function.} depending only on $G$. Let $\Upsilon=\{\beta_{gU}\mid gU\in \Upsilon'\}$, where $\Upsilon'$ is a Poisson point process on $G/U$ with respect to a $G$-invariant mean measure. The IPVT is defined as the generalized Voronoi tessellation $\Vor(\Upsilon)$.
\end{definition}

\begin{remark}\label{rmk-ACELU23}
The approach to convergence of Poisson-Voronoi tessellations in \cite{ACELU23} does not use generalized Voronoi tessellations. Instead, in the setting of hyperbolic space, the authors enumerate points of Poisson point processes $\Pi_\eta=\{x_i^\eta\mid i\in\mathbb N\}$ for $\eta>0$ in such a way so that for each $i\in\mathbb N$, the Voronoi cell of $x_i$ associated to the process $\Pi_\eta$ converges almost-surely to a cell of the IPVT with respect to the Hausdorff topology on compact subsets \cite[Corollary 3.6]{ACELU23}.
\end{remark}

\begin{example}[Ideal Poisson-Voronoi tessellation for $\SL_n(\mathbb R), n\geq 2$]\label{slnr}
    Let $n\geq 2$ and suppose $G=\SL_n(\mathbb R)$ and $K=\SO(n)$. Then the symmetric space $X=G/K$ can be identified with the set of positive definite matrices, with the action given by $g[M]:=[gMg^{T}]$. We fix a base point $o=\SO(n)\in X$, representing the identity matrix. We let $P$ be the subgroup of upper triangular matrices in $G$; it is a minimal parabolic subgroup of $G$. Let $A$ be the subgroup of diagonal matrices and $\mathfrak a$ the Lie algebra of $A$ consisting of trace zero diagonal matrices. There is an Euclidean metric on $\mathfrak a$, denoted by $\|\cdot\|$, tied to the Riemannian metric on $X$:
    $$\left\|\begin{pmatrix}
        t_1 & & & \\
        & t_2 & & \\
        & & \ddots &\\
        & & & t_n
    \end{pmatrix}\right\|^2:=2n \sum_{i=1}^n t_i^2.$$

     The visual boundary of $X$, denoted $\partial X$, consists of equivalence classes of geodesics rays (see for example \cite[Chapter II.8]{BH}). These geodesics are represented by 
     $$
     \gamma_{k,H}(t):=ke^{tH},\quad k\in K, H\in \mathfrak a, \|H\|=1.
     $$
     If $n=2$, there is just one $G$ orbit on $\partial X$ and $\partial X\simeq G/P$ as a left $G$-action. For $n\geq 3$ there are infinitely many orbits, corresponding to directions in a certain cone of $\mathfrak a$ (the positive Weyl chamber), but one can distinguish the $G$-orbit corresponding to the directions where the volume of $X$ grows ``fastest.'' The directions correspond to the geodesics $\gamma_{k,\hat\rho}$, where $\hat\rho$ is a specific unit length vector in $\mathfrak a$. For example, when $n=2,3,4$, the vector $\hat\rho$ is respectively
     $$ \begin{pmatrix}
         \frac{1}{2\sqrt{2}} & 0\\
         0 & -\frac{1}{2\sqrt{2}}
     \end{pmatrix},\, \begin{pmatrix}
         \frac{1}{2\sqrt{3}} & 0 & 0\\
         0 & 0 & 0\\
         0 & 0 & -\frac{1}{2\sqrt{3}}
     \end{pmatrix},\, \begin{pmatrix}
         \frac{3}{4\sqrt{10}} & 0 & 0 & 0 \\
         0 & \frac{1}{4\sqrt{10}} & 0 & 0 \\
         0 & 0 & -\frac{1}{4\sqrt{10}} & 0\\
         0 & 0 & 0 &-\frac{3}{4\sqrt{10}}\\
     \end{pmatrix}.
     $$
     For readers familiar with the structure theory of semisimple real Lie groups, we remark $\hat\rho$ is the normalization of the vector dual to the half sum of positive roots $\rho$; that is, $\langle \hat\rho, H\rangle= \|\rho\|^{-1}\rho(H), H\in \mathfrak a$, where $\langle\cdot,\cdot\rangle$ is the inner product on $\mathfrak a$ associated to $\|\cdot\|.$ 
    
    The family of geodesics $\gamma_{k,\hat\rho}$ represents a single $G$-orbit in $\partial X$, which is isomorphic to $G/P.$  For each of these geodesics, the limit of distance functions
    $$\beta_{k,\hat\rho}(x):=\lim_{t\rightarrow\infty} d(x,\gamma_{k,\hat\rho}(t))-t$$
    is a Busemann function on $X$. The set $\{\beta_{k,\hat\rho}\mid k\in K\}$ is not $G$-invariant because such functions at $o$ output $0$. However, the set $\{\beta_{k,\hat\rho}+s\mid k\in K, s\in \mathbb R\}$ forms a single $G$-orbit, isomorphic to $G/U$, where $U$ is the kernel of the modular character of $P$. In fact, using the formula from Definition \ref{IPVTmodel}, we have $\beta_k+s=\beta_{ke^{s\hat\rho}U}$. 
    
    The push-forward $\mu$ of any $G$-invariant measure on $G/U$ via the $G$-equivariant map $gU\mapsto \beta_{gU}$ gives rise to a corona action $G\curvearrowright (\mathcal C(X),\mu)$ (see Definition \ref{def-corona}). 
    One can imagine $G/U$ as a line bundle $G/P\times \mathbb R$ over $G/P$ where the real coordinate tracks the ``time delay'' in the parametrization of the geodesic $\gamma_{k,\hat\rho}(t)$.
    If $\Upsilon$ is a Poisson point process on $G/U$ (see Definition \ref{ppp}), then the IPVT is defined as the generalized Voronoi tessellation $\Vor(\{\beta_{gU}\mid gU\in \Upsilon\})$, as in Definition \ref{def-GeneralizedVoronoi}.
    \end{example}
The group $G$ acts on the space of locally finite subsets of the $G$-orbit on $\mathcal C(X)$ identified with $G/U$, preserving the distribution of $\Upsilon$. The dynamics of this action yield significant, and in higher rank, surprising, information on the geometry of the IPVT, which we state in Theorem \ref{cellpairs} and prove in Section \ref{sec-cellpairs}.

\begin{theoremA}\label{cellpairs}
Any two cells in the IPVT for a higher rank semisimple real Lie group or a product of automorphism groups of trees share an unbounded border between them, with probability $1$.
\end{theoremA}

For any $g,g'\in G$, the intersection $S:=gUg^{-1}\cap g'Ug'^{-1}$ stabilizes $gU,g'U\in G/U$ and contains $\mathbb R^{d-1},$ where $d$ is the real rank of $G$ (see Lemma \ref{lem-UIntersections}). In particular, when the real rank of $G$ is at least $2$, the stabilizer $S$ is noncompact, and the orbit of a point $x\in X$ equidistant from $gU$ and $gU'$ (meaning $\beta_{gU}(x)=\beta_{g'U}(x)$) is unbounded. This behavior starkly contrasts the situation in rank $1$, where any two IPVT cells need not share a border at all, and if they do, it must be bounded (the rank $1$ situation quickly follows from \cite[Proposition 9.35]{BH}). This phenomenon is primarily responsible for \Cref{cellpairs}; from here, the statement follows from properties of the Poisson point process, the Howe-Moore property for semisimple Lie groups, and Borel-Cantelli. Theorem \ref{cellpairs} is proved in Section \ref{sec-cellpairs}.

\subsection{Cost via point processes}\label{sec-introcost}

Point processes provide a way to formalize cost theory and fixed price problems for non-discrete, locally compact second countable (lcsc) groups. Intuitively, the cost of a point process is the smallest ``average number of edges per point'' needed to connect up all the points in the process.  The connections must be placed in a $G$-equivariant way, which we explain in Section \ref{sec-cost}. In particular, connecting vertices by an arbitrary Hamiltonian cycle is not allowed.

This viewpoint on cost first appeared in \cite{SamMiklos}, where Ab\'{e}rt and the third author made the connection with an equivalent earlier definition via cross-section equivalence relations and proved any Poisson point process has maximal cost (we make this precise in Section \ref{PoiCost}), generalizing an analogue for countable groups in \cite{AbertWeiss}. The cross-section definition of cost also appears in \cite{Carderi} and \cite{CGMTD}. Since the cost of a Poisson point process is maximal, to prove fixed price one may show that a Poisson point process has cost $1$.
We show that a Poisson point process weakly factors (see Definition \ref{def-WeakFactoring}) onto a new action, which roughly speaking is a Poisson point process on $G$ coupled with an independent copy of the ideal Poisson-Voronoi tessellation. We prove this weak factor has the same cost as the original Poisson point process on $G$ (see Theorem \ref{WeakLimitTheorem}). 


In this way, we reduce the proof of Theorem \ref{fixedpriceG} to showing that, using the IPVT as a guideline, one may connect up the points of a Poisson point process ``very cheaply'' using just a bit more than one directed edge per point on average. This construction is carried out in Section \ref{sec-CostOneProof}, where we prove a more general and technical version of Theorem \ref{fixedpriceG}. This part of the proof bears some resemblance to how fixed price is typically proved for discrete groups, using \cite[VI.24,(3)]{Gaboriau}, although the setting of lcsc groups brings out new and significant challenges. 

\subsection{Outline}

Theorems \ref{rankgradient} and \ref{modphom} follow from Theorem \ref{fixedpriceG} and prior results in \cite{SamMiklos} and \cite{Carderi} as described earlier in Section \ref{sec-intro}. Sections \ref{introIPVT} and \ref{sec-introcost} contain motivation for the proof of Theorem \ref{fixedpriceG}.

We list frequent notation in Section \ref{sec-freqnot}. In Section \ref{PoiCost}, we define point processes, including Poisson point processes, and their cost as actions. We introduce examples of point processes important in the proof of Theorem \ref{fixedpriceG} and motivate the method of point process cost as a route to proving fixed price. We prove Theorem \ref{fixedpricelattice} assuming Theorem \ref{fixedpriceG} in Section \ref{sec-cost}.

In Section \ref{sec-coronas}, we define corona actions (which we sometimes call coronas) for locally compact second countable groups acting properly, transitively, and isometrically on on locally compact metric spaces (Definition \ref{def-corona}). We prove their existence under some conditions; in particular when the group is non-amenable (Corollary \ref{cor-coronasExistence}). We describe them explicitly for semisimple real Lie groups in Theorem \ref{thm-LieGpcorona} and prove some properties in this and the tree setting in Theorem \ref{thm-Treecoronas}.

Sections \ref{sec-LieGpcorona} and \ref{sec-Treecoronas} contain computations necessary for, and the proofs of, Theorems \ref{thm-LieGpcorona} and \ref{thm-Treecoronas}, respectively. Section \ref{sec-cellpairs} contains the proof of Theorem \ref{cellpairs}, which is relatively simple to follow given only Section \ref{sec-coronas}.

We prove Theorem \ref{fixedpriceG} in Section \ref{sec-CostOneProof}. We begin with an abstract criteria for fixed price $1$ stated in Theorem \ref{thm-AbstractCostOne} and a summary of steps in the proof, which relies on the Palm equivalence relation of a certain point process (defined in Section \ref{PoiCost}) having cost $1$. We construct a subrelation (effectively we restrict our attention to inside each cell of the ideal Poisson-Voronoi tessellation) and prove it has cost $1$, via hyperfiniteness, in Section \ref{sec-Subrelation}. In Section \ref{sec-Graphing}, we apply Theorem \ref{cellpairs} for a quick proof of Theorem \ref{fixedpriceG} in the real Lie group case and proceed to lay groundwork for the proof of the more general Theorem \ref{thm-AbstractCostOne}; we give the proof of Theorem \ref{thm-AbstractCostOne} in Section \ref{sec-AbstCostOneProof}. In Section \ref{sec-GaboriauComp}, we describe the similarities between Theorem \ref{thm-AbstractCostOne} and Gaboriau's criteria for fixed price $1$ for countable groups.
We ask several further questions in Section \ref{Sec-Questions}.

A reader familiar with measured group theory might want to start with Section \ref{sec-CostOneProof} and refer back as needed.

\subsection{Frequent notation}\label{sec-freqnot}
Let $B$ be a set; a multiset $A\subseteq B$ is a subset of $B$ where each element $a\in A$ appears with multiplicity $m(a)\in \mathbb N\cup\{\infty\}.$ 
For example, if $A=\{1,1,2\}$, then $\sum_{a\in A}a=1+1+2=4.$ A function $f$ on a multiset $A$ has $m(a)$ values at $a$ and $f(A)=\{f(a)\mid a\in A\}$ is in general also a multiset. 

For any set $S$ and $x\in S$ we write $\mathds 1_S$ for the indicator function 
$$
    \mathds 1_S(x):=\begin{cases}
			 1 & \text{ if } x\in S\\ 
			 0 & \text{ otherwise.}
    \end{cases}
$$

We write $f \ll g, f\gg g$ if there exists a constant $C>0$ such that $f\leq Cg$ and $f\geq Cg$ respectively. We use $f\asymp g$ if $f\ll g$ and $f\gg g$ and $f\sim g$ if $\lim_{t\to\infty} f(t)/g(t) =c>0$. When $G$ is a group we write $1$ for the identity element. If $H\subset G$ is a subgroup and $g\in G$, then $H^g:=gHg^{-1}$. We usually use $X$ for a metric space with an isometric action of a locally compact second countable group $G$ and $\vol$ for the Riemannian volume on $X$ if it is a symmetric space or the counting measure if $X$ is discrete. We reserve $m_G$ for a fixed Haar measure on $G$ and $dg$ for integration with respect to this measure.

The abbreviation lcsc stands for locally compact second countable. All measured spaces in this paper are standard Borel and the default sigma algebra on a topological space is the Borel sigma algebra. Let $(Y,\nu)$ be a measured space with a measure $\nu$. A measurable action of an lcsc group $G\times X\to X,$ denoted $G\curvearrowright (Y,\mu)$, is measure preserving if $\mu(A)=\mu(gA)$ for any measurable subset $A\subset Y$ and $g\in G$. All actions in this paper are left actions unless explicitly mentioned otherwise. An action $G\curvearrowright (Y,\mu)$ is a probability measure preserving (pmp) action if it is measure preserving and $\mu$ is a probability measure. If $\mu$ is a probability measure and the action of $G$ preserves the measure class, that is, $\mu(A)>0$ if and only if $\mu(gA)>0$ for all $g\in G$, then we call the action quasi-pmp. 

If $Y$ is an lcsc metric space, we write $\mathcal M(Y)$ for the space of locally finite Borel measures equipped with the topology of weak-* convergence on compact sets. 

When $\phi\colon Y_1\to Y_2$ is a measurable map between measure spaces and $\mu$ is a measure on $Y_1$ we write $\phi_*\mu(A):=\mu(\phi^{-1}(A))$ for the push-forward measure on $\phi(Y_1)$. We write $\delta_y$ for the Dirac measure at $y$.
We say that a sequence of measures on a locally compact space is tight if they are uniformly bounded on compact subsets.

A countable measured equivalence relation $\mathcal O$ on a measured space $Y$ is a measurable subset $\mathcal O\subset Y\times Y$, which is an equivalence relation on $Y$ such that the classes $[y]_{\mathcal O}:=\{y'\in Y\mid (y,y')\in\mathcal O\}$ are countable for every $y\in Y$ or almost every $y\in Y$ if we have already fixed a measure on $Y$. If $\mathcal S$ is another countable measured equivalence relation on $Y$ and $\mathcal S\subset \mathcal O$, we say that $\mathcal S$ is a sub-relation of $\mathcal O$. A countable equivalence relation $\mathcal O$ on a measured space $(Y,\mu)$ is (quasi) measure preserving if, for every measurable bijection $\phi\colon Y\to Y$ such that $(y,\phi(y))\in \mathcal O,$ we have $\phi_*\mu=\mu$ (in the quasi mp case, we have that $\phi_*\mu$ and $\mu$ are in the same class). For convenience we write pmp equivalence relation or quasi-pmp equivalence relation for pmp and quasi-pmp countable measured equivalence relations, respectively.

\subsection{Acknowledgements} MF thanks Miklós Abért, Tsachik Gelander and Shmuel Weinberger for discussions and valuable comments. AW thanks Lewis Bowen for helpful conversations. 

\section{Point processes and cost}\label{PoiCost}

We define point processes and all necessary machinery needed to work with them in our proof of Theorem \ref{fixedpriceG}. In particular, in Section \ref{sec-cost}, we define weak convergence and weak factors of point processes. Part of the strategy of the proof of Theorem \ref{fixedpriceG} is identifying a point process action with cost $1$ and then proving every essentially free pmp action weakly factors onto such an action. Since cost is non-decreasing under weak factors \cite{SamMiklos}, this strategy is sufficient to prove fixed price $1$.

\subsection{Point processes}

Let $Y$ be a locally compact second countable (lcsc) space. Then $Y$ is Polish and we may fix a metric for $Y$ so that it is a complete and separable metric space (csms). Generally we assume $Y$ is non-discrete, although occasionally we consider the discrete case. Since $Y$ is a csms, it is a natural setting for probability theory. Let $\mathcal B(Y)$ be the Borel $\sigma$-algebra on $Y$. We say a measure $\nu$ on $Y$ is \textit{locally finite} if $\nu(B)<\infty$ for all relatively compact $B\in\mathcal B(Y)$. 

Informally, a point process on $Y$ is a random and discrete collection of points, possibly with multiplicities. Technically it is a random measure on an abstract probability space that takes values in the \textit{configuration space} $\mathbb M(Y)$, the set of locally finite sums of Dirac point measures on $Y$. The space $\MM(Y)$ can also be understood as the space of locally finite multisubsets of $Y$. Most of the point processes considered in this paper are simple, meaning that the mass of any point is either $0$ or $1$. In this case, the corresponding multiset is just a set. Since $Y$ is lcsc, any element of $\MM(Y)$ can be expressed as
$$
    \Pi=\sum_{i=1}^N \delta_{y_i},
$$
where $N\in\mathbb N\cup\{\infty\}$, $y_i\in Y$, and $\delta_{y_i}$ is the Dirac measure at $y_i$ \cite[Corollary 6.5]{LP}. 

The Dirac measure perspective helps to define a natural topology on $\MM(Y)$, but as mentioned we may think of elements $\MM(Y)$ simply as locally finite multisubsets. In particular, we write $y_i\in\Pi$ to mean $\delta_{y_i}$ is a term of the countable sum $\Pi$.

Later we consider and define weak convergence of point processes. For this purpose we endow $\mathbb M(Y)$ with the weak-* topology $\mathcal T$. The space $(\MM(Y),\mathcal T)$ is itself lcsc (see \cite{DVJv1}, Section A2.6). 
The topology $\mathcal T$ generates the Borel $\sigma$-algebra $\mathcal B(\MM(Y))$, the smallest $\sigma$-algebra such that the projection maps $\pi_B:\MM(Y)\to \mathbb N\cup\{\infty\}$ sending $\Pi\in\MM(Y)$ to $\Pi(B)$ are measurable for all $B\subseteq \mathcal B(Y)$ \cite[A2.6.III]{DVJv1}. Note $\Pi(B)$ is a random variable that returns the number of points of $\Pi$ residing in $B$.

Formally, a \textit{point process} $\Pi$ on $Y$ is a random element of $\MM(Y)$, meaning it is a measurable map from some abstract probability space to $(\mathbb M(Y),\mathcal B(\mathbb M(Y)))$. The \textit{law} of $\Pi$, denoted $\plaw(\Pi)$, is the push-forward of the probability measure to $(\mathbb M(Y),\mathcal B(\mathbb M(Y)))$.
The \textit{mean measure} of $\Pi$ is $\nu\colon\mathcal B(Y)\rightarrow \mathbb N\cup\{\infty\}$ where $\nu(B):=\mathbb E(\Pi(B))$, sometimes called the \textit{Campbell} or \textit{intensity measure} of $\Pi$. Often $Y$ will be a lcsc group with a fixed Haar measure, and the mean measure will be a fixed multiple of this Haar measure. We refer to this multiple as the \textit{intensity} of $\Pi$ and denote it as $\intensity(\Pi)$.

The main point process of interest in this paper is the Poisson point process.

\begin{example}[Poisson point process]\label{ppp}
    Let $\mu$ be a $\sigma$-finite measure on a lcsc space $Y$. A point process $\Pi$ is a \textit{Poisson point process on $Y$ with mean measure $\mu$} if it satisfies the conditions:
    \begin{enumerate}
        \item For every $B\in\mathcal B(Y)$ with $\mu(B)<\infty$, the number of points of $\Pi$ in $B$, denoted as $\Pi(B)$, is a Poisson random variable with mean $\mu(B)$. 
        \item For disjoint $B_1,B_{2}\in \mathcal B(Y)$, the random variables $\Pi(B_1),\Pi(B_2)$ are independent. 
    \end{enumerate}
    Condition (2) implies the analogous statement for finite collections of disjoint sets \cite[Theorem 6.12]{LP}, and conditions (1) and (2) imply a concrete characterization of $\Pi$:  The restriction $\Pi\cap B$ has the law of a finite sum of Dirac point measures. A Poisson random variable determines the number of point measures, and independent random variables uniformly distributed on $B$ determine their locations \cite[Theorem 1.2.1]{Reiss}. If the measure $\mu$ has no atoms, then the process $\Pi$ is simple and is a random locally finite subset of $Y.$
\end{example}

If $Y$ is countable and $\mu$ is its counting measure, a Poisson point process on $Y$ is a multiset and the underlying set of elements is a Bernoulli random subset of $Y$. The atomic case comes up only once in the paper, indirectly, when we prove that a Poisson point process on $(D,\mu)$ is a weak limit of Poisson point processes on $D$ with mean measures $\mu_t$ (see Section \ref{sec-AbstCostOneProof}). We sometimes write $\Poisson(Y,\mu)$ for the distribution of the Poisson point process on $Y$ with mean measure $\mu$.

Suppose $G$ is a lcsc group acting continuously on $Y$. Then there is a natural $G$-action on $\mathbb M(Y)$ given by $g\Pi:=\{gy\mid y\in \Pi\}$ for all $g\in G$, $\Pi\in\mathbb M(Y)$. 
A $G$-invariant point process $\Pi$ is \textit{essentially free} if its stabilizer is trivial $\mathcal L(\Pi)$-almost surely. 

\begin{example}[Lattice action]\label{latticeaction}
    Let $\Gamma < G$ be a lattice. Then $G/\Gamma$ has a $G$-invariant Borel probability measure. Cosets $a\Gamma \in G/\Gamma$ are configurations by definition, and the inclusion $G/\Gamma \subset \MM(G)$ is equivariant. In this way lattice shift actions are invariant point processes. 
    More concretely, one can sample a point $a$ uniformly on a fundamental domain $\mathscr{F} \subset G$, and then take $a\Pi$ as the point process. In this way we also see that the intensity of $G \acts G/\Gamma$ is $\covol(\Gamma)^{-1}$.
\end{example}

We also consider point processes such that labels mark each point, where labels live in some lcsc space $Z$. A \textit{$Z$-marked point process on $Y$} is an element of
$$
    \mathbb M(Y,Z):=
		\left\{ \Pi \in \mathbb M(Y \times Z) \middle| \begin{aligned}
			 & \text{ if } (y,z),(y,z') \in \Pi \text{ then } z = z', \text{ and}\\ 
			 & \text{ the projection of } \Pi \text{ to } Y \text{ is locally finite}\\
		\end{aligned} \right\}.
$$
We refer to the unmarked point process underlying the marked one as the \emph{base process}. The first condition ensures each point in the base process receives exactly one mark. If $Z$ is compact, the second condition is satisfied automatically, but we often consider noncompact marking spaces. We write $\ell(y)$ for the mark of a point $y$.
A marked point process $\Pi$ is $G$-invariant if $\mathcal L(\Pi)$ is a $G$-invariant probability measure on $\MM(Y,Z).$ Note $G$ does not act on labels (even if $Z$ has a $G$-action). 

The marked point processes in Example \ref{ex-iidlabeling} and Definition \ref{def-ProductMarking} appear many times throughout the paper.
We remark the product marking in Definition \ref{def-ProductMarking} encodes the independent coupling of a point process on $G$ with a $G$-invariant random measure as a marked point process on $G$.

\begin{example}[IID point process]\label{ex-iidlabeling}
For a point process $\Pi$ on $Y$, its \textit{IID marking} or \textit{Bernoulli extension} is the $[0,1]$-marked point process with base process $\Pi$ and each base point equipped with an IID uniform random variable. The labels introduce additional randomness not determined by $\Pi$. When $\Pi$ is $G$-invariant, the resulting marked point process is $G$-invariant. 
\end{example}

\begin{definition}[Product marking]\label{def-ProductMarking}
Let $G$ be a lcsc group with a pmp action on $(Y,\nu)$, and let $\omega\in \MM(G,Z)$. Denote the $Z$-label on $g\in\omega$ as $\ell(g)$. Let $y\in Y.$ Define
$$
    \omega\times y:=\{(g,\ell(g),g^{-1}y)\in G\times Z\times Y \mid g\in \omega\}\in \MM(G,Z\times Y).
$$
An important property of this product operation is that $g(\omega\times y)=(g\omega\times gy).$ Let $\Theta$ be a $Y$-valued random variable with probability law $\nu$ and $\Pi$ a marked point process. The \emph{product marking} of $\Pi$ by $\Theta$ is the marked point process $\Pi\times \Theta$. Note the mark at any single point of $\Pi\times \Theta$ determines the $Y$-part of the marking on the whole set.
\end{definition}

The name product marking is justified by the following property. We have a $G$-equivariant measurable map 
$$
    \MM(G,Z)\times Y\ni (\omega,z)\to \omega\times z\in \MM(G,Z\times Y)
$$
that induces an isomorphism of $G$-actions $(\MM(G,Z)\times Y,\mathcal L(\Pi)\times \nu))\to (\MM(G,Z\times Y), \mathcal L(\Pi\times \Theta)).$  When it is clear from context that we are talking about measures on $\MM(G,Z\times Y)$, we shall abuse notation slightly and write $\mathcal L(\Pi)\times \mathcal L(\Theta)$ for the probability law of $\Pi\times \Theta$.

\begin{example}[Induced action]\label{inducedactiondefn}
Let $\Gamma < G$ be a lattice and $\Gamma \acts (Z, \mu)$ a pmp action. One can \emph{induce} it to an action of $G$ that is naturally a $Z$-marked point process. First fix a section $\sigma : G/\Gamma \to G$, that is, a measurable and bijective map such that
\[
	\sigma(a\Gamma)\Gamma = a\Gamma \text{ and } \sigma(\Gamma) = 1.
\]
Note that
\[
	\sigma(ga\Gamma)\Gamma = ga\Gamma \text{ and } g\sigma(a\Gamma)\Gamma = ga\Gamma.
\]
Thus we may define $c : G \times G/\Gamma \to \Gamma$
\[
	c(g,a\Gamma) := \sigma(ga\Gamma)^{-1} g \sigma(a\Gamma).
\]
One can check that $c$ satisfies the cocycle identity, that for all $g, h \in G$ and $a\Gamma \in G/\Gamma$, we have
\[
	c(gh, a\Gamma) = c(g, ha\Gamma) c(h, a\Gamma).
\]
Lastly, we define the action $G \acts G/\Gamma \times Z$ as
\[
	g(a\Gamma, z) := (ga\Gamma, c(g, a\Gamma)z).
\]
One can check that this is pmp for the product measure.

Now define the point process $\Pi : G/\Gamma \times Z \to \MM(G, Z)$ to be
\[
	\Pi(a\Gamma, z) := \{(g, c(g^{-1}, a\Gamma)z) \in G \times Z \mid g\Gamma = a\Gamma \}.
\]
Note that $\Pi$ is an equivariant map, and thus this is an invariant point process. As an unlabelled point process it is just the lattice shift, and the label at any one point determines the label at any other point. 
\end{example}

\subsection{Factor maps and graphs}\label{sec-facgraph}

Let $G$ be an lcsc group with pmp actions on $(Y,\nu)$ and $(Y',\nu')$. A map $\Phi:Y\rightarrow Y'$ is a \textit{factor map} if, on a set of full $\nu$-measure, for all $g\in G$ we have $g\circ \Phi=\Phi\circ g$, and $\nu'=\Phi_*\nu$. 
The factor graphs in this paper are only used in the context of (possibly marked) invariant point processes on $G$; they are a particular type of factor map.
A factor graph of a point process $\Pi$ on $G$ is a measurably and equivariantly defined directed graph $\mathscr{G}(\Pi)$ whose vertex set is $\Pi$. We state the definition for marked point processes; the unmarked version is recovered by choosing a trivial label space $Z=\{0\}.$

Let $\Pi\in \mathbb M(G,Z)$. More formally, a \textit{factor graph} is a measurable map $\mathscr{G}\colon\MM(G,Z) \to \MM(G \times G)$ such that:
\begin{enumerate}
    \item $\mathscr{G}(g\Pi) = g \mathscr{G}(\Pi)$ for all $g \in G$.
    \item $\mathscr{G}(\Pi) \subset \Pi \times \Pi$.
\end{enumerate}
Here the action of $G$ on $\MM(G \times G)$ is induced by the diagonal (left) action of $G$ on $G \times G$.


A key property of factor graphs is that they are deterministically defined: $\mathscr{G}(\Pi)$ is a random graph, but the randomness only comes from $\Pi$ and the labeling $\ell\colon \Pi\to Z$. Factor graphs are essentially \textit{graphings} in the sense of \cite[Definition 1.1]{Gaboriau}. That is, graphings of the associated Palm equivalence relation can be identified with factor graphs of the point process and vice versa. See Section 3 of \cite{SamMiklos} for further details.

\subsection{Palm measures and equivalence relations}

We make use of a characterization of a Poisson point process known as the \emph{Mecke equation}. We write $\Pi + \delta_y$ for $\Pi \cup \{y\}$, since point processes on lcsc spaces are countable sums of Dirac measures on their underlying space.  

\begin{theorem}[{\cite[Thm 4.1]{LP}}]\label{Mecke}
    Let $\mu$ be a Borel measure on a lcsc space $Y$. Then a point process $\Pi$ on $Y$ is Poisson with mean measure $\mu$ if and only if for all nonnegative measurable functions $f\colon \MM(Y) \times Y \to \mathbb{R}_{\geq 0}$ we have
    \[
     \int_{\MM(Y)}\sum_{y \in \Pi} f(\Pi,y)d\plaw(\Pi) = \int_Y \int_{\MM(Y)}f(\Pi + \delta_y,y) d\mathcal L(\Pi) d\mu(y).
    \]
\end{theorem}

There is also a multivariate Mecke equation, although we'll only require the bivariate case. 

\begin{theorem}[{\cite[Thm 4.4]{LP}}]\label{2Mecke}
Let $\mu$ be a Borel measure on a lcsc space $Y$. If $\Pi$ is a Poisson point process with mean measure $\mu$, then for all nonnegative measurable functions $f\colon \MM(Y)\times Y \times Y \to \mathbb{R}_{\geq 0}$ we have
\begin{align*}
    \int_{\MM(Y)}\sum_{x, y \in \Pi, x\neq y} f(\Pi,x,y) d\mathcal L(\Pi) = \int_Y \int_Y \int_{\MM(Y)}f(\Pi + \delta_x + \delta_y,x,y)d\mathcal L(\Pi) d\mu(x)d\mu(y).
\end{align*}
If $\mu$ is atomic and $\Pi$ has a point $x$ with multiplicity $m$ we count the each of the ``multiples'' of $x$ as distinct on the left hand side.
\end{theorem}

The proof is technical and relies on unique properties of the Poisson distribution, but the idea behind its validity is simply that a Poisson point process restricted to a single point should be independent of the rest of itself.

Given an invariant marked point process $\Pi\in \MM(G,Z)$ on a lcsc group $G$, where its base process has intensity $\eta>0$, we would like to condition on it to contain a point at the identity $1\in G$. This is a probability zero event on non-discrete spaces but one can circumvent the problem using measure disintegration. We consider the space $\MM(G,Z)\times G$ equipped with the $G$-invariant (infinite) measure 
$$\mathcal L(\Pi)^\circ:=\int_{\MM(G,Z)}\left(\sum_{g\in \Pi}\delta_{(\Pi,g)}\right)d\mathcal L(\Pi).$$
Let $p$ be the projection map onto the $G$ factor. By definition, the push-forward $p_* \mathcal L(\Pi)^\circ$ is the mean measure $\eta dg$. By \cite[10.4.8]{Bogachev} we can disintegrate the measure $\mathcal L(\Pi)^\circ$ along the fibers of $p$
$$\mathcal L(\Pi)^\circ=\eta\int_G (\mathcal L(\Pi)_g\times \delta_g)dg,$$ where $\mathcal L(\Pi)_g\times \delta_g$ is a probability measure on the fiber $p^{-1}(g)$, which consists of pairs $(\omega, g)$ where $\omega$ is a marked point configuration containing $g$. By invariance of $\mathcal L(\Pi)^\circ$ and the uniqueness  of disintegration \cite[10.4.3]{Bogachev} we get that $\mathcal L(\Pi)_g\times \delta_g =g_*(\mathcal L(\Pi)_1\times \delta_1)$ for almost every $g\in G$. The measure $\mathcal L(\Pi_\circ):=\mathcal L(\Pi)_1$ is called the \emph{Palm measure} of $\Pi$, and we call the point process $\Pi_\circ$ with this distribution the \emph{Palm version} of $\Pi.$ We denote the subspace of marked point processes containing the identity as $\MMo(G,Z).$ The following theorem, called the \emph{Campbell-Little-Mathes-Mecke Theorem} \cite{baccelli2020random}, or \emph{the refined Campbell theorem} in \cite{LP}, follows immediately from the disintegration formula.

\begin{theorem}[CLMM theorem]\label{CLMM}
    Let $\Pi$ be an invariant marked point process on $G$ with Palm version $\Pi_\circ$. If $f : \MMo(G,Z)\times G \to \RR_{\geq 0}$ is a measurable function,
    then
    \begin{equation}\label{eq-CLLM}
        \int_{\MM(G,Z)}\sum_{g \in \Pi} f(g^{-1}\Pi,g) d\mathcal L(\Pi)= \eta\int_{\MM_\circ(G,Z)}\int_G f(\Pi_\circ,g) dgd\mathcal L(\Pi_\circ).
    \end{equation}
\end{theorem}

Below we sketch a short proof (see also \cite[Theorem 3.7]{SamMiklos}).
\begin{proof} Consider the function $F'(\Pi,g):=f(g^{-1}\Pi,g)$ defined on the set of pairs $(\Pi,g)$ such that $g\in \Pi$. We extend it by zero to the whole space $\MM(G,Z)\times G$ and call the extension $F$. Then
    \begin{align*}
    \int_{\MM(G,Z)}\sum_{g \in \Pi} f(g^{-1}\Pi,g) d\mathcal L(\Pi)&=\int_{\MM(G,Z)\times G}F(\Pi,g)d\mathcal L(\Pi)^\circ
    =\eta\int_G \int_{\MM(G,Z)}F(\Pi,g)d\mathcal L(\Pi)_g dg\\
    &=\eta \int_G \int_{\MM(G,Z)}F(g\Pi_\circ,g)d\mathcal L(\Pi_\circ) dg=\eta\int_G \int_{\MM(G,Z)}f(\Pi_\circ,g)d\mathcal L(\Pi_\circ) dg.\qedhere
    \end{align*}
\end{proof}
An alternative way to define the Palm measure is described in \cite[Def 3.5]{SamMiklos}: for measurable subsets $A\subset \MMo(G,Z)$ and $U\subset G$ we have 
\begin{equation}\label{eq-PalmCondition}\mathbb E(|\Pi\cap U|)\mathcal L(\Pi_\circ)(A)=\mathbb E(\{g\in \Pi\cap U\mid g^{-1}\Pi\in A\}).\end{equation} 
One may verify that the two definitions match by applying Theorem \ref{CLMM} with $f(\Pi,g):=\mathds 1_U(g)\mathds 1_{A}(\Pi)$.

\begin{lemma}[Palm invariance]\label{lem-PlamInvariant}
Let $A\in \mathcal B(\MM(G,Z))$ be a $G$-invariant set. Then $\mathcal L(\Pi)(A)=\mathcal L(\Pi_\circ)(A).$
\end{lemma}
\begin{proof} Let $U$ be a finite, positive measure subset of $G$. Let $\eta$ be the intensity of $\Pi$. By (\ref{eq-PalmCondition}),
\[
\eta m_G(U)\mathcal L(\Pi_\circ)(A)=\mathbb E(\{g\in \Pi\cap U\mid g^{-1}\Pi\in A\})=\mathbb E(\{g\in \Pi\cap U\mid \Pi\in A\})=\eta m_G(U)\mathcal L(\Pi)(A).\qedhere
\]
\end{proof}
In the case when $\Pi$ is the Poisson point process or IID Poisson point process on $G$ the associated Palm version is simply $\Pi +\delta_1$ for the unmarked process and $\Pi+\delta_1$ with a random label at $1$ for the IID marked one. This characterization of the Poisson point process is known as the Mecke-Slivnyak theorem. 

\begin{example}[Palm for a lattice shift]
    The Palm measure of a lattice shift $G \acts G/\Gamma$ is simply $\delta_\Gamma$. 
\end{example}

\begin{example}[Palm for an induced action]\label{inducedactionpalm}
Let $\Pi$ denote the $Z$-marked point process associated to an induced action $G \acts G/\Gamma \times Z$, as in Example \ref{inducedactiondefn}. We claim $\Pi(\Gamma, z)$ is a Palm version of $\Pi$. That is to say, the Palm measure of $\Pi$ is exactly the pushforward of $\mu$ under the map
\[
z \mapsto \{(\gamma, \gamma^{-1}z) \in G \times Z \mid \gamma \in \Gamma \}.
\]
We prove this using the description of Palm measures in terms of relative rates of intensities of thinnings; see Section 3.2 of \cite{SamMiklos} for more details. Note that the Palm measure of $\Pi$ is supported on the set of $Z$-marked configurations whose underlying set is $\Gamma$ and whose labels transform correctly:
\[
S = \{ \omega \in \MMo(G, Z) \mid \text{underlying set of } \omega \text{ is } \Gamma, \ell(\gamma) = \gamma^{-1}\ell(1) \}.
\]
Fix a fundamental domain $\mathscr{F} \subset G$ for $\Gamma$. Observe that there is a unique point $\tau(a\Gamma)$ in $\mathscr{F} \cap a\Gamma$. It is clear that $\Pi(\Gamma, z)$ has the same distribution as $\tau(a\Gamma)^{-1}\Pi(a\Gamma, z)$. We show that $\tau(a\Gamma)^{-1}\Pi(a\Gamma, z)$ is a Palm version of $\Pi$. 

Let $A \subseteq S$, and write $\theta_A$ for the associated thinning. Note that the events
\[
\{\tau^{-1} \Pi \in A\} \text{ and } \{\theta_A(\Pi) \cap \mathscr{F} \text{ is nonempty} \}
\]
are equal. Thus
$$
    \PP[\tau^{-1}\Pi \in A] = \PP[\theta_A(\Pi) \cap \mathscr{F} \text{ is nonempty}]
    = \EE\abs{\theta_A(\Pi) \cap \mathscr{F}}
    =  \frac{\EE\abs{\theta_A(\Pi) \cap \mathscr{F}}}{\EE\abs{\Pi \cap \mathscr{F}}}
    = \frac{\intensity(\theta_A(\Pi))}{\intensity(\Pi)}
    = \PP[\Pi_\circ \in A].\qedhere
$$
\end{example}

The \emph{Palm equivalence relation} is the countable equivalence relation on the space $(\MMo(G,Z),\mathcal L(\Pi_\circ))$ defined as follows: two marked configurations $(\omega,\ell),(\omega',\ell')\in \MMo(G,Z)$ are equivalent if there exists $g\in \omega$ such that $\omega'=g^{-1}\omega$ as marked configurations; so $\ell(h)=\ell'(gh)$ for every $h\in \omega.$ It can also be seen as the restriction of the $G$-orbit equivalence relation from $\MM(G,Z)$ to $\MMo(G,Z).$

The Palm equivalence class of $\Pi_\circ$ is naturally identified with the set of points in $\Pi_\circ$ via the map
\begin{equation}\label{eq-PalmClass}\Pi_\circ\ni g\mapsto g^{-1}\Pi_\circ\in [\Pi_\circ]_\mathcal R=\{g^{-1}\Pi_\circ\mid g\in \Pi_\circ\}.\end{equation}
The Palm equivalence relation is probability measure preserving (see for example \cite[Thm 3.13]{SamMiklos}) and ergodic if the original point process is.

There is a one-to-one correspondence between the connected factor graphs of $\Pi$ and the generating graphings of the Palm equivalence relation, as explained in Section 3 of \cite{SamMiklos}. Given a factor graph $\mathscr G\colon \MM(G,Z)\rightarrow \mathbb M(G\times G)$, its restriction to $\MMo(G,Z)$ determines $\mathscr G$ uniquely and defines a measurable graph structure on $\Pi_\circ$ given an input $\Pi$. As the equivalence class of $\Pi_\circ$ is identified with the underlying point configuration, we get a measurable connected graph structure on the equivalence class which is nothing else than a generating graphing of the Palm equivalence relation. Going in reverse, starting with any graphing $\Pi_\circ\in \MMo(G,Z)$ we can induce a unique $G$-invariant factor graph on $\Pi$ by it in a $G$-equivariant way from $\MMo(G,Z)$. The factor graph is connected if and only if the corresponding graphing generates the Palm equivalence relation. 

We conclude the discussion of the Palm equivalence relation with a closer look at the case of product markings introduced in Definition \ref{def-ProductMarking}. 

\begin{lemma}[Palm equivalence relation for the product marking]\label{lem-PalmProductMarking}
If $\Pi\times \Theta$ is the product marking as in Definition \ref{def-ProductMarking}, then $\mathcal L((\Pi\times \Theta)_\circ)=\mathcal L(\Pi_\circ)\times \mathcal L(\Theta)$ and the Palm equivalence relation $\mathcal R$ is given by 
$$ \omega\times z\sim_{\mathcal R}\omega'\times z' \text{ iff there exists }g\in \omega \text{ such that }g^{-1}\omega=\omega' \text{ and } g^{-1}z=z'.$$
\end{lemma}

\begin{proof} Let $g\in G$ and $\Theta$ a $Z$-valued random variable with distribution $\nu$. The proof is just a formal verification that $\mathcal L(\Pi)_g\times \mathcal L(\Theta)\times \delta_g$ is the disintegration of $\mathcal L(\Pi\times \Theta)^\circ.$ For any $z\in Z$ the map $\MM(G)\times G \ni (\omega,g) \to (\omega\times z,g) \in \MM(G,Z)\times G$ induces the push-forward map $(\times z)_*$ between measures on these spaces. As a push-forward, this map is of course linear. We have
\begin{align*}
    \mathcal L(\Pi\times \Theta)^\circ&=\int_{\MM(G)}\int_Z\sum_{g\in\Pi}\delta_{(\Pi\times z,g)}d\nu(z)d\mathcal L(\Pi) = \int_Z (\times z)_* (\mathcal L(\Pi)^\circ) d\nu(z)\\&=\eta\int_Z \int_G (\times z)_*(\mathcal L(\Pi)_g\times \delta_g)dg d\nu(z)=\eta
    \int_G \left(\int_Z (\times z)_*(\mathcal L(\Pi)_g\times\delta_g)d\nu(z)\right) dg\\ &=\eta  \int_G (\mathcal L(\Pi)_g\times \mathcal L(\Theta))\times \delta_g dg.
\end{align*}
By the uniqueness and $G$-equivariance of the disintegration we have 
$$
    \mathcal L(\Pi\times \Theta)_\circ=\mathcal L(\Pi\times\Theta)_1=\mathcal L(\Pi)_1\times \mathcal L(\Theta)=\mathcal L(\Pi_\circ)\times \mathcal L(\Theta).
$$
The description of the relation $\mathcal R$ follows from the way $G$ acts on product markings: $g(\omega\times z)=g\omega\times gz.$
\end{proof}

\subsection{Cost of point processes}\label{sec-cost}

Let $\Pi$ be an invariant marked point process on $G$ with nonzero intensity $\eta$. The \textit{cost} of $\Pi$ is defined as 
\begin{equation}\label{eq-costlcsc}
     \text{cost}(\Pi)-1=\eta\inf_{\mathscr{G}} \int_{\MMo(G,Z)}(\deg_{\mathscr{G}(\Pi_\circ)}(1)-1)d\mathcal L(\Pi_\circ),
\end{equation}
where the infimum ranges over all connected directed factor graphs $\mathscr{G}$ of $\Pi$. It had been known how to define cost for free actions of lcsc unimodular groups via cross-sections for many years, but the definition first appeared explicitly in \cite[Definition 34]{Carderi}. That definition is essentially equivalent to (\ref{eq-costlcsc}); see Sections 3 and 4 of \cite{SamMiklos}. We note that if $G$ is connected and noncompact, the cost of a point process on $G$ is at least $1$. 
The quantity 
$$\inf_{\mathscr{G}} \int_{\MMo(G,Z)}\deg_{\mathscr{G}(\Pi_\circ)}(1)d\mathcal L(\Pi_\circ)$$
appearing in (\ref{eq-costlcsc})
is the cost of the Palm equivalence relation (see \cite[Definition 1.5.2]{Gaboriau}).
The reason to subtract one in the point process formula is that $\cost(\Pi)-1$ scales linearly with the choice of Haar measure, so in that aspect it is a more natural invariant to consider. 

A group has \textit{fixed price} if every essentially free and invariant point process on $G$ has the same cost.
This definition is modelled after Gaboriau's fixed price property for countable groups  \cite[Definition 1.5.4]{Gaboriau}. We say that the group $G$ has \textit{fixed price one} if it has fixed price and the cost of any free point process on $G$ is $1$. As shown in \cite[Thm 1.1]{SamMiklos}, every essentially free probability measure preserving action of $G$ is isomorphic to a point process (an isomorphism here is a bijective factor map whose inverse is also a factor map). We can then define the cost of a pmp action $G \curvearrowright(X,\mu)$ to be the cost of any isomorphic finite intensity point process; this is well-defined as cost is invariant under isomorphism \cite[Section 4.1]{SamMiklos}. 

\begin{example}[Lattice shift]
    Note that by equivariance, a factor graph $\mathscr{G}$ of a lattice shift $G \acts G/\Gamma$ is determined by a subset $S \subseteq \Gamma$, and has edges of the form
    \[
        \mathscr{G}(a\Gamma) = \{(a\gamma, a\gamma s) \in a\Gamma \times a\Gamma \mid s \in S\}.
    \]
    This graph is connected exactly when $S$ generates $\Gamma$. Thus
    \[
        \cost(G \acts G/\Gamma) - 1 = \frac{d(\Gamma) - 1}{\covol(\Gamma)},
    \]
    where $d(\Gamma)$ denotes the rank of $\Gamma$ (minimum size of a generating set).
\end{example}

\begin{example}[Induced action]\label{ex-latcost}
    Suppose $\Gamma \acts (Z, \mu)$ is a free pmp action. Then the induced action $G \acts G/\Gamma \times Z$ discussed in Example \ref{inducedactiondefn} is also free. We have computed its Palm measure in Example \ref{inducedactionpalm}; it can be identified with $\mu$. In fact, the Palm equivalence relation is simply the orbit equivalence relation of $\Gamma \acts (Z, \mu)$. Thus by the definition of cost we have
    \begin{equation}\label{eq-inducedcost}
        \cost(G \acts G/\Gamma) - 1 = \frac{\cost(\Gamma \acts (Z, \mu)) - 1}{\covol(\Gamma)}.
    \end{equation}
\end{example}

Defining the fixed price of a group in terms of point processes allows us a short proof of Theorem \ref{fixedpricelattice}, assuming Theorem \ref{fixedpriceG}, using the induced action in Example \ref{ex-latcost}.

\begin{proof}[Proof of Theorem \ref{fixedpricelattice} from Theorem \ref{fixedpriceG}]
By (\ref{eq-inducedcost}), we see that fixed price $1$ for a group $G$ implies fixed price $1$ for all of its lattices.
\end{proof}

If $\Pi$ is an invariant marked point process on $G$, then a (point process) factor of $\Pi$ is a measurable and equivariant map $\Phi: \MM(G,Z)\to \MM(G,Z')$. We note that $\Phi(\Pi)$ is an invariant marked point process, with possibly different marking space. 
We continue with a discussion of weak factoring of point processes; we use weak factors in the proof of Theorem \ref{fixedpriceG}. Before we define a weak factor we mention several relevant examples and results.

\begin{example}[Thinning]\label{metricthinning}
Let $\delta > 0$ and let $d\colon G\times G\to\mathbb R_{\geq 0}$ be a left invariant semi-metric. Let $\omega\in\mathbb M(G)$. The \emph{metric thinning} is the factor map $\theta_\delta : \MM(G) \to \MM(G)$ constructed by deleting all pairs of points closer than $\delta$ belonging to an input configuration:
\[
    \theta_\delta(\omega) = \{ x \in X \mid d(x, y) > \delta \text{ for all } y \in \omega \setminus \{x\} \}.
\]

\end{example}

Cost is non-decreasing under factor maps. In particular, it is invariant under factor map isomorphisms.
\begin{lemma}[{\cite[Lemma 4.6]{SamMiklos}}]\label{lem-CostMono}
Let $\Pi\in \MM(G,Z)$ be a marked point process of finite intensity, and let $\Phi$ be a factor map of $\Pi$ such
that $\Phi(\Pi)\in \MM(G,Z')$ has finite intensity. Then
$\text{cost}(\Pi)\leq \text{cost}(\Phi(\Pi)).$
\end{lemma}

\begin{theorem}[\cite{SamMiklos}; see also \cite{Tucker-Drob} for the discrete case]
    Let $\Pi$ denote a free point process. Then its cost is equal to the cost of its IID marking (see Example \ref{ex-iidlabeling}).
\end{theorem}

Finally, we know that the Poisson point process realizes the maximal cost among all invariant free point processes on $G$. This fact is one of the key ingredients of our proof of Theorem \ref{fixedpriceG}. A similar statement for countable groups was proved in \cite{AbertWeiss}, where it is shown the Bernoulli shift $[0,1]^\Gamma$ has maximal cost among all free actions of a finitely generated $\Gamma$. 

\begin{theorem}[{\cite[Theorem 1.2]{SamMiklos}}]\label{thm-CostMaxPoisson}
    Poisson point processes have maximal cost among all free invariant point processes on $G$. In particular, the cost does not depend on the intensity.
\end{theorem}

We remark the second author proved in \cite{Wilkens} that Poisson point processes on $G$ of any intensity are isomorphic; this implies all Poisson point processes on $G$ have the same cost.

\begin{corollary}\label{cor-FPOneReduction}
    An lcsc group $G$ has fixed price $1$ if and only if any Poisson point process on $G$ has cost $1$, or equivalently, if and only if the Palm equivalence relation of any Poisson point process has cost $1$.
\end{corollary}


In order to define a weak factor, we define weak convergence of point processes.
Let $Y,Z$ be lcsc metric spaces. Suppose $\Pi_n\in \MM(Y,Z)$ is a sequence of marked point processes. We say that $\Pi_n$ \emph{weakly converges} to a marked point process $\Pi\in \MM(Y,Z)$ if the distributions of $\Pi_n$ converge weakly-* to the distribution of $\Pi$ as probability measures on $\MM(Y\times Z)$. This coincides with with the notion of weak convergence considered in \cite[Prop. 11.1.VIII(i)]{DVJv2}.\footnote{Convergence in \cite[Prop. 11.1.VIII(i)]{DVJv2} is the weak-* convergence of distributions of $\Pi_n$ to the distribution of $\Pi$ as probability measures on $\MM(G\times Z)\subset \mathcal M(Y\times Z),$ where the latter is equipped with the topology of weak-* convergence.} For example, the sequence of Poisson processes with a weakly-* converging sequence of intensity measures $\mu_n$ will weakly converge to the Poisson point process with intensity measure $\lim_{n\to\infty}\mu_n.$

\begin{lemma}\label{lem-PoissonContinuity}
The map $\mathcal M(Y) \ni\mu\mapsto \Poisson(Y,\mu)\in \MM(Y)$ is continuous with respect to the weak convergence of point processes. The same holds for the map $\MM(Y)\mapsto \Poisson(Y,\mu)^{\rm IID},$ where $\Poisson(Y,\mu)^{\rm IID}$ is the IID marking of $\Poisson(Y,\mu)$.
\end{lemma}

The above lemma can be shown by using the characterization of weak convergence in terms of Laplace functionals (see Proposition 11.1.VIII of \cite{DVJv2} and Section 3.3 of \cite{LP}).

In the remainder of this section we specialize to marked point processes on a lcsc group $G$. 

\begin{definition}[Weak factor]\label{def-WeakFactoring} Let $\Pi\in \MM(G,Z)$ be a marked point process where $Z$ is a lcsc metric space. We say $\Pi$ \emph{weakly factors} onto $\Upsilon\in \MM(G,Z)$ if there exists a sequence of point process factor maps $\Phi_n$ such that $\Phi_n(\Pi)\in \MM(G,Z)$ weakly converges to $\Upsilon$. 
\end{definition}

We prove a general analogue of cost monotonicity (Lemma \ref{lem-CostMono}) for weak factors.

\begin{theorem}[Cost monotonicity for weak factors]\label{thm-CostMonotonicity}
Let $\Pi$ be a free process and suppose $\Pi$ weakly factors onto an ergodic process $\Upsilon$. Then
\[
    \text{cost}(\Pi) \leq \text{cost}(\Upsilon).
\]
\end{theorem}
Before the proof of Theorem \ref{thm-CostMonotonicity}, we recall Theorem \ref{CostMonoCase} from \cite{SamMiklos} and prove Theorem \ref{thm-WeakTransitivity}, below. We say a point configuration $\omega\in \MM(G,Z)$ is uniformly separated if there is an open neighborhood $W\subset G$ of $1$ such that $gh^{-1}\not\in W$ for every $g\neq h\in \omega$.

\begin{theorem}[{\cite[Theorem 5.10]{SamMiklos}}]\label{CostMonoCase}
Let $Z$ be an lcsc metric space. If $\Pi_n\subset \MM(G,Z)$ is a sequence of free invariant marked point processes and $\Pi_n$ weakly converges to a uniformly separated marked process $\Pi\in \MM(G,Z)$, then
\[
    \limsup_{n \to \infty} \text{cost}(\Pi_n) \leq \text{cost}(\Pi).
\]
\end{theorem}

\begin{theorem}[Weak factor transitivity]\label{thm-WeakTransitivity}
	Weak factoring is a transitive notion for point processes of finite intensity. That is, if $\Pi, \Pi'$, and $\Pi''$ are point processes (with marks in complete and separable metric spaces $Z$, $Z'$, and $Z''$, respectively) such that $\Pi$ weakly factors onto $\Pi'$ and $\Pi'$ weakly factors onto $\Pi''$, then $\Pi$ weakly factors onto $\Pi''$.
\end{theorem}

\begin{proof}
	The proof works the same for marked and unmarked point processes. For the sake of clarity we write the proof for unmarked processes.
	
	We repeatedly make use of the following two facts:
	\begin{itemize}
		\item Pointwise convergence implies distributional convergence, and
		\item A continuous and proper map of point processes preserves weak limits.
	\end{itemize}
	Note that properness is required as weak convergence of point processes is really \emph{vague convergence,} that is, convergence of integrals against continuous and bounded functions with compact support.
	
We prove the theorem by making a series of reductions to simpler and simpler cases, and then solve these simple cases.

\begin{claim}
	It suffices to prove the following \emph{limited transitivity} statement: if $\Pi$ weakly factors onto $\Pi'$ and $\Psi(\Pi')$ is a factor of $\Pi'$, then $\Pi$ weakly factors onto $\Psi(\Pi')$.
\end{claim}
	
To see this, suppose $\Phi_n(\Pi)$ weakly converges to $\Pi'$, and $\Phi'_n(\Pi')$ weakly converges to $\Pi''$. For each $i \in \NN$, produce a sequence $\Psi_n^i(\Pi)$ weakly converging to $\Phi'_i(\Pi')$. By using a metric for weak convergence one can extract a subsequence of $\Psi_n^i(\Pi)$ weakly converging to $\Pi''$, as desired.

We now prove that statement. We will express $\Psi(\Pi')$ as a composition of two factor maps. One is a labelling, and the other is an ``implementation'' map. We first show that the limited transitivity statement holds if $\Psi(\Pi')$ is an \emph{arbitrary} factor marking (not necessarily a $\Xi$-factor, see the next sentence). A factor $\Xi$-marking is a measurable and equivariant map $\mathscr{C} : \MM(G) \to \MM(G, \Xi)$ such that the underlying set of $\mathscr{C}(\omega)$ is $\omega$. One thinks of $\mathscr{C}$ as assigning a color or mark to each point of the input configuration. 

Note that every $\Xi$-marked point process $\Upsilon$ is the limit of point processes with only finitely many $\Xi$-marks. To see this, fix a countable dense subset $Q \subseteq \Xi$ and enumerate it as $Q = \{q_1, q_2, \ldots, \}$. Define approximation functions $A_n\colon\Xi \to Q$ by setting $A_n(\Xi)$ to be the closest element of $Q$ to $\Xi$ amongst the first $n$ elements of $Q$. If there is a tie, then choose the element of $Q$ with the lowest index. Then $A_n(\Upsilon)$ converges to $\Upsilon$ pointwise, so if we can solve the limited transitivity problem for each $A_n(\mathscr{C}(\Pi'))$ then we are done. So we have reduced to the case where $\mathscr{C}(\Pi')$ is a factor marking with only finitely many types of marks.

Observe that $\mathscr{C}$ induces a \emph{partition} $\MMo = P_1 \sqcup \cdots \sqcup P_d$ into measurable pieces, where $d \in \NN$ is the number of marks (for more information on this, see Section 3.1 of \cite{SamMiklos}). Choose a sequence of partitions $\MMo(G) = S_1^i \sqcup \cdots \sqcup S_d^i$, where
\begin{itemize}
	\item Each $S_k^i$ is a continuity set for the Palm measure of $\Pi'$, that is, $\PP[\Pi'_\circ \in \partial S_k^i] = 0$, and
	\item Each $S_k^i$ approximates the corresponding $P_k^i$ well, so $\PP[\Pi'_\circ \in S_k^i \triangle P_k^i] < 2^{-i}$.
\end{itemize}
The approximating partitions determine a sequence of factor markings $\mathscr{C}_i$ and the first condition guarantees that $\mathscr{C}_i(\Phi_n(\Pi))$ converges weakly to $\mathscr{C}_i(\Pi')$ as $n\to\infty$. Note that by Borel-Cantelli,
\[
	\PP[\text{for all } k \in \NN, \Pi'_\circ \in S_k^i \triangle P_k^i \text{ for infinitely many } i \in \NN] = 0,
\]
and therefore $\mathscr{C}_i(\Pi')$ converges pointwise almost surely to $\mathscr{C}(\Pi')$. This implies weak convergence, and so we have established the limited transitivity statement for arbitrary factor markings.

We return to the original limited transitivity statement in the claim. We introduce the following factoring marking, with marks valued in $\MM(G)$:
\[
	\mathcal{C}(\Pi') = \{ (g, g^{-1}(C^{\Pi'}(g) \cap \Psi(\Pi'))) \in G \times \MM(G) \mid g \in \Pi' \}.
\]
In words, each point $g \in \Pi'$ looks at the points of $\Psi(\Pi')$ that fall in its Voronoi cell $C^{\Pi'}(g)$ and records them (the shift by $g^{-1}$ is to make it equivariant). For $g \in \Pi'$, we write $\ell(g) := g^{-1}(C^{\Pi'}(g) \cap \Psi(\Pi'))$ for its label (which implicitly depends on $\Pi'$).

We know that $\Pi$ weakly factors onto $\mathcal{C}(\Pi')$ by our earlier work on limited transitivity for factor markings. We now define a further factor map, the \emph{implementation} factor
\[
	I(\mathcal{C}(\Pi')) = \bigcup_{g \in \Pi'} g\ell(g) = \Psi(\Pi').
\]
Our goal is to approximate $\mathcal{C}(\Pi')$ by other factor markings for which the implementation map is continuous and proper. This gives us the full weak factoring result.

We introduce the notation $\MM_{<R} := \{ \omega \in \MM(B_G(R)) \mid \omega(B_G(R)) < R \}$ for the set of subsets of the $R$-ball $B_G(R)$ with at most $R$ elements. It is easy to see that the implementation map defined as a map $I : \MM(G, \MM_{<R}) \to \MM(G)$ is continuous and proper. 

Recall that in the factor marking $\mathcal{C}(\Pi')$ each point $g \in \Pi'$ has a label consisting of a finite\footnote{Finiteness is a consequence of finite intensity and the fact that Voronoi cells almost surely have finite volume, by the Voronoi inversion formula \cite[Section 9.4]{LP}.} configuration in $G$. However, this set can be arbitrarily large and span an arbitrarily large region of $G$. We wish to truncate this region. Let $T_R$ be any measurable and equivariant rule which takes a finite configuration $\omega$ and chooses $R$ points of $\omega$ contained in $B_G(R)$ in such a way that $T_R(\omega)$ is increasing in $R$ and $\bigcup_{R>0}T_R(\omega)=\omega.$ We define new factor markings
\[
	\mathcal{C}_R(\Pi') = \{ (g, T_R(g^{-1}(C^{\Pi'}(g) \cap \Psi(\Pi')))) \in G \times \MM(G) \mid g \in \Pi' \}.
\]
In words, each point $g \in \Pi'$ looks at the points of $\Psi(\Pi')$ that fall in its Voronoi cell and records their $R$-truncation. Note that $\Pi$ weakly factors onto each $\mathcal{C}_R(\Pi')$. By continuity and properness, $\Pi$ also weakly factors onto $I(\mathcal{C}_R(\Pi'))$, which converges pointwise to $\Psi(\Pi')$, as $R\to\infty$. Hence $\Pi$ also weakly factors onto $\Psi(\Pi')$, as desired.
\end{proof}

\begin{proof}[Proof of Theorem \ref{thm-CostMonotonicity}]
Consider the metric thinning $\theta_\delta(\Upsilon)$ of $\Upsilon$. This sequence of point processes converges to $\Upsilon$ pointwise as $\delta\to 0$, so it must be non-empty with positive probability for sufficiently small $\delta$. Then, by ergodicity of $\Upsilon$, it is almost surely non-empty. Now define the factor $\Psi(\Upsilon)\in \MM(G,\MM(G))$ as follows:
\[
    \Psi(\Upsilon) := \theta_\delta(\Upsilon)\times\Upsilon =\{ (g, g^{-1}\Upsilon) \in G \times \mathbb{M}(G) \mid g \in \theta_\delta(\Upsilon) \}.
\]
Since it is a product marking, the mark at any point of $\Psi(\Upsilon)$ determines $\Upsilon$ uniquely. It follows that $\Psi$ is an $G$-equivariant isomorphism, so $\text{cost}(\Upsilon) = \text{cost}(\Psi(\Upsilon))$. But $\Psi(\Upsilon)$ is a $\delta$ uniformly separated marked point process and $\Pi$ weakly factors onto $\Psi(\Upsilon)$ by transitivity (Theorem \ref{thm-WeakTransitivity}), so by Theorem \ref{CostMonoCase} we get
\[
    \text{cost}(\Pi) \leq \text{cost}(\Psi(\Upsilon)) = \text{cost}(\Upsilon),
\]
as desired.
\end{proof}

\begin{theorem}[Cost of a product marking by a weak limit process]\label{WeakLimitTheorem}
Let $\Pi$ be the IID Poisson point process on $G$ and $Z$ a lcsc metric space with a continuous $G$-action. Suppose $\Phi_n(\Pi)\in \MM(Z)$ is a sequence of measurable and equivariant $Z$-valued factors of $\Pi$ with the property that $\Phi_n(\Pi)$ weakly converges to a random process $\Upsilon\in\MM(Z)$. Let $\Pi \times \Upsilon\in \MM(G,[0,1]\times\MM(Z))$ be the product marking defined as in Example \ref{def-ProductMarking}. Then $$\cost(\Pi)=\cost(\Pi \times \Upsilon).$$
\end{theorem}

\begin{proof}
Note that the IID marked Poisson $\Pi$ on $G$ factors onto the product marking $\Pi\times \Pi'$ where $\Pi'$ is an independent copy of the IID Poisson on $G$, so $\Pi$ and $\Pi\times\Pi'$ have the same cost. To construct such a factor, recall that IID marked point processes factor onto the Poisson (see for example Proposition 5.1 of \cite{SamMiklos}). Note that having one IID uniform $[0,1]$ label at each point is equivalent, in the sense of isomorphism, to having two independent ones (as $[0,1]$ with Lebesgue measure and $[0,1]^2$ with the corresponding product measure are abstractly isomorphic as measure spaces). If we use these two separate labels in the above construction, then the resulting two factor processes are independent. Finally, note that if a process $\Pi$ factors onto $\Upsilon$, then the IID marking of $\Pi$ factors onto the IID marking of $\Upsilon$ (see for example Lemma 6.5 of \cite{SamMiklos}).

Consider the sequence of product markings $\Pi\times \Phi_n(\Pi')$. Each element in the sequence is a factor of $\Pi\times \Pi'$ by the map ${\rm id}\times \Phi_n$. Since the sequence of point processes  $\Phi_n(\Pi')$ converges weakly to $\Upsilon$, we get that the sequence of marked point processes $\Pi\times \Phi_n(\Pi')$ converges weakly to $\Pi\times \Upsilon$. By Theorem \ref{thm-CostMonotonicity}, 
$$\cost(\Pi)=\cost(\Pi\times \Pi')\leq \cost(\Pi\times \Upsilon).$$
The equality follows now from Theorem \ref{thm-CostMaxPoisson}.
\end{proof}

\section{Corona actions}\label{sec-coronas}

\subsection{The space of distance-like functions}

Let $G$ be an lcsc group acting continuously, properly and transitively by isometries on a metric space $X$. These assumptions force $X$ to be locally compact second countable. Fix a root $o\in X$ and let $K=\Stab_G o$, so that $$G/K\ni gK\mapsto go\in X$$ is an isomorphism of $G$-spaces. The properness of the action implies $K$ is a compact subgroup of $G$. Let $\vol$ be the push-forward of the Haar measure $m_G$ via the map $G\ni g\mapsto go\in X.$ Write $B_X(x,r)=\{y\in X\mid d(x,y)\leq r\}.$ To shorten notation we write $B(r):=B_X(o,r)$ and $B_G(r):=B(r)K$. Note that the balls of positive radius have positive measure. Consider the set of ``distance-like'' functions  
$$D:=\overline{\{y\mapsto d(x,y)+t\mid x\in X,t\in \mathbb R\}}\subset \mathcal C(X)=\mathcal C(G)^K,$$ where the closure is taken with respect to the topology of uniform convergence on compact sets. We identify $\mathcal C(X)$ with the space of right $K$-invariant continuous functions in $\mathcal C(G)$ and think of functions on $X$ as functions on $G$ whenever convenient. For $f\in D,g\in G$ we set $f(g):=f(gK).$ Since any function in $D$ is $1$-Lipschitz, the space $D$ is lcsc (and hence metrizable; for an explicit choice of metric see the proof of Lemma \ref{lem-unifcontinuity} below). 
The group $G$ acts continuously on $D$ by left translations $gf(x):=f(g^{-1}x).$ For each positive real number $r$ we fix a compact subset $$D_r:=\{f\in D\mid |f(o)|\leq r\}\subset D.$$ 
The collection of such subsets exhausts $D$.

In the sequel, we require the following basic fact regarding the topology of $D$.

\begin{lemma}\label{lem-unifcontinuity}
    Let $\varphi\colon D\to \mathbb R$ be continuous and compactly supported. For any $\varepsilon>0$ there exists a compact subset $\Omega\subset X$ and $\delta>0$ such that $|\varphi(f_1)-\varphi(f_2)|\leq \varepsilon$ for every $f_1,f_2\in D$ with $\sup_{x\in\Omega}|f_1(x)-f_2(x)|\leq\delta.$
\end{lemma}

\begin{proof}
    Let $\Omega_n\subset X$ be a sequence of compact subsets of $X$ such that $\bigcup_{n=1}^\infty \Omega_n=X$. We metrize the topology of uniform convergence on compact sets of $X$ with 
    $$d_{\mathcal C(X)}(f_1,f_2):=\sum_{n=1}^\infty \frac{1}{2^n}\sup_{x\in \Omega_n}\min\{|f_1(x)-f_2(x)|,1\}.$$
    Any continuous function on a compact set is uniformly continuous, so there exists $\delta_0>0$ such that $|\varphi(f_1)-\varphi(f_2)|\leq\varepsilon$ for every $f_1,f_2\in D$ such that $d_{\mathcal C(X)}(f_1,f_2)\leq \delta_0$. Choose $n_0$ such that $1/2^{n_0}\leq \delta_0/2$, and set $\Omega:=\bigcup_{n=1}^{n_0} \Omega_n$. Then the condition $\sup_{x\in \Omega}|f_1(x)-f_2(x)|\leq \delta$ implies $$d_{\mathcal C(X)}(f_1,f_2)\leq \sum_{n=1}^{n_0}\frac{1}{2^{n}}\delta+\sum_{n=n_0+1}^\infty \frac{1}{2^n}\leq \delta+\frac{\delta_0}{2}.$$ The lemma follows with $\delta=\delta_0/2.$
\end{proof}

\subsection{Corona actions}\label{sec-CoronaActions}
Let $x,y\in X$. For $t\in \mathbb R$ we define the $G$-equivariant embedding $\iota_t\colon X\to D$ as
$$ \iota_t(x)(y)=d(x,y)-t.$$
We use these embeddings to construct a tight sequence of $G$-invariant measures on $D$. We introduce the real parameter $\eta>0$ related to $t$ by $\eta=\eta(t):=\vol(B(t))^{-1}$. Note the condition $t\to \infty$ is equivalent to $\eta\to 0$. Consider the one parameter family of $G$-invariant locally finite Borel measures $\mu_t$ on $D$, for $t>0$, obtained as push-forward measures:
\begin{equation}\label{eq-mueatdef}\mu_t:=\eta(t) (\iota_t)_*(\vol).\end{equation}
We normalize by $\eta(t)$ so that $\mu_t(\{f\in D\mid f(o)\leq 0\})=1$.

We are ready to define the corona actions, which are the central object of Section \ref{sec-coronas}. The name ``corona'' is borrowed from \cite{ACELU23}, where a similar object was independently introduced to describe ideal Poisson-Voronoi tessellations for real hyperbolic spaces. Our approaches to constructing IPVT are different, in particular, our method does not require the rank one condition. On the other hand, we do not focus on the convergence question i.e. we do not formally show that Poisson-Voronoi tessellations converge to the IPVT in any precise way. 

\begin{definition}[compare with \cite{ACELU23}]\label{def-corona}
 Let $G$ be a lcsc group acting properly transitively and isometrically on a locally compact metric space $(X,d)$. Let  $\mu$ be a subsequential weak-* limit of $\mu_t$ as $t\to\infty$. The action $G\curvearrowright (D,\mu)$ is a \emph{corona action}, or just corona, for $G$ acting on $X$.
\end{definition}

 We say that a corona is non trivial if $\mu\neq 0.$ 
 We remark that for any corona $(D,\mu),$ the measure $\mu$ is automatically $G$-invariant since it arises as a weak-* limit of $G$-invariant measures. 
\begin{proposition}\label{prop-WeakLimitHoro}
    Suppose that there exists $a,r_0>0$ such that $\vol(B(t+r_0))\leq e^{a}\vol(B(t)),$ for every $t>1$. Then, the sequence $\mu_t$ is relatively compact in $\mathcal M(D)\setminus \{0\}$ and $\mu_t(\{f\in D\mid f(o)\leq -s\})\leq e^{-a\lfloor \frac{s}{r_0}\rfloor}$ for all $t\geq s+1$.
\end{proposition}
\begin{proof}
    The space $D$ can written as an ascending union of compact sets $D=\bigcup_{r>0} D_r$. To prove the proposition we need to show that $\mu_t(D_r)$ is uniformly bounded for all $t>1$ and that $0$ is not an accumulation point of the sequence $\mu_t$. We have
    $$\mu_t(D_r)=\frac{1}{\vol(B(t))} \int_X \mathds 1_{|d(x,o)-t|\leq r}d\vol(x)= \frac{\vol(B(t+r))}{\vol(B(t))}\leq e^{a\lceil \frac{r}{r_0}\rceil}.$$ To verify that zero in not an accumulation point  we note that 
    $$\mu_t(D_{r_0})\geq \frac{\vol(B(t+r_0))-\vol(B(t-r_0))}{\vol(B(t))}\geq e^a-e^{-a}>0.$$
    
    Finally, for $t\geq s+1$ we have 
    \[
    \mu_t(\{f\in D\mid f(o)\leq -s\})= \frac{\vol(B(t-s))}{\vol(B(t))}\leq e^{-a\lfloor \frac{s}{r_0}\rfloor}.\qedhere 
    \] 
\end{proof}

Proposition \ref{prop-WeakLimitHoro} implies that non-trivial coronas will exist for any non-amenable lcsc group $G$. 
\begin{corollary}\label{cor-coronasExistence}
    Let $G,X,K,d$ be as above, with $G$ non-amenable. Then, the sequence of measures $\mu_t$ is relatively compact in $\mathbb M(D)\setminus\{0\}$. Moreover, any weak-* limit $\mu$ of $\mu_t$ satisfies $\mu(\{ f\in D \mid f(o)\leq 0\})=1$ and $\mu(\{ f\in D \mid f(o)\leq r\})<\infty$ for any $r\in\mathbb R.$
\end{corollary}

\begin{proof}
    By Proposition \ref{prop-WeakLimitHoro}, to prove relative compactness, it is enough to verify that $\vol(B(t+r_0))\leq e^{a}\vol(B(t)),$ for some $r_0,a>0$ and every $t>1$. Let $r_0>0$ be such that $B_G(r_0)$ generates the group and define the averaging operator $M\colon L^2(G,dg)\to L^2(G,dg)$ by
    $$ M\phi(g):=\frac{1}{\vol(B(r_0))}\int_{B(r_0)}\phi(h^{-1}g)dh.$$
    The fact that $G$ is non-amenable is equivalent to $\|M\|_{\rm op}<1,$ where $\|\cdot\|_{\rm op}$ stands for the operator norm. By Cauchy-Schwartz 
    \begin{align*}\|M(\mathds 1_{B(t)})\|_2^2\|\mathds 1_{B(t+r_0)}\|_2^2&\geq \left(\int_{G}M(\mathds 1_{B_G(t)})(g)dg\right)^2=\vol(B(t))^2,\\
    \frac{\vol(B(t+r_0))}{\vol(B(t))}&\geq \|M\|_{\rm op}^{-1}>1.\end{align*}
    This concludes the proof of relative compactness. Inequality $\mu_t(\{f\in D\mid f(o)\leq -s\})\leq e^{-a\lfloor \frac{s}{r_0}\rfloor}$, which now follows from Proposition \ref {prop-WeakLimitHoro}, implies that $$1-e^{-a\lfloor \frac{s}{r_0}\rfloor}\leq \mu_t(\{f\in D\mid -s<f(o)\leq 0\})\leq 1 \text{ for }t\geq s+1,$$ so the same inequality holds for $\mu$. Taking $s\to\infty$ we get $\mu(\{f\in D\mid f(o)\leq 0\})=1.$ Finally $\mu(\{ f\in D\mid f(o)\leq r\})\leq \mu(\{f\in D\mid f(o)\leq 0\})+\mu(D_{\leq r})=1+\mu(D_{\leq r})<\infty,$ because $D_{\leq r}$ is compact.
\end{proof}

We can now define the ideal Poisson-Voronoi tessellations in the abstract setting of a proper isometric action $G\curvearrowright X.$
\begin{definition}\label{def-AbstIPVT}
Let $G$ be an lcsc group acting continuously, properly and transitively by isometries on a metric space $X$. Let $(D,\mu)$ be a corona action for $G\curvearrowright X$ and let $\Upsilon$ be the Poisson point process on $D$ with mean measure $\mu$. The \emph{ideal Poisson-Voronoi tessellation} associated to $(D,\mu)$ is the generalized Voronoi tessellation $\Vor(\Upsilon).$ It is a tessellation of $X$ or equivalently a tessellation of $G$ with right $K$-invariant cells. If this tessellation does not depend on the choice of the corona action, we simply call it \emph{the} ideal Poisson-Voronoi tessellation of $X.$
\end{definition}

 In this paper we carry out the computation of explicit examples of coronas for semisimple real Lie groups and products of automorphism groups of trees. We expect the general structure of the argument extends to all CAT$(0)$ groups with a Gelfand pair (see \cite{Monod}). For semisimple real Lie groups, the corona actions admit an elegant description. 

\begin{theorem}\label{thm-LieGpcorona}
Let $G$ be a real semisimple Lie group with a minimal parabolic subgroup $P$. Let $U$ be the kernel of the modular character of $P$.  The corona $(D,\mu)$ for the action of $G$ on its symmetric space $X$ is unique up to scaling. As a $G$-action it is isomorphic to $(G/U,dg/du),$ where  $dg/du$ is a $G$-invariant measure on $G/U$.
\end{theorem}
\begin{remark}
For all real $\lambda>0$ the actions $G\curvearrowright (G/U,\lambda dg/du)$ are isomorphic. In particular invariant Poisson point processes on $G/U$ are all isomorphic.   
\end{remark}
The isomorphism between $(D,\mu)$ and $(G/U,dg/du)$ arises as follows. The measure $\mu$ is supported on a single $G$ orbit, consisting of Busemann functions of special type (see (\ref{eq:SpecialOrbit})), and the stabilizer of any function in this orbit is conjugate to $U$.  When $G=\SO(n,1), X=\mathbb H^n$, there is only one orbit of Busemann functions and the corona action is line bundle over the visual boundary of $\mathbb H^n,$ which is a structure shared with the corona described in \cite{ACELU23}.
In the case of products of automorphism groups of regular trees the description of coronas is not as clean, but we can still verify all the properties relevant for the proofs of the fixed price property. In both cases we have:
\begin{theorem}\label{thm-Treecoronas} 

 Let $G$ be a semisimple Lie group acting on its symmetric space $X$ or let $T_i, i=1,\ldots, m$ be regular trees, $G=\prod_{i=1}^m{\Aut}(T_i)$ and let $X$ be the product of vertex sets of $T_1,\ldots, T_m$ with induced metric. Let $G\curvearrowright (D,\mu)$ be a corona action. Then:
\begin{enumerate}
    \item The action of $G$ on $(D,\mu)$ is amenable.
    \item For any finite measure set $E\subset D$ we have $\lim_{g\to\infty} \mu(E\cap gE)=0.$
    \item If $G$ is a higher rank semisimple Lie group or $G=\prod_{i=1}^m{\Aut}(T_i)$ and $m\geq 2$, then the action of $G$ on $(D,\mu)$ is doubly recurrent.
\end{enumerate}
\end{theorem}
We recall that a measure class preserving action of $G$ on $(Z,\nu)$ is called \emph{conservative} if the conclusion of the Poincare recurrence lemma holds: for every measurable $A\subset Z$ and $\nu$-almost every $z\in A$ the set $\{g\in G\mid gz\in A\}$ is unbounded. An action $G\curvearrowright (Z,\nu)$ is \emph{doubly recurrent} if the diagonal action of $G$ on $(Z\times Z, \nu\times \nu)$ is conservative. For the definition of an amenable action we refer to \cite{Zimmer78} and \cite{kaimanovich2005amenability}.
The proofs of Theorems \ref{thm-LieGpcorona} and \ref{thm-Treecoronas} are carried out in Sections \ref{sec-LieGpcorona} for semisimple real Lie groups and Section \ref{sec-Treecoronas} for tree automorphisms, respectively.

\section{Computations for real semisimple Lie groups}\label{sec-LieGpcorona}

The goal of this section is to prove Theorem \ref{thm-LieGpcorona}. This boils down to an explicit although quite involved computation. For a gentle introduction to the analysis on semisimple Lie groups we refer to \cite{Knapp, GV88}. We follow mostly the standard notation of \cite[Chapter 2]{GV88} with one important difference; our Iwasawa decomposition is $NAK$, not $KAN$ as in \cite{GV88}.  

\subsection{Semisimple Lie groups}\label{ssec-LGpNot} 

Let $G$ be a semi-simple real Lie group with a maximal compact subgroup $K$ stabilized by the Cartan involution $\Theta$ (see \cite[Def. 2.1.9]{GV88}), a maximal split torus $A$ stabilized by $\Theta$ and a minimal parabolic $P$ containing $A$. Let $P=MAN$ be the Langlands decomposition of $P$ (see \cite[(2.3.6)]{GV88}), where $N$ is the unipotent radical and $M$ is the maximal compact subgroup of the centralizer of $A$ in $G$. 
Unless stated otherwise, we will write lowercase script letter to denote the Lie algebras, in this way $\mathfrak g$ is the Lie algebra of $G$, $\mathfrak a$ is the Lie algebra of $A$ e.t.c. We record that $M=P\cap K$ is the centralizer of $A$ in $K$. Let $\mathfrak a$ be the Lie algebra of $A$ and let $\mathfrak{a}^*:=\Hom(\mathfrak a,\R)$ be the dual space to $\mathfrak a$. We write $W=N_G(A)/A$ for the Weyl group of $A$. The Weyl group acts on $A,\mathfrak a,\mathfrak{a}^*$ and can be realized inside $K$, i.e. for every $w\in W$ there exists $k\in K$ such that $k^{-1}ak=w^{-1}aw.$  Inside $\mathfrak a$ we distinguish the positive Weyl chamber:
$$\mathfrak{a}^{++}:=\{H\in \mathfrak a\mid \lim_{t\to+\infty} e^{-tH}ne^{tH}=1 \text{ for every } n\in N\},$$
consisting of elements $H$ such that $e^{H}$ uniformly contracts the unipotent radical $N$. The closure of $\mathfrak a^{++}$ will be denoted by $\mathfrak{a}^+$. The set of Weyl chambers is given by $\{w\mathfrak{a}^+\mid w\in W\}$. Weyl group acts on it freely transitively and we have $\mathfrak a=\bigcup_{w\in W}w\mathfrak{a}^+$ \cite[\S2.2]{GV88}.
We fix an inner product on $\mathfrak g$ by setting 
\begin{equation}\label{eq:Inner}\langle Y, Z\rangle:=\tr(\ad(Y)\ad(\Theta Z))\text{ for } Y,Z\in \mathfrak g.\end{equation}
It restricts to a $W$-invariant inner product on $\mathfrak a$. Write $\|H\|^2=\langle H,H\rangle$ for the resulting Euclidean norm. We write $S(\mathfrak a),S(\mathfrak a^+),S(\mathfrak a^{++})$ for the unit sphere and its intersections with $\mathfrak a^+, \mathfrak a^{++}$ respectively. The \emph{Cartan decomposition} or \emph{polar decomposition} \cite[(2.2.26)]{GV88} of $G$ tells us that every element $g\in G$ can be written as $g=k_1\exp(a(g))k_2$, where $k_1,k_2\in K$ and $H\in\mathfrak{a}^+$. The central component $a(g)$ is defined uniquely up to the action of the Weyl group \cite[Lemma 2.2.3]{GV88}. In particular, the length $\|a(g)\|$ is defined unambiguously. 
Let $\Sigma\subset \mathfrak{a}^*$ be the set of roots of $G$ relative to $A$, $\Sigma^+$ be the subset of positive roots and let $\rho$ be the half-sum of positive roots \cite[(2.4.3)]{GV88}. Half sum of positive roots can be also defined by its relation to the modular character of $P$, denoted $\chi_P$ \cite[(2.4.5) and (2.4.6))]{GV88}:
$$\chi_P(e^Hmn)=e^{2\rho(H)}, \text{ for any }H\in\mathfrak a, m\in M,n\in N$$

We use the modular character $\chi_P$ to define the subgroup $U:=\ker \chi_P$. We record the key property of $U$ which is responsible for \Cref{cellpairs}.
\begin{lemma}\label{lem-UIntersections}
    For almost every pair $g_1U,g_2U\in G/U$ the intersection $U^{g_1}\cap U^{g_2}$ is conjugate to the subgroup $\ker\chi|_{AM}.$ In particular, it is a noncompact amenable subgroup of $G$ isomorphic to $\mathbb R^{d-1}\times M,$ where $d=\dim\mathfrak a$ is the real rank of $G$.
\end{lemma}
\begin{proof}
Since $U^{g_1}\cap U^{g_2}$ is conjugate to $U\cap U^{g_2^{-1}g_1},$ it is enough to show that $U\cap U^g$ is conjugate to $\ker\chi_P|_{AM}$ for almost all $g\in G$. By the Bruhat decomposition \cite[(2.5.5)]{GV88} we have  
$$G=\sqcup_{w\in W} PwP.$$ There is a unique cell $Pw_0P$ in this decomposition which is open and has non-zero measure \cite[p.79]{GV88}. It corresponds to the longest element of the Weyl group, denoted $w_0$, characterized by the property that $\rho^{w_0}=-\rho$ or by the fact that $\mathfrak n$ is orthogonal to $\mathfrak n^{w_0}$. Ignoring a measure zero set, we can write $g=p_1w_0p_2, p_1,p_2\in P$. Then $U^{g}\cap U=U^{p_1w_0p_2}\cap U=(U^{w_0}\cap U)^{p_1}$. 
We argue that $U^{w_0}\cap U=\ker\chi_P|_{AM}.$ Since $$(\ker\chi_P|_{AM})^{w_0}=M\ker \exp\rho^{w_0}=M\ker\exp(-\rho)=\ker(\chi_P^{-1})|_{AM}=\ker\chi_P|_{AM},$$ we get the inclusion $\ker\chi_P|_{AM}\subset U^{w_0}\cap U.$ To show the reverse inclusion we compute the Lie algebra $\mathfrak u$ of $U$ and its $w_0$-conjugate:
$$\mathfrak u=\ker \rho|_{\mathfrak a} + \mathfrak m +\mathfrak n, \quad \mathfrak u^{w_0}=\ker (-\rho)|_{\mathfrak a} + \mathfrak m +\mathfrak n^{w_0}.$$ These decompositions are orthogonal with respect to the inner product $\langle\cdot,\cdot\rangle$ on $\mathfrak g$, so $\mathfrak u\cap \mathfrak u^{w_0}=\ker \rho|_{\mathfrak a}+\mathfrak m$. The latter is the Lie algebra of $\ker\chi_P|_{AM}.$ Every element of $U^{w_0}\cap U$ has to normalize $\ker \rho|_{\mathfrak a}+\mathfrak m$, so it also normalizes its centralizer $\mathfrak a+\mathfrak m$. This implies $U^{w_0}\cap U\subset WAM\cap U=AM\cap U=\ker\chi_P|_{AM}.$ For the step $WAM\cap U=AM\cap U$ we used the fact that $WAM\cap P=AM,$ which follows from the disjointedness in the Bruhat decomposition.
\end{proof}
The \emph{Iwasawa decomposition} \cite[(2.2.7)]{GV88} asserts that every element $g\in G$ can be written uniquely as $$g=n(g)e^{H(g)}k(g), \quad n(g)\in N, H(g)\in \mathfrak a, k(g)\in K.$$ We note that our convention is different than the one used in \cite{GV88}, we use the order $NAK$ instead of $KAN$ to define $H$ (see \cite[(2.2.11)]{GV88}), since we want $H$ to be a right $K$-invariant function on $G$. 
The function $H$ will play a critical role in determining the limits of appropriately shifted distance functions (see Proposition \ref{prop-DistanceEst}).

\begin{example} Take $G=\SL(3,\mathbb R)$. We have $$K=\SO(3),\quad P=\begin{pmatrix} \ast & \ast & \ast\\ 0 & \ast & \ast \\ 0 & 0 & \ast\end{pmatrix},\quad A=\begin{pmatrix} \ast & 0 & 0\\ 0& \ast & 0\\ 0& 0&\ast\end{pmatrix},\quad N=\begin{pmatrix}
    1 & \ast & \ast\\ 0 & 1 & \ast\\ 0 & 0 & 1 \end{pmatrix},\quad M=\begin{pmatrix} \pm 1 & 0 & 0\\ 0& \pm 1 & 0\\ 0&0& \pm 1\end{pmatrix}.$$
Cartan involution $\Theta$ is the inverse transpose. The Lie algebra $\mathfrak a$ is identified with $\{s_1+s_2+s_3=0\}\subset \mathbb R^3$  via
$$\mathfrak a=\begin{pmatrix} s_1 & 0 & 0\\ 0 & s_2 & 0\\ 0 & 0 & s_3
\end{pmatrix}, \quad s_1+s_2+s_3=0 \quad \text{ and }\quad \mathfrak a^{++}=\begin{pmatrix} s_1 & 0 & 0\\ 0 & s_2 & 0\\ 0 & 0 & s_3
\end{pmatrix}\in\mathfrak a,\text{ with } s_1>s_2>s_3.$$ The roots are $\lambda_{i,j}\in \mathfrak a^*,$ $\lambda_{i,j}(H)=s_i-s_j$ for $H={\rm diag}(s_1,s_2,s_3)$. We note that $e^{-H}\exp(E_{i,j})e^{H}=\exp(e^{-\lambda_{i,j}(H)}E_{i,j})$, where $E_{i,j}$ is the elementary matrix with the only non zero entry in the $i,j$-place. The half sum of positive roots $\rho$ is equal to $\lambda_{1,3}$ (although in general $\rho$ does not have to be equal to any of the roots). Finally, the group $U$, element $w_0$ and the intersection $U^{w_0}\cap U$ (as in the proof of Lemma \ref{lem-UIntersections}) are
$$U=\begin{pmatrix}
\pm e^s & \ast & \ast \\ 0 & \pm e^{-2s} & \ast \\ 0 & 0 & \pm e^s
\end{pmatrix},\quad w_0=\begin{pmatrix}
0 & 0 & 1\\ 0 & -1 & 0 \\ 1 & 0 & 0
\end{pmatrix},\quad U^{w_0}\cap U=\begin{pmatrix}
\pm e^s & 0 & 0\\ 0 & \pm e^{-2s} & 0 \\ 0 & 0 & \pm e^s
\end{pmatrix}.$$
\end{example}
\subsection{Integration formulas} Let $J\colon \mathfrak{a}\to\mathbb R_{\geq 0}$ be given by $J(H):=\prod_{\alpha\in\Sigma^+}|e^{\alpha(H)}-e^{-\alpha(H)}|.$ The formula for the integration in spherical coordinates is \cite[Proposition 2.4.6]{GV88}
\begin{equation}\label{eq-SphericalInt}\int_G f(g)dg=\int_K\int_K\int_{\mathfrak a^{++}}f(k_1e^{H}k_2)J(H)dHdk_1dk_2,\end{equation} where $dk_1,dk_2$ are suitably normalized Haar measures on $K$ and $dH$ is the usual Lebesgue measure on the Euclidean space $\mathfrak a$. Let $c_K:=\int_Kdk$. Formula \eqref{eq-SphericalInt} yields:
$$\int_X h(x)d\vol(x)=c_K \int_K\int_{\mathfrak a^+}h(k_1e^XK)J(X)dk_1.$$
The above formula will also be referred to as the integration in spherical coordinates. 
The formulas for integration in Iwasawa coordinates (in $KAN$ or $KNA$ order) are \cite[Proposition 2.4.3 and Corollary 2.4.4]{GV88}
\begin{equation}\label{eq-IwasawaInt}\int_G f(g)dg=\int_K\int_N\int_{\mathfrak a} f(kne^{X})dXdndk=\int_K\int_{\mathfrak a}\int_N f(ke^{X}n)e^{2\rho(X)}dndXdk,\end{equation} where $dX,dk$ are as before and $dn$ is a suitably normalized Haar measure on $N$.
\subsection{Symmetric space and Busemann functions}\label{sec-SymSpace} The quotient space $X=G/K$ is equipped with a canonical left $G$-invariant Riemannian metric, induced by the Killing form on the Lie algebra of $G$. The resulting metric $d\colon X\times X\to \R_{\geq 0}$, makes $X$ into a $CAT(0)$ space \cite{BH}. The distance function can be described explicitly using the Cartan decomposition \cite[(4.6.24)]{GV88}. For every $g_1K,g_2K\in X$ we have
$$d(g_1K,g_2K)=d(K,g_1^{-1}g_2K)=\|a(g_1^{-1}g_2)\|.$$
Throughout the Section \ref{sec-LieGpcorona} we fix a base point $o\in X$ corresponding to the trivial coset $K$. The symmetric space $X$ is equipped with the volume form $\vol$ induced by the Riemannian metric. We fix the Haar measure $m_G$ on $G$ so that $\vol$ is the pushforward of $m_G$. For every measurable set $A\subset X$ we write $\vol(A)$ for its volume. 
The \emph{visual boundary} of $X$, denoted $\partial X$, is the set of equivalence classes of geodesic rays \cite[II.8]{BH}. Two geodesic rays $\gamma_1,\gamma_2\colon [0,+\infty)\to X$ are declared equivalent if $\limsup_{t\to\infty} d(\gamma_1(t),\gamma_2(t))<+\infty.$ Every class in $\partial X$ is represented by a unique geodesic ray starting at $o$. 
Let $\xi=[\gamma]\in \partial X$. The Busemann function (c.f. \cite[8.17]{BH} and) $b_\xi\colon X\times X\to \mathbb R$ is defined as 
$$b_\xi(x,y)=\lim_{t\to\infty}\left[d(x,\gamma(t))-d(y,\gamma(t))\right],\quad x,y\in X.$$ The limit does not depend on the choice of the geodesic ray $\gamma$ representing $\xi$.  Let $k_1\in K, H_1\in S(\mathfrak a^{+})$. We define the function 
\begin{equation}\label{eq:BusemannFormula}
    \beta_{[k_1,H_1]}(gK)=-\langle H_1, H(k_1^{-1}g)\rangle.
\end{equation}
In some sources these function are called horofunctions \cite[8.14]{BH}. It is a priori not obvious how does the formula relates to the Busemann functions defined above. The proposition below clarifies that for $H_1\in \mathfrak a^{++}$.
\begin{proposition}\label{prop-DistanceEst} Let $k_1\in K, H_1\in S(\mathfrak a^{++}),t>0$. Let $\Omega\subset G$ be a compact set. Then for every $g\in\Omega$
    $$d(gK, k_1e^{tH_1}K)=-\langle H_1, H(k_1^{-1}g)\rangle+t+O_\Omega(e^{-c_{H_1}t}+t^{-1/2}),$$ where $c_{H_1}=\min_{\alpha\in \Sigma^+}\alpha(H_1).$ 
    In particular, $\lim_{t\to\infty} d(gK, k_1e^{tH_1}K)-t=\beta_{[k,H_1]}(gK)$ and the convergence is uniform on compact subsets of $X$.
\end{proposition}
\begin{proof}
    Let $(k_1^{-1}g)=n_2^{-1}e^{H_2}k_2$ be the Iwasawa decomposition of $k_1^{-1}g$, so that  $H_2=H(k_1^{-1}g)$. We have 
    \begin{align*}
        d(gK, k_1e^{tH_1}K)&=d(k_1^{-1}gK,e^{tH_1}K)=d(n_2^{-1}e^{H_2}K,e^{tH_1}K)\\
        &=d(e^{H_2-tH_1}K,e^{-tH_1}n_2e^{tH_1}K). 
    \end{align*}
    We have 
    \begin{align*}d(e^{H_2-tH_1}K,K)&=\|H_2-tH_1\|=t-\langle H_2,H_1\rangle+O(t^{-1/2}\|H_2\|)\\&=t-\langle H_2,H_1\rangle+O_\Omega(t^{-1/2}),\end{align*} as the length $\|H_2\|$ is bounded in terms of $\Omega$ alone. 

    Let us write $n_2=e^E$ where $E\in\mathfrak n$ and let $E=\sum_{\alpha\in \Sigma^+}E_\alpha$ be the decomposition of $E$ into eigenvectors of $A$. Then $e^{-tH_1}Ee^{tH_1}=\sum_{\alpha\in\Sigma^+}e^{-t\alpha(H_1)}E_\alpha.$
    It follows that $$\|e^{-tH_1}Ee^{tH_1}\|^2\leq e^{-2c_H}\sum_{\alpha\in\Sigma^+}\|E_\alpha\|^2=e^{-2c_H}\|E\|^2.$$ Vector $E$ is contained in a compact set determined by $\Omega$, so $d(e^{-tH_1}n_2e^{tH_1}K,K)=O_\Omega(e^{-c_Ht}).$ 
    By the triangle inequality 
    \begin{align*}
    |d(gK,k_1e^{tH_1}K)-t+\langle H_1,H_2\rangle|& \leq O_\Omega(t^{-1/2})+d(e^{-tH_1}n_2e^{tH_1}K,K)\\
    &=O_\Omega(t^{-1/2}+e^{-c_Ht}).\qedhere
    \end{align*}
\end{proof}
Let $\hat\rho$ be the element of $\mathfrak a$ characterized by the identity $\|\rho\|\langle \hat\rho,H\rangle=\rho(H).$
The Busemann functions $\beta_{[k,\hat\rho]}$ play a special role in the proof of Theorem \ref{thm-LieGpcorona} so we introduce a more convenient parametrization. Recall that the subgroup $U$ was defined as the kernel of the modular character $\chi_P(e^Hmn)=e^{2\rho(H)}, H\in\mathfrak a,m\in M, n\in N.$ For any $gU\in G/U$ define 
\begin{equation}\label{eq:SpecialOrbit}
    \beta_{gU}(sK):=-\langle \hat\rho, H(g^{-1}s)\rangle.
\end{equation}
Note that for any $t\in\mathbb R$ we have 
\begin{equation}\label{eq-ShiftScaling}\beta_{ke^{t\hat\rho}U}=\beta_{[k,\hat\rho]}+t.
\end{equation}

\subsection[Proof of Theorem 3.6]{Proof of Theorem \ref{thm-LieGpcorona}}
Let $v(t):=\vol(B(t)),$ be the volume of a ball of radius $t$. Using the formula for the integration in the spherical coordinates \eqref{eq-SphericalInt} we have 
$$v(t)=c_K\int_K\int_{\mathfrak a^+, \|X\|\leq t}J(X)dXdk_1.$$
The following asymptotic is well known to experts, we reproduce the computation because a similar one will be used in the proof of Theorem \ref{thm-LieGpcorona}.
\begin{lemma}\label{lem-volumeformula}  
  \begin{equation}
  v(t)\sim e^{2\|\rho\|t}t^{\frac{\rk G-1}{2}}.\end{equation}
\end{lemma}
\begin{proof}By \eqref{eq-SphericalInt} and the Weyl denominator formula $J(X)=\sum_{w\in W}\sgn(w)e^{2\rho^w(X)}$ \cite[Corollary 5.76]{BabyKnapp}, we have
\begin{align*}
v(t)&=c_K^2\int_{\substack{X\in \mathfrak a^+\\ \|X\|\leq t}}J(X)dX
\sim \sum_{w\in W}\sgn(w)\int_{\substack{X\in \mathfrak a^+\\ \|X\|\leq t}}e^{2\rho^w(X)}dX\\
&=\int_{\|X\|\leq t}e^{2\rho(X)}dX-\sum_{w\in W\setminus \{1\}}\int_{\substack{X\in \mathfrak a^+\\ \|X\|\leq t}}(1-\sgn(w))e^{2\rho^w(X)}dX.
\end{align*} 
We estimate the first integral. Put $m:=\frac{\rk G-1}{2}$. Using the substitution $X=s\hat\rho+Y,Y\in \ker\rho,$ we have
\begin{align*}
    \int_{\|X\|\leq t}e^{2\rho(X)}dX&=\int_{-t}^{t}\int_{\substack{Y\in \hat\rho^\perp\\ \|Y\|\leq \sqrt{t^2-s^2}}}e^{2\rho(s\hat\rho+Y)}dYds
    \sim \int_{-t}^t (t^2-s^2)^{m}e^{2\|\rho\|s}ds\\
    &=e^{2\|\rho\|t}\int_{-t}^t(t+s)^m(t-s)^me^{2\|\rho\|(s-t)}ds
    =e^{2\|\rho\|t}\int_{0}^{2t}(2t-u)^mu^me^{-2\|\rho\|u}du\\
    &\sim e^{2\|\rho\|t}(2t)^m\int_0^{2t}u^me^{-2\|\rho\|u}du
    \sim e^{2\|\rho\|t}t^m.
\end{align*}
The terms $$\sum_{w\in W\setminus \{1\}}\int_{\substack{X\in \mathfrak a^+\\ \|X\|\leq t}}(1-\sgn(w))e^{2\rho^w(X)}dX$$ can be estimated by $O(t^{\rk G}e^{2\|\rho\|M})$ where $$\|\rho\|M:=\max\{ \rho^w(X)\mid X\in \mathfrak a^{++}, \|X\|\leq 1, w\neq 1\}=\max\{ \rho(X)| X\in \mathfrak a\setminus \mathfrak a^+, \|X\|\leq 1\}.$$ Using the fact that the unique maximum of $\rho(X)$ on $S(\mathfrak a)$ is achieved at $\hat\rho\in\mathfrak a^{++}$ we can easily check that $M<1$. Therefore the first integral asymptotically dominates the sum and we get 
\[
v(t)\sim e^{2\|\rho\|t}t^{\frac{\rk G-1}{2}}.\qedhere
\]

\end{proof}
Recall $\eta(t):=v(t)^{-1}.$ We consider the sequence of $G$-invariant locally finite measures $\mu_t:=\eta(t)(\iota_t)_*(\vol)$, as in the definition of corona actions. Let $\varphi\colon D\to\mathbb R$ be a continuous compactly supported function. We have
\begin{equation}\label{def-mueta}
    \int_D \varphi(f)d\mu_t(f)=\eta(t)\int_{X}\varphi([d(x,\cdot)-t])d\vol(x).
\end{equation}
To prove \Cref{thm-LieGpcorona}, we will need to show that 
$$\lim_{t\to \infty} \int_{D}\varphi(f)d\mu_{t}(f)=c\int_{G/U} \varphi(\beta_{gU})dg,$$ for some positive constant $c$.
By (\ref{def-mueta}) and (\ref{eq-SphericalInt}), 
\begin{align*}\int_{D}\varphi(f)d\mu_{t}(f)&=c_K\eta(t)\int_K\int_{\mathfrak a^+}\varphi(d(-,k_1e^XK)-t)J(X)dXdk_1.\end{align*}
From now on we will work with a fixed $\varphi$. In order to find the limit we will show that as $t\to \infty$, the main contribution in the last integral comes from $X\in\mathfrak a^{+}$ which are relatively  close to $t\hat\rho.$ We recall that $\|\rho\|\langle X,\hat\rho\rangle=\rho(X).$

Define
\begin{align*}
    A &:=\sup\{|f(K)|\mid f\in\supp\varphi\},\\
    R &:=\{X\in\mathfrak a^+\mid t-A\leq \|X\|\leq t+A\},\\
    R^+ &:=\{X\in R\mid \langle\hat\rho,X\rangle\geq  t-(\log t)^2\},\\
    R^- &:=\{X\in R\mid \langle\hat\rho,X\rangle<  t-(\log t)^2\}.\\
\end{align*}
We note that $\supp\varphi\subset \{f\in D\mid |f(K)|\leq A\}.$ In particular 
\begin{align*}\int_{D}\varphi(f)d\mu_{t}(f)&=c_K\eta(t)\int_K\int_R\varphi(d(-,k_1e^XK)-t)J(X)dXdk_1\\
=\mu_t^+(\varphi)+\mu_t^-(\varphi),
\end{align*} where 
\begin{align*}
\mu_t^\pm(\varphi):=c_K\eta(t)\int_K\int_{R^\pm}\varphi(d(-,k_1e^XK)-t)J(X)dXdk_1.
\end{align*}
\begin{lemma}\label{lem-muetaminusvanishing}
    $\lim_{t\to\infty}\mu_t^-(\varphi)=0.$
\end{lemma}
\begin{proof}
    \begin{align*}
    |\mu_t^-(\varphi)| &\leq  c_K\eta(t)\int_{K}\int_{R^-}\max_{f \in D}|\varphi(f)|J(X)dXdk_1\\
&\leq c_K^2\eta(t) \max_{f \in D}|\varphi(f)|\int_{R^-}\prod_{\alpha\in\Sigma+}|e^{\alpha(X)}-e^{-\alpha(X)}|dX.
    \end{align*}
Since $R^- \subset \mathfrak a^+$, we can drop the absolute value in the product and estimate it using the Weyl denominator formula: $$J(X)=\prod_{\alpha\in\Sigma+}(e^{\alpha(X)}-e^{-\alpha(X)})=\sum_{w\in W}\sgn(w)e^{2\rho^w(X)}\leq |W|e^{2\rho(X)}.$$
Using the definition of $R^-$ and Lemma \ref{lem-volumeformula} we get 
\begin{align*}|\mu_t^-(\varphi)|&\leq\eta(t)c_K^2\max_{f\in D}|\varphi(f)||W| \int_{R^-}e^{2\|\rho\|(t-(\log t)^2)}dX\\ &\ll \eta(t) e^{2\|\rho\|(t-(\log t)^2)}t^{\rk G}\ll \frac{e^{2\|\rho\|(t-(\log t)^2)}t^{\rk G}}{e^{2\|\rho\|t}t^{(\rk G-1)/2}}\ll t^{-1}.\qedhere
\end{align*}
\end{proof}
\begin{lemma}\label{lem-JacobianApprox}
There exists $\kappa>0$, such that for every $X\in R^+$ and $t$ big enough we have
$$|J(X)-e^{2\rho(X)}|\ll e^{2\rho(X)-\kappa t}.$$
\end{lemma}
\begin{proof}
    By the definition of $R^+,$ for every $X\in R^+$ we have
    $$\|X\|\leq t+A \text{ and } \langle X,\hat\rho\rangle\geq t-(\log t)^2.$$
    Write $X=\hat\rho\langle X,\hat\rho\rangle+X', X'\in \hat\rho^\perp.$ Then,
    \begin{align*}
        \|X\|^2=\langle X,\hat\rho\rangle^2+\|X'\|^2&\leq t^2+2At+A^2\\
        t^2-2t(\log t)^2+(\log t)^4+\|X'\|^2&\leq t^2+2At+A^2\\
        \|X'\|^2&\leq 2t(\log t)^2+2At+A^2\\
        \|X'\|&\ll t^{1/2}\log t.
    \end{align*}
    It follows that \begin{align*}2\rho^w(X)=2\|\rho\|\langle X^w,\hat \rho\rangle&=2\|\rho\|\langle\hat\rho^w,\hat\rho\rangle \langle X,\hat\rho\rangle+2\|\rho\|\langle (X')^w,\hat \rho\rangle\\&=2\|\rho\|\langle\hat\rho^w,\hat\rho\rangle \langle X,\hat\rho\rangle+O(t^{1/2}\log t)
    \end{align*} for every $w\in W$ and $X\in R^+.$ Let $1-2\kappa_0:=\max_{w\neq 1}\langle \hat\rho^w,\hat\rho\rangle.$ Then $\kappa_0>0$ and for any $w\neq 1, X\in R^+$ and $t$ large enough, we have 
    $$2\rho^w(X)\leq 2\|\rho\|(1-2\kappa_0)\langle X,\hat\rho\rangle+O(t^{1/2}\log t)\leq 2\rho(X)-2\|\rho\|\kappa_0\langle X,\hat\rho\rangle.$$ Going back to $J(X)$, we use the Weyl denominator formula to get
    $$|J(X)-e^{2\rho(X)}|\leq \sum_{w\in W\setminus \{1\}}e^{2\rho(X)-2\|\rho\|\kappa_0\langle X,\hat \rho\rangle}\ll e^{2\rho(X)-\kappa t},$$ for $\kappa:=\kappa_0\|\rho\|.$
\end{proof}
\begin{lemma}\label{lem-Rplusangle}
    Let $U$ be an open neighborhood of $\hat\rho$ in $S(\mathfrak a)$. For every $X\in R^+$ and $t$ big enough we have $X\|X\|^{-1}\in U$.
\end{lemma}
\begin{proof}
    As in the proof of Lemma \ref{lem-JacobianApprox}, we let $X=\langle X,\hat\rho\rangle\hat\rho+X'$. The estimates $\|X'\|\ll t^{1/2}\log t, \|X\|\geq t-A$ yield
    $$X\|X\|^{-1}=\hat\rho+ O(t^{-1/2}\log t).$$ The lemma is proved. 
\end{proof}
\begin{lemma} \label{lem-FirstLimitApprox}
    $$\lim_{t\to \infty}\varphi(f)=\lim_{t\to \infty}c_K\eta(t)\int_K\int_{R^+}\varphi(\beta_{[k,\hat\rho]}+\|X\|-t)e^{2\rho(X)}dXdk_1.$$
\end{lemma}
\begin{proof}
    Let $\varepsilon>0$. Using Lemma \ref{lem-unifcontinuity} we choose a compact subset $\Omega\subset X$ and $\delta>0$ such that $|\varphi(f_1)-\varphi(f_2)|\leq \frac{\varepsilon}{2}$ for every $f_1,f_2\in D$ with $\sup_{x\in\Omega}|f_1(x)-f_2(x)|\leq \delta$ and an open neighborhood $U\subset S(\mathfrak a)$ of $\hat\rho$ satisfying the following properties:
    \begin{enumerate}
        \item the error term in Proposition \ref{prop-DistanceEst} satisfies 
        $O_\Omega(e^{-c_{H_1} t}+t^{-1/2})\leq \frac{\delta}{3}$ for $H_1\in U$ and big enough $t$,
        \item $|\langle H_1,H(k^{-1}g)\rangle-\langle\hat\rho, H(k^{-1}g)\rangle|\leq \frac{\delta}{3}$ for all $H_1\in U.$
    \end{enumerate}
 We are letting $t\to\infty$, so by \Cref{lem-Rplusangle} we can assume that $X\|X\|^{-1}\in U$ for every $X\in R^+$. By Lemma \ref{lem-muetaminusvanishing}, 
 \begin{equation}\label{eq-iotapprox0}\lim_{t\to \infty} \mu_t(\varphi)=\lim_{t\to\infty}\mu_t^+(\varphi)=\lim_{t\to \infty}c_K\eta \int_K \int_{R^+}\varphi(\iota_t(k_1e^XK))J(X)dXdk_1.\end{equation}
 By Proposition \ref{prop-DistanceEst} and the properties of $U$,
 $$\left|d(gK,k_1e^XK)-(-\langle X\|X\|^{-1}, H(k_1^{-1}g)\rangle+\|X\|)\right|\leq \frac{\delta}{3}.$$
 Therefore
 \begin{equation}\label{eq-iotaapprox1}
 |\iota_t(k_1e^XK)(gK)-(-\langle X\|X\|^{-1},H(k_1^{-1}g)\rangle+\|X\|-t)|\leq\frac{\delta}{3},
 \end{equation} for big enough $t$. The properties of $U$ also yield 
  \begin{equation}\label{eq-iotaapprox2}
  |\langle X\|X\|^{-1},H(k_1^{-1}g)\rangle-\langle \hat\rho, H(k_1^{-1}g)\rangle|\leq \frac{\delta}{3}.
  \end{equation}
Using \eqref{eq-iotaapprox1},\eqref{eq-iotaapprox2} and the triangle inequality we get:
$$|\iota_t(k_1e^XK)(gK)-(-\langle \hat\rho, H(k_1^{-1}g)\rangle+\|X\|-t)|\leq \delta,$$
for every $X\in R^+$ and $gK\in \Omega.$ Recall that $\beta_{[k_1,\hat\rho]}(gK)=-\langle \hat \rho, H(k_1^{-1}g)\rangle$ so the last estimate and the choice of $\delta,\Omega$ imply that 
$$|\varphi(\iota_t(k_1e^XK))-\varphi(\beta_{[k_1,\hat\rho]}+\|X\|-t)|\leq \varepsilon,$$ for $X\in R^+.$ Integrating this inequality we get 
\begin{equation}\label{eq-ioatapprox3}\begin{split}
    \left|\mu_t^+(\varphi)-c_K\eta(t) \int_K\int_{R^+}\varphi(\beta_{[k_1,\hat\rho]}+\|X\|-t)J(X)dXdk_1 \right|\\\leq c_K\eta(t)\int_K\int_{R^+}\varepsilon J(X)dXdk_1\leq\varepsilon.\end{split}
\end{equation}
Finally, by \Cref{lem-JacobianApprox}
\begin{equation}\label{eq-ioatapprox4}\begin{split}
    &\left|c_K\eta(t)\int_K\int_{R^+}\varphi(\beta_{[k_1,\hat\rho]}+\|X\|-t)(e^{2\rho(X)}-J(X))dXdk_1\right|\\ &\leq c_K\eta(t)\sup_{f\in D}|\varphi(f)|\int_K\int_{R^+}e^{2\rho(X)-\kappa t}dXdk_1\ll\sup_{f\in D}|\varphi(f)|e^{-\kappa t/2}.\end{split}
\end{equation}
We finish the proof by combining \eqref{eq-iotapprox0},\eqref{eq-ioatapprox3},\eqref{eq-ioatapprox4} and letting $\varepsilon\to 0$.
\end{proof}
\begin{proof}[Proof of \Cref{thm-LieGpcorona}]
Let $U$ be as in the proof of Lemma \ref{lem-FirstLimitApprox}. We have
$$R^+\subset \{tH\mid H\in U, t -A\leq t\leq t+A\}\subset R^+\cup R^-.$$
Since $\lim_{t\to \infty}\eta \int_K\int_{R^-}e^{2\rho(X)}dXdk_1=0,$ we have 
\begin{align*}\lim_{t\to\infty}& c_K\eta(t)\int_K\int_{R^+}\varphi(\beta_{[k,\hat\rho]}+\|X\|-t)e^{2\rho(X)}dXdk_1\\ & =\lim_{t\to\infty}c_K\eta(t)\int_K\int_{t-A}^{t+A}\int_{U}\varphi(\beta_{[k,\hat\rho]}+w-t)e^{2\rho(wH)}w^{\rk G-1}dHdw dk_1\\
& =c_K\int_K\int_{-A}^{A} \varphi(\beta_{[k_1,\hat\rho]}+s)\left(\lim_{t\to \infty}\eta(t) (t+s)^{\rk G-1}\int_Ue^{2\rho((t+s)H)}dH\right)dsdk_1,\end{align*} 
where $dH$ is a measure on $S(\mathfrak a)$ chosen so that $dX=w^{\rk G-1}dHdw$.
We evaluate the limit in the innermost integral. By Lemma \ref{lem-volumeformula} and the Laplace's method 
\begin{align*}\lim_{t\to\infty}\eta(t) (t+s)^{\rk G-1}\int_U e^{2\rho((t+s)H)}dH\sim &\lim_{t\to\infty} \frac{e^{2\|\rho\|(t+s)}(t+s)^{\frac{\rk G-1}{2}}}{e^{2\|\rho\|t}t^{\frac{\rk G-1}{2}}}=  e^{2\|\rho\|s}
\end{align*}
Plugging it back into the integral we get 
$$\lim_{t\to\infty}\mu_t(\varphi)=c\int_K\int_{-A}^{A}\varphi(\beta_{[k_1,\hat\rho]}+s)e^{2\|\rho\|s}ds dk_1,$$
for some $c>0$. Recall that $\beta_{[k_1,\hat\rho]}$ is defined so that $\beta_{[k_1,\hat\rho]}(K)=0.$ Since $\varphi(f)=0$ for any $f$ with $|f(K)|>A,$  we can extend the innermost integral to the entire real line without changing the value. 
\begin{equation}\label{eq-LimitMeasure1}\lim_{t\to\infty}\mu_t(\varphi)=c\int_K\int_{\mathbb R}\varphi(\beta_{[k_1,\hat\rho]}+s)e^{2\|\rho\|s}ds dk_1.\end{equation}
It remains to relate \eqref{eq-LimitMeasure1} to the integral $\int_{G/U}\varphi(\beta_{gU})dg.$

Let $g=ke^{X}n$ where $k\in K,X\in\mathfrak a, n\in N$. By (\ref{eq-ShiftScaling}), 
$\beta_{gU}=\beta_{[k,\hat\rho]}+\langle \hat\rho, X\rangle.$ To integrate over $G/U$ we can use the formula  \eqref{eq-IwasawaInt} and the decomposition $U=e^{\ker\rho}M N$
\begin{align*}\int_{G/U}\varphi(\beta_{gU})dg&=\int_K\int_{\mathfrak a/\ker\rho} \varphi(\beta_{ke^XU})e^{2\rho(X)}dXdk=\int_K\int_{\mathbb R}\varphi(\beta_{ke^{s\hat\rho}U})e^{2\|\rho\|s}dsdk\\
&= \int_K\int_{\mathbb R}\varphi(\beta_{[k,\hat\rho]}+s)e^{2\|\rho\|s}dsdk=c^{-1}\lim_{\eta\to 0}\mu_t(\varphi),\end{align*}
proving the theorem.
\end{proof}

Now we can verify that the assertions of Theorem \ref{thm-Treecoronas} hold when $G$ is a semisimple real Lie group.
\begin{proof}[Proof of real semisimple case of Theorem \ref{thm-Treecoronas}]
The group $U$ is a closed subgroup of $P$ which is amenable so it is amenable itself. By \cite[Thm 1.9]{Zimmer78} the action of $G$ on $G/U$ is amenable. The double recurrence follows from \ref{lem-UIntersections}, as it shows that the stabilizer of almost every point in $G/U\times G/U$ is unbounded. Finally, the fact that 
$\mu(gA\cap A)\to 0$ as $g\to\infty$ for every finite measure set $A\subset G/U$ follows from the Howe-Moore theorem \cite[Thm. 5.1]{HM79} and the fact that $L^2(G/U,dg/du)$ does not have vectors fixed by any simple factor of $G$.
\end{proof}

\section{Computations for products of trees}\label{sec-Treecoronas}
In this section we adapt the argument from the real semisimple Lie group case to prove Theorem \ref{thm-Treecoronas} for products of automorphism groups of trees. 
\subsection{Setup and notation}
Let $T_i$, $i=1,\ldots,m$ be regular trees of degrees $q_i+1$ respectively equipped with standard metrics $d_i$ where neighboring vertices are at distance $1$. Put $G_i:=\Aut(T_i)$ and $G=\prod_{i=1}^m G_i$. Let  $\partial T_i$ be the visual boundary of the tree $T_i$ \cite[p.427]{BH}. In each tree choose a root $o_i$ and a bi-infinite geodesic $\mathcal A_i$ with ends $\xi_i^+,\xi_i^-\in \partial T_i$, such that $o_i\in \mathcal A_i$. Let $K_i:=\Stab_{G_i}(o_i)$.
For each $i=1,\ldots,m$ choose automorphisms $\sigma_i\in G_i$, such that $\sigma_i$ fixes both ends of $\mathcal A_i$ and moves each vertex to the neighboring vertex closer to $\xi_i^+$. Let $w_i$ be a an element of $G_i$ that swaps the ends $\xi_i^+$ and $\xi_i^-$. Let 
$$ N_i:=\{n\in G_i\mid \lim_{s\to+\infty}\sigma_i^{-s}n\sigma_i^s=1\},$$
let $M_i$ be the pointwise stabilizer of $\mathcal A_i$ and let $P_i=\Stab_{G_i}\xi_i^+.$ We have a decomposition $P_i=\sigma_i^{\mathbb Z}M_iN_i.$ Both $M_i$ and $N_i$ are normalized by $\sigma_i$. We write $P=\prod_{i=1}^m P_i, K=\prod_{i=1}^m K_i, N=\prod_{i=1}^m N_i$ etc. We use the convention that a variable without index is an $m$-tuple, for example $g\in G$ stands for $(g_1,\ldots, g_m)$ and $o$ is the point $(o_1,\ldots, o_m)\in \prod_{i=1}^m T_i$. For $s\in \mathbb Z^m$ we write $\sigma^s:=(\sigma_1^{s_1},\ldots, \sigma_m^{s_m})$. 

Let $dg_i$ be the unique Haar measure on $G_i$ such that $K_i$ has measure $1$ and $dg=dg_1dg_2\ldots dg_m, dk:=dk_1dk_2\ldots dk_m$. Like in the case of the Lie groups, we have the Iwasawa decomposition and the functions $H_i\colon G_i\to \mathbb Z$ defined by $g_i=n_i \sigma_i^{H_i(g_i)}k_i,$ $n_i\in N_i, H_i(g_i)\in \mathbb Z$ and $k_i\in K_i$. We shall write $H(g):=(H_1(g_1),\ldots, H_m(g_m)), H\colon G\to \mathbb Z^m$. This function plays the role analogous to the central Iwasawa coordinate from the real semisimple Le group case, also denoted $H$. For each factor, the integration in Iwasawa coordinates is given by 
\begin{equation}\label{eq-TreeIwasawa}
    \int_{G_i}f(g_i)dg_i=\sum_{s_i\in\mathbb Z}\int_{N_i} \int_{K_i} f(k_i\sigma_i^{s_i}n_i)q_i^{s_i}dk_idn_i.
\end{equation}
Likewise, we have an analogue of the Cartan decomposition $g_i=k_i \sigma_i^{s_i}k_i',$  $k_i, k_i'\in K_i$ and $s_i\in \mathbb N$. The integration in spherical coordinates is given by 
\begin{equation}\label{eq-TreeSpherical}
    \int_{G_i}f(g_i)dg_1=\sum_{s_i=0}^\infty \int_{K_i}\int_{K_i} f(k_i\sigma_i^{s_i}k_i') J_i(s_i)dk_idk_i',
\end{equation}
where 
$$J_i(s_i)=|K_i/(K_i\cap \sigma^{s_i}K_i\sigma^{-s_i})|=\begin{cases} 1 & \text{ if } s_i=0\\ (q_i+1)q_i^{s_i} & \text{ if } s_i\geq 1.\end{cases}$$
For $s\in \mathbb Z$ we shall write $J(s)=\prod_{i=1}^m J_i(s_i).$
Finally, we have the Bruhat decomposition 
\begin{equation}\label{eq-Bruhat}
   G_i=P_i\sqcup P_i w_i P_i, 
\end{equation} which follows from the fact that $G_i$ acts transitively on the pairs of distinct ends. 

We shall use the standard induced Euclidean metric on  $\prod_{i=1}^m T_i$. Namely  $d(v,v')^2:= \sum_{i=1}^m d_i(v_i, v_i')^2$ for any vertices $v_i,v_i'\in T_i$.  We let $\langle\cdot,\cdot\rangle$ be the standard inner product on $\RR^m$ and $\|\cdot\|$ the corresponding Euclidean norm. We shall often use the fact that 
$d(\sigma^{s}o,o)=\|s\|.$

By analogy with the real semisimple case, we define a linear functional $\rho \colon \mathbb R^m\to\mathbb R$ corresponding to the half-sum of positive roots $$2\rho(s):=s_1\log q_1+\ldots +s_m\log q_m.$$ Then $2\|\rho\|=\sqrt{\sum_{i=1}^m \log q_i}.$ The normalized dual vector $\hat\rho$ is given by $\frac{1}{2}\|\rho\|^{-1}(\log q_1,\ldots, \log q_m).$ The modular character $\chi_P$ of $P$ is given by $\chi_P(\sigma^s l n)=2\rho(s), s\in \mathbb Z^m, l\in M, n\in N. $

We define a family of Busemann functions on $\prod_{i=1}^m T_i$ which will play the role of $\beta_{gU}$ from the real semisimple case. Unfortunately, if we take $U=\ker \chi_P$, the group does not have to satisfy the analogue of Lemma \ref{lem-UIntersections}. This has to do with the fact that the kernel of $\rho$ restricted to $\mathbb Z^m$ might of rank lower than $m-1$ when the ratios logarithms $\log q_i$ are irrational. For this reason we don't emphasize the $U$ cosets, and for any $g\in G$ define 
$$ \beta_{g}(ho)=-\langle \hat\rho,H(g^{-1}h)\rangle.$$
For future reference we make the observation that for any $s\in\mathbb Z^m$
\begin{equation}\label{eq-TBetaShift} \beta_{g\sigma^s}(ho)=\beta_{g}(ho)+\langle \hat\rho, s\rangle.\end{equation}

Let $v(t)$ be the volume of a (closed) ball of radius $t$ in $\prod_{i=1}^m T_i$. By a standard computation similar to \ref{lem-volumeformula} we have 
$v(t)\sim e^{2\|\rho\|t}t^{\frac{m-1}{2}}$. Let $\eta(t):=v(t)^{-1}.$ 
\subsection{Support of the limits}
Following the setup from Section \ref{sec-coronas}, we consider the sequence of measures $\mu_t$ on $D$ defined as
$$\int_D \varphi(f)\mu_t(f):=\eta(t)\int_{G}\varphi(\iota_t(go))dg.$$
\begin{proposition}\label{prop-coronaTrees}
    Any weak-* limit of $\mu_t$ is supported on the set $\{\beta_{g}+ t\mid g\in G, t\in \mathbb R\}.$
\end{proposition}
Unlike in the real semi-simple case, the support of the limit is not necessarily a single $G$ orbit. This has to with the fact that the set of values of $\beta_g(o)$, which is $\|\rho\|^{-1}\rho(\mathbb Z^m)$, can be a dense subgroup of $\mathbb R$ but it is never the whole $\mathbb R$. This is also the reason why we need to allow for the constant $t$. To prove Proposition \ref{prop-coronaTrees} we need to show that for any continuous compactly supported function $\varphi\colon D\to\mathbb R$ such that $\varphi(\beta_{g}+t)=0$ for every $g\in G, t\in \mathbb R,$ we have 
$$\lim_{t\to\infty} \int_D \varphi(f)d\mu_t(f)=0.$$ 
From now on we fix such a $\varphi$. 
Using (\ref{eq-TreeSpherical}), we get 
$$\int_D \varphi(f)d\mu_t(f)=\eta(t)\sum_{s\in \mathbb N^m}\int_{K}\varphi(\iota_t(k\sigma^{s}o))J(s)dk.$$

Define
\begin{align*}
    A &:= \sup\{|f(K)|\mid f\in\supp\varphi\},\\
    R &:= \{s\in \mathbb Z^m \mid t-A\leq \|s\|\leq t+A\},\\
    R^+ &:= \{s\in R\mid\langle \hat\rho,s\rangle \geq  (t-(\log t)^2)\},\\
    R^- &:= \{s\in R\mid\langle \hat\rho,s\rangle < (t-(\log t)^2)\}.\\
\end{align*}
We note that $\supp\varphi\subset \{f\in D\mid |f(K)|\leq A\}.$ Define 
$$\mu_t^{\pm}(\varphi)=\eta(t)\sum_{s\in R^{\pm}}\int_{K}\varphi(\iota_t(k\sigma^{s}o))J(s)dk.$$ Then $\int_D \varphi(f)d\mu_t=\mu_t^+(f)+\mu_t^-(f).$
\begin{lemma}\label{lem-TreeVanishingPart}
    $\lim_{t\to\infty} \mu_t^-(\varphi)=0.$
\end{lemma}
\begin{proof} Let $B:=\sup_{f\in D}|\varphi(f)|$. 
    \begin{align*}
    |\mu_t^-(\varphi)|&\leq \eta(t)B\sum_{s\in R^-}\int_K J(s)dk\\
    \ll & \eta(t)B\sum_{s\in R^-}e^{s_1\log q_1+\ldots+s_m\log q_m}=\eta(t)B\sum_{s\in R^-}e^{2\|\rho\|\langle\hat\rho,s\rangle}\\
    \leq & \eta(t)\sum_{\|s\|\leq t+A}e^{2\|\rho\|(t-(\log t)^2)}\\
    \ll & \eta(t)t^{m}e^{2\|\rho\|t-(\log t)^2)}\ll t^{-1}.\qedhere
    \end{align*}
\end{proof}

\begin{lemma}\label{lem-TreeBusemannApprox}
    For any $\varepsilon>0$ and a compact set $\Omega\subset\prod_{i=1}^m T_i$ there exists $t_0\in \mathbb R$ such that for any $t>t_0, s\in R^+, k\in K$ and $go\in \Omega$ we have $$\left|\iota_t(k\sigma^s)(go)-\left(\beta_{k}(go)+\|s\|-t\right)\right|\leq \varepsilon.$$
\end{lemma}
\begin{proof}
Let $go\in \Omega$. Write $u:=H(k^{-1}g)$ and $k^{-1}g=n\sigma^{u}k'$ for some $n\in N, k'\in K$. Note that the condition $go\in \Omega$ restricts $u\in \mathbb Z^m$ and $n\in N$ to compact sets. We have
\begin{align*}d(go,k\sigma^{s}o)&=d(k^{-1}go,\sigma^{s}o)
= d(n\sigma^{u}o,\sigma^{s}o)
= d(\sigma^{u}o, n^{-1}\sigma^{s}o)
= d(\sigma^{u-s}o, \sigma^{-s}n^{-1}\sigma^{s}o).
\end{align*}
Since $n=(n_1,\ldots, n_m)$ is restricted to a compact subset of $N$, it is easy to see that $\sigma_i^{-s_i}n_i^{-1}\sigma_i^{s_i}o_i=o_i$ will hold for all $i=1,\ldots,m$ if $s_i$ is big enough. Since $s\in R^+,$ this will be satisfied for big enough $t$. Therefore, we can assume that $d(go,k\sigma^{s}o)=d(\sigma^{u-s}o,o).$
We have 
\begin{align*}
d((\sigma^{u-s}o,o)&=\|u-s\|
= \|s\|-\|s\|^{-1}\langle s,u\rangle + O\left(\frac{\|u\|^2}{\|s\|}\right).\end{align*}
The condition $s\in R^+$ forces $\|s\|^{-1}s=\hat\rho+ O(t^{-1/2}\log t)$ so 
\begin{align*}
    d(go,k\sigma^{s}o)&= \|s\|-\langle \hat\rho,u\rangle + O\left(t^{-1/2}\log t\right)
    = \|s\|+ \beta_{k}(go) +O\left(t^{-1/2}\log t\right).\qedhere
\end{align*}\end{proof}
We are now ready to prove Proposition \ref{prop-coronaTrees}.
\begin{proof}[Proof of Proposition \ref{prop-coronaTrees}]
As mentioned before, we need to show that for every compactly supported continuous function $\varphi\colon D\to\mathbb R$ which vanishes on the set $\{\beta_{g}+ t\mid t\in \mathbb R\}$ we have $\lim_{t\to\infty}\int_D \varphi(f)d\mu_t(f)=0.$
    By Lemma \ref{lem-unifcontinuity}, for any $\delta>0$ there exists an $\varepsilon>0$ and a compact set $\Omega\subset \prod_{i=1}^m T_i$ such that $|\varphi(f_1)-\varphi(f_2)|\leq \varepsilon$ whenever $$\max_{go\in \Omega} |f_1(go)-f_2(go)|\leq \delta.$$ In particular, by Lemma \ref{lem-TreeBusemannApprox}, for $t$ large enough and $s\in R^+$ we will have 
    $$|\varphi(\iota_t(k\sigma^{s}o))|=|\varphi(\iota_t(k\sigma^{s}o))-\varphi\left(\beta_{k}+\|s\|-t\right)|\leq \varepsilon.$$
    By Lemma \ref{lem-TreeVanishingPart},
    $$\lim_{t\to\infty}\left|\int_D \varphi(f)d\mu_t(f)\right|=\lim_{t\to\infty}|\mu_t^+(\varphi)|\leq \lim_{t\to\infty}\eta(t)\int_K\sum_{s\in R^+}|\varphi(\iota_t(k\sigma^{s}o))|dk\leq\lim_{t\to\infty}\eta(t)v(t)\varepsilon=\varepsilon.$$
    We prove the proposition by letting $\varepsilon\to 0$
\end{proof}

\subsection[Proof of Theorem 3.7]{Proof of Theorem \ref{thm-Treecoronas}}

\begin{lemma}\label{lem-TreeTopAction}
    Let $G/P\rtimes_{\hat\rho} \mathbb R$ be the skew product action of $G$ defined by $$g(hP, t):=(ghP,t+\langle \hat\rho, H(h^{-1}g)-H(h^{-1})\rangle).$$The map $\kappa\colon \{\beta_{g}+ t\mid g\in G, t\in \mathbb R\}\to G/P\rtimes_{\hat\rho} \mathbb R$ given by $$\kappa(\beta_{g}+t)=(gP,\beta_{g}(o)+t)$$ is a well defined $G$-equivariant homeomorphism. 
\end{lemma}

\begin{proof}
 The stabilizer of $\beta_{g}+t$ is precisely the group $gUg^{-1},$ where $U\subset P$ is the kernel of the modular character $\chi_P$. Therefore, the map $\beta_{g}+t\mapsto gP$ is well defined. Verification that $\kappa$ is $G$ equivariant is a straightforward computation using the definition of $\beta_g$. Write $g=k(g) \sigma^{H(g)}n, k(g)\in K$. The inverse map is given by $\kappa^{-1}(gP,t)=\beta_{k(g)}+t.$ Since $g\mapsto k(g)$ is continuous, we showed that $\kappa$ is a homeomorphism. 
\end{proof}

\begin{lemma}\label{lem-InvMeasures}
    Suppose that the group $A:=\{\langle \hat\rho, s\rangle\mid s\in\mathbb Z^m\}\subset \mathbb R$ is dense. Then, up to scaling, the only $G$ invariant measure on $G/P\rtimes_{\hat\rho} \mathbb R$ is $\nu\times e^{2\|\rho\|t}dt$ where $\nu$ is the unique $K$-invariant probability measure on $G/P$.
\end{lemma}

\begin{proof}
    Let $\mu$ be a $G$-invariant locally finite measure on $G/P\times_{\hat\rho} \mathbb R$. The action of $K$ is transitive on the first factor and leaves the second invariant. It follows that $\mu=\nu\times \nu'$ where $\nu'$ is some locally finite measure on $\mathbb R$. The measure $\nu$ on $G/P$ is quasi invariant under $G$ and by (\ref{eq-TreeIwasawa}) the Radon-Nikodym derivative given by
    $$\frac{dg_*\nu}{d\nu}(hP)=e^{-2\rho(H(h^{-1}g)-H(h^{-1}))}.$$
    The Radon-Nikodym derivative of $\mu$ is then expressed by 
    $$1=\frac{dg_*\mu}{d\mu}(hP,t)=e^{-2\rho(H(h^{-1}g)-H(h^{-1}))} \frac{d \langle \hat \rho, H(h^{-1}g)-H(h^{-1})\rangle_*\nu'}{d\nu'}.$$
     From there we deduce there for any $u\in A,$ the measure $\nu'$ transforms according to $\frac{d u_* \nu'}{d\nu'}=e^{2\|\rho\|u}.$ Since $A$ is dense, this property extends to every $u\in \mathbb R$. Up to scalar, the only measure that satisfies this equation is $\nu'=e^{2\|\rho\| t}dt.$
\end{proof}

\begin{lemma}\label{lem-DenseTranslates}
    Let $A\subset \mathbb R$ be a dense subgroup and let $W\subset \mathbb R^2$ be a measurable subset. For Lebesgue almost every $(t,t')\in W$ the set $\{u\in A\mid (t+u,t'-u)\in W\}$ is infinite. 
\end{lemma}

\begin{proof}
    Let $u_n\in A$ be sequence tending to zero. Put $W_n:=\{(t,t')\in W\mid (t+u,t'-u)\in W\}$. Since the action of $\mathbb R^2$ on $L^2(\mathbb R^2)$ by translations is continuous (see Appendix A.6 of \cite{kazhdanbook}), we deduce that $\lim_{n\to\infty} {\rm Leb}(W\cap (W+(u_n,-u_n)))=\lim_{n\to\infty}{\rm Leb}(W_n)={\rm Leb}(W).$ The Lemma follows now from the Fatou lemma. 
\end{proof}

\begin{lemma}\label{lem-DoubleRecurrence}
    For $m\geq 2$, the action of $G$ on $(G/P\rtimes_{\hat\rho} \mathbb R, \nu\times e^{2\|\rho\|t}dt)$ is doubly recurrent.
\end{lemma}

\begin{proof}
Let $E\subset (G/P\rtimes_{\hat\rho} \mathbb R)^2$ be a positive measure subset. For $hP,h'P\in G/P$ write $$E_{hP,h'P}:=\{(t,t')\in\mathbb R^2\mid (hP,t,h'P,t')\in E\}.$$
Fubini's theorem immediately implies that for almost all $(hP,t,h'P,t')\in E$ the set $E_{hP,h'P}$ has positive measure. Hence, by Lemma \ref{lem-DenseTranslates}, for almost every $(hP,t,h'P,t')\in E$ there is infinitely many $u\in A,$ such that $(hP,t+u,h'P,t'-u)\in E.$  
By Bruhat's decomposition (\ref{eq-Bruhat}), for almost every point $(hP,t,h'P,t')\in E$ there exists $g\in G$ such that $hP=gP$ and $h'P=gwP$. Let $E_0$ be the set of points $(hP,t,h'P,t')\in E$ which satisfy the three above properties at the same time, i.e. $E_{hP,h'P}$ has positive measure,  $(hP,t+u,h'P,t'-u)\in E$ for infinitely many $u\in A$ and there exists $g\in G$ such that $hP=gP$ and $h'P=gwP$, with $w:=(w_1,\ldots,w_m)$. It is a full measure subset of $E$. Now let $(hP,t,h'P,t')\in E_0$ and choose $g\in G$ such that $(hP,t,h'P,t')=(gP,t,gwP,t')$. Then, for any $s\in \mathbb Z^m$ we have
$$g\sigma^s g^{-1}(gP,t,gwP,t')=(gP, t+\langle \hat\rho, s\rangle, gwP, t+\langle \hat\rho, wsw^{-1}\rangle)=(gP, t+\langle \hat\rho, s\rangle, gwP, t-\langle \hat\rho, s\rangle).$$
Since $\langle \hat\rho, s\rangle$ can take any value in  $A,$ we deduce that $g\sigma^s g^{-1}(gP,t,gwP,t')\in E$ for infinitely many $s\in\mathbb Z^m$. Since $E_0$ was full measure, this implies that the action of $G$ on $(G/P\rtimes_{\hat\rho} \mathbb R)^2$ is conservative.
\end{proof}

\begin{proposition}\label{prop-ThreePropertiesTrees}
    Any locally finite $G$-invariant measure $\mu$ supported on the set $\{\beta_{g}+ t\mid t\in \mathbb R\}\subset D$ has the following properties:
    \begin{enumerate}
        \item The diagonal action of $G$ on $(D\times D,\mu\times \mu)$ is conservative.
        \item The action of $G$ on $(D,\mu)$ is amenable.  
        \item For any measurable set $A\subset D$ with $\mu(A)<\infty$ we have $\lim_{g\to\infty}\mu(A\cap gA)=0.$
    \end{enumerate}
    \begin{proof}

            (1) By Lemma \ref{lem-TreeTopAction} and Proposition \ref{prop-coronaTrees} it is enough to show that $(G/P\rtimes_{\hat\rho} \mathbb R, \mu)^2$ is conservative for any $G$-invariant locally finite measure $\mu$. We consider two cases, depending on whether the group $A:=\{\langle \hat\rho, s\rangle\mid s\in\mathbb Z^m\}$ is a discrete or a dense subgroup of $\mathbb R$. The first case corresponds to the situation when the group $\Lambda:=\{s\in \mathbb Z^m\mid \rho(s)=0\}$ is of rank $m-1$. In this case, we argue that almost every point  $(hP,t,h'P,t')$ is stabilized by an unbounded subgroup of $G$. This will of course imply the desired double recurrence. As in the proof of Lemma \ref{lem-DoubleRecurrence}, for $\mu$-almost every $(hP,t,h'P,t')$ there exists $g\in G$ such that $hP=gP$ and $h'P=gwP$. Then, the point $(gP,t,gwP,t')$ is stabilized by $g\sigma^{s}g^{-1}, s\in \Lambda$, which is an unbounded subgroup of $G$. 
            In the second case, when $\Upsilon$ is dense, double recurrence follows from Lemmas \ref{lem-TreeTopAction}, \ref{lem-InvMeasures} and \ref{lem-DoubleRecurrence}.
            
            (2) By Lemma  \ref{lem-TreeTopAction} we have a $G$-equivariant projection map from  $(D,\mu)$ to $G/P\simeq \prod_{i=1}^m \partial T_i$ 
            which is well known to be amenable \cite{Adams94}. By \cite[Theorem 1.9]{Zimmer78} the action $(D,\mu)$ must be itself amenable.
            
            (3) By \cite[Thm 1.1]{Ciobotaru}, the groups $G_i$ have the Howe-Moore property, meaning that the matrix coefficients of unitary representations with no fixed vectors vanish at infinity. If follows that either $\lim_{g\to\infty} \mu(A\cap gA)=0$ or $L^2(D,\mu)$ has a non-zero vector fixed by $G_i$ for some $i=1,\ldots, m$. By \cite[Thm. 4.3.1]{AnanDel2003} and the second point of this proposition, the representation $L^2(D,\mu)$ is weakly contained in $L^2(G).$ Since $L^2(G)$ has no almost $G_i$-invariant vectors for any $i=1,\ldots, m,$ we deduce that $L^2(D,\mu)$ has no $G_i$-invariant vectors $i=1,\ldots,m$.
    \end{proof}
    
    \begin{proof}[Proof of Theorem \ref{thm-Treecoronas}]
    By Proposition \ref{prop-coronaTrees}, for any corona action $(D,\mu)$, the measure $\mu$ is supported on the set $\{\beta_g+t\mid g\in G, t\in\mathbb R\}.$ By Proposition \ref{prop-ThreePropertiesTrees} any such measure satisfies the conclusions of Theorem \ref{thm-Treecoronas}. 
    \end{proof}
\end{proposition}

\section{Borders between IPVT cells}\label{sec-cellpairs}

Here we prove Theorem \ref{cellpairs}. Our goal is to prove, for any pair of distinct points in a Poisson point process on a corona $(D,\mu)$ for $G$, that the shared border between the associated IPVT cells is unbounded with probability $1$. In fact, we prove a slightly stronger statement used in the eventual proof of Theorem \ref{fixedpriceG}. Let $\upsilon\in\mathbb M(D)$ be admissible (see Definition \ref{def-GeneralizedVoronoi}), $r>0$, and let $f_1,f_2\in \upsilon$. Set
$$
    W_r^{\upsilon}(f_1,f_2):=\{g\in G \mid f_1(g)=f_2(g) \text{ and } f(g)> f_1(g)+r \text{ for every } f\in \upsilon\setminus\{f_1,f_2\}\}.
$$
This is the \textit{$r$-wall} of $\upsilon$ with respect to the pair $(f_1,f_2)$. This set is always right $K$-invariant so we may think of it as a subset of $X=G/K$. Before proving Theorem \ref{cellpairs}, we prove Theorem \ref{thm-cellpairsthick}.

\begin{theorem}[Unbounded IPVT cell walls]\label{thm-cellpairsthick}
Let $G\curvearrowright(D,\mu)$ be as in Theorem \ref{thm-LieGpcorona}. A Poisson point process $\Upsilon$ on $(D,\mu)$ has the following properties $\mathcal L(\Upsilon)$-almost surely. For all $f_1,f_2\in \Upsilon,$ the set  $W_r^{\Upsilon}(f_1,f_2)$ is unbounded. In particular the border between any pair of cells $C_{f_1}^\Upsilon$ and $C_{f_2}^\Upsilon$ in the IPVT tessellation $\Vor(\Upsilon)$ is unbounded.
\end{theorem}

Theorem \ref{cellpairs} immediately follows. Any $g\in W_r^{\Upsilon}(f_1,f_2)$ belongs to the border between $C_{f_1}^{\Upsilon}$ and $C_{f_2}^{\Upsilon}$, and its $r/2$-neighborhood witnesses precisely the cells $C_{f_1}^{\Upsilon}$ and $C_{f_2}^{\Upsilon}$. The latter observation follows from the fact that $f_1,f_2$ are $1$-Lipschitz. Indeed, for $g\in W_r^{\Upsilon}(f_1,f_2)$ and $g'\in G$ such that $d(g,g')\leq r/2$, we have 
$$
    f(g')\geq f(g)-r/2> f(g)+r/2\geq f_i(g')
$$ 
for $i=1,2$ and $f\in \Upsilon\setminus\{f_2,f_2\}.$ So $g'$ is in either $C_{f_1}^{\Upsilon}$ or $C_{f_2}^{\Upsilon}$.

\begin{proof}[Proof of Theorem \ref{thm-cellpairsthick}]
Let $\upsilon\in\mathbb M(D)$ and $f_1,f_2\in D$, and define $F:\mathbb M(D)\times D^2\rightarrow \{0,1\}$ so that
$$
     F(\upsilon,f_1,f_2):= \begin{cases}
         0 & \text{if } f_1,f_2\in \upsilon \text{ and } W_r^{\upsilon}(f_1,f_2) \text{ is unbounded}\\
         1 & \text{otherwise.}
     \end{cases}
$$
The bivariate Mecke equation (Theorem \ref{2Mecke}) implies
$$\mathbb E\left[\sum_{f_1,f_2\in \Upsilon} F(\Upsilon,f_1,f_2)\right]=\int_D \int_D \mathbb E\left[ F(\Upsilon\cup\{f_1,f_2\},f_1,f_2)\right]d\mu(f_1)d\mu(f_2).$$ 
Proposition \ref{prop-R-thick-walls}, below, implies for $\mu\times\mu$-almost every $(f_1,f_2)\in D^2$ the $r$-wall $W^{\Upsilon\cup\{f_1,f_2\}}_r(f_1,f_2)$ is unbounded $\mathcal L(\Upsilon)$-almost surely. Then the right hand side in the above equation is $0$. Thus $W_r^{\Upsilon}(f_1,f_2)$ is $\mathcal L(\Upsilon)$-almost surely unbounded for all $f_1,f_2\in \Upsilon$. 
\end{proof}

\begin{proposition}\label{prop-R-thick-walls}
    Let $G$ be a higher rank semisimple real Lie group, and let $\Upsilon$ be a Poisson point process on $D$ with intensity $\mu$. Fix $r>0$. For $\mu\times\mu$-almost every $(f_1,f_2)\in D^2$ the associated $r$-wall $W^{\Upsilon\cup\{f_1,f_2\}}_r(f_1,f_2)$ is unbounded $\mathcal L(\Upsilon)$-almost.
\end{proposition}

\begin{proof}
    By Theorem \ref{thm-LieGpcorona}, the measure $\mu$ is the push-forward of the $G$-invariant measure $dg/du$ on $G/U$ via the map $gU\mapsto \beta_{gU}$. By Lemma \ref{lem-UIntersections}, for almost every $g_1U,g_2U\in G/U,$ the stabilizer of the pair $g_1U,g_2U$ is an unbounded subgroup of $G$. The map $gU\to \beta_{gU}$ is $G$-equivariant, so it follows that the stabilizer of the pair $f_1,f_2$ is unbounded for $\mu\times\mu$-almost every pair $f_1,f_2$.

    Since $G/K$ is path connected, we may choose a point $g_0\in G$ such that $f_1(g_0)=f_2(g_0)$. Set $A:=f_1(g_0)+r=f_2(g_0)+r$ and $B:=\{f\in D\mid f(g_0)\leq A\}.$ By Corollary \ref{cor-coronasExistence}, we have $\mu(B)<\infty.$ Let $S<G$ be the stabilizer of the pair $f_1,f_2$. Since 
    $$f_1(sg_0)=f_2(sg_0)=f_1(g_0)=f_2(g_0)$$ 
    for every $s\in S$, we have $sg_0\in W^{\Upsilon\cup\{f_1,f_2\}}_r(f_1,f_2)$ if and only if $f(sg_0)> A$ for every $f\in\Upsilon$ (note $\Upsilon$ is disjoint from $\{f_1,f_2\}$ with probability $1$ because $\mu$ is atomless). The last condition can be restated as $\Upsilon(sB)=0$. We proved that 
    $$\{sg_0\mid \Upsilon(sB)=0\}\subset W_r^{\Upsilon\cup\{f_1,f_2\}}(f_1,f_2).$$
    It remains to show that the set $\{s\mid \Upsilon(sB)=0\}$ is unbounded $\mathcal L(\Upsilon)$-almost surely. 

    By Theorem \ref{thm-Treecoronas}, we have $\lim_{s\to\infty} \mu(B\cap sB)=0.$ Since $S$ is unbounded, we may inductively choose a uniformly separated sequence $s_i\in S, i\in\mathbb N$, such that 
    $$e^{\mu(s_iB\cap s_jB)}\leq (1+2^{-i})(1+2^{-j}), \text{ for all }i\neq j.$$
    Let $D_i$ be the event that $\Upsilon(s_iB)=0.$ 
    Then $\mathbb P(D_i)=e^{-\mu(s_iB)}=e^{-\mu(B)}$ and 
    $$\mathbb P(D_i\cap D_j)=e^{-\mu(s_iB\cup s_jB)}=e^{-2\mu(B)}e^{\mu(s_iB\cap s_jB)}\leq e^{-2\mu(B)}(1+2^{-i})(1+2^{-j})\text{ for all }i\neq j.$$
    To apply the Kochen-Stone theorem (see Lemma \ref{KS}), a generalization of Borel-Cantelli for almost independent events, consider that we have
    $$\limsup_{n\to\infty} \frac{\left(\sum_{i=1}^n \mathbb P(D_i)\right)^2}{\sum_{i,j=1}^n \mathbb P(D_i\cap D_j)}\geq \limsup_{n\to\infty} \frac{n^2e^{-2\mu(B)}}{ne^{-\mu(B)}+\sum_{i\neq j=1}^n e^{-2\mu(B)}(1+2^{-i})(1+2^{-j})}=\lim_{n\to\infty}\frac{n^2}{(n+1)^2}=1.$$
    Kochen-Stone implies infinitely many events $D_i$ do occur $\mathcal L(\Upsilon)$-almost surely, so the set $\{s\mid \Upsilon(sB)=0\}$ is unbounded $\mathcal L(\Upsilon)$-almost surely.
\end{proof}

\begin{lemma}[Kochen-Stone \cite{KS1964}]\label{KS}
    Let $A_i$ be a sequence of events such that $\sum_{i=1}^\infty \mathbb P(A_i)=+\infty$. Then the probability that infinitely many $A_i$ occur is at least 
    $$\limsup_{n\to\infty} \frac{\left(\sum_{i=1}^n \mathbb P(A_i)\right)^2}{\sum_{i,j=1}^n \mathbb P(A_i\cap A_j)}.$$
\end{lemma}

Versions of Theorem \ref{thm-cellpairsthick} and Proposition \ref{prop-R-thick-walls} can be proved when $G$ is the automorphism group for products of trees. In the proof of Proposition \ref{prop-R-thick-walls} one just has to replace the stabilizer $S$ by an appropriate set of returns to a small neighbourhood of the pair $f_1,f_2$.

\section{Cost 1 in higher rank}\label{sec-CostOneProof}

In this section we prove an abstract condition for the fixed price $1$ property. From it we shall deduce Theorem \ref{fixedpriceG}.
\begin{theorem}\label{thm-AbstractCostOne}
    Let $G$ be a unimodular lcsc group with a corona action $(D,\mu)$ satisfying the following properties:
    \begin{enumerate}
        \item the action $G\curvearrowright (D,\mu)$ is amenable,
        \item the action $G\curvearrowright (D,\mu)$ is doubly recurrent.
    \end{enumerate}
    Then $G$ has fixed price one.
\end{theorem}

We will need several preliminary results. We give an outline of the proof of Theorem \ref{thm-AbstractCostOne} to justify these results and also to introduce the notation. A precise argument proving Theorem \ref{thm-AbstractCostOne} as well as the proof of Theorem \ref{fixedpriceG} will follow in section \ref{sec-AbstCostOneProof}.
\begin{itemize}

\item\textbf{Reduction to verifying that the Palm equivalence relation $\mathcal R$ of $\Pi\times \Upsilon$ has cost one.} By Corollary \ref{cor-FPOneReduction} the Theorem \ref{thm-AbstractCostOne} will follow once we establish that the Poisson point process on $G$ has cost one. Using Theorem \ref{WeakLimitTheorem} we further reduce the problem to cost one for the product marking $\Pi\times \Upsilon$, where $\Pi$ is an IID marked Poisson point process on $G$ and $\Upsilon$ is the IID marking (see Example \ref{ex-iidlabeling}) of $\Poisson(D,\mu)$ - the Poisson point process on $D$ with mean measure $\mu$. Write $Z=\MM(D,[0,1])$ for the space of configurations on $D$ marked with labels from $[0,1]$. The marked point process $\Pi\times \Upsilon\in \MM(G,[0,1]\times Z)$ has cost one if and only if its Palm equivalence relation, denoted $\mathcal R$, has cost one. Therefore, the proof of Theorem \ref{thm-AbstractCostOne} is reduced to showing that $\mathcal R$ has cost one. This part of the proof does not use the conditions (1) or (2).

\item \textbf{Construction of a sub-relation $\mathcal S\subset \mathcal R$ using Voronoi cells}. An equivalence class of the Palm equivalence relation of $\Pi \times\Upsilon$ is in bijection with the configuration underlying $\Pi_\circ$ (see (\ref{eq-PalmClass})). The points in $\Upsilon$ give rise to a generalized Voronoi tessellation $\Vor(\Upsilon)$ which splits $G$ and consequently $\Pi$ into a countable collection of cells $\Pi\cap C_f^\Upsilon$, indexed by $f\in \Upsilon$. We use the marking by $\Upsilon$ to split the equivalence classes of $\mathcal R$ into sub-classes according to the cells they belong to. This defines a sub-relation $\mathcal S\subset \mathcal R$. This part of the proof does not use the conditions (1) or (2). We do this in Section \ref{sec-Subrelation}.

\item\textbf{Showing that $\mathcal S$ is hyperfinite}. To show that $\mathcal S$ is hyperfinite we construct an explicit hyperfinite quasi-p.m.p countable equivalence relation $\mathcal S'$ on $(\MM(G)\times D,\mathcal L(\Pi)\times \mu)$ which is a class injective\footnote{A measure preserving map between quasi-pmp equivalence relations $\psi\colon (X_1,\mu_1,\mathcal O_1)\to (X_2,\mu_2,\mathcal O_2)$ is class injective factor if it induces injection from the equivalence class $[x]_{\mathcal S_1}$ to $[\psi(x)]_{\mathcal S_2}$ for almost every $x\in X_1$.} factor of the relation $\mathcal S$. The hyperfiniteness of $\mathcal S$ can be then lifted from $\mathcal S'$. The construction crucially relies on the fact that the action of $G$ on $(D,\mu)$ is amenable but it doesn't use (2). We do this in Section \ref{sec-Subrelation}.

\item \textbf{Construction of a low cost graphing by connecting the cells}. Since $\mathcal S$ is hyperfinite, it admits a generating graphing of cost one. To complete it to a generating graphing of $\mathcal R$ it is enough to add a graphing $\mathscr G$ with connects every pair of cells. In the case of symmetric spaces Theorem \ref{cellpairs} guarantees this can be done by connecting pairs of points at distance less than certain constant $C$ independently randomly with some small probability. 
The general proof follows the same idea but uses only the double recurrence condition (2), bypassing Theorem \ref{cellpairs}. This is the only step where we used condition (2). We do this in Section \ref{sec-Graphing}.
\end{itemize}

\subsection{Restricted sub-relation}\label{sec-Subrelation}
Implicit in the definition of the corona action is an isometric transitive action of $G$ on a lcsc metric space $(X,d)$ and a point $o\in X$ such that $K=\Stab_G o$ is a compact subgroup of $G$. We will freely identify subsets of $X$ and functions on $X$ with right $K$-invariant subsets of $G$ and right $K$-invariant functions on $G$ respectively. Let $\upsilon\in Z:=\MM(D,[0,1])$ be a point configuration on $D$ marked by labels from $[0,1]$. Let $\ell(f)$ be the label of $f\in \upsilon.$ We will reserve letter $\Upsilon$ for the IID Poisson point process on $D$ with mean measure $\mu$.
\begin{definition} Let $\upsilon\in\MM(D,[0,1])$ be admissible (see Definition \ref{def-GeneralizedVoronoi}). The \emph{ tie-breaking Voronoi tessellation} $\Vor(\upsilon)$ is defined as 
$$\Vor(\upsilon)=\{C_f^\upsilon\mid f\in \upsilon\},\quad C_f^\upsilon:=\{x\in X\mid f(x)\leq h(x) \text{ or } f(x)=h(x) \text{ and } \ell(f)<\ell(h) \text{ for all } h\in\upsilon\}.$$
\end{definition}
For an arbitrary configuration $\upsilon$, the tessellation might display pathological properties but when $\Upsilon$ is an IID marking of $\Poisson(D,\mu)$ we can guarantee that is admissible and it splits $X$ into countably many pairwise disjoint cells. The tie-breaking mechanism is added to avoid the situation where the points of $\Pi$ lie on the boundary of the Voronoi tessellation, which could happen in the ``ordinary'' Voronoi tessellation if the space $X$ is countable.

\begin{lemma}\label{lem-TBVornoi}
The tessellation $\Vor(\Upsilon)$ is a disjoint countable cover of $X$ almost surely. 
\end{lemma}
\begin{proof}
    The following statements hold almost surely in $\Upsilon$. The labels of elements in $\Upsilon$ are pairwise disjoint, so the cells $C_f^\Upsilon$ are pairwise disjoint. We prove that $\Vor(\Upsilon)$ covers the space. By Corollary \ref{cor-coronasExistence}, for every $r>0$ the set $\{f\in D\mid f(o)\leq r\}=D_{\leq r}\cup \{f\in D\mid f(o)\leq 0\}$ has finite measure. It follows that the set $\{f\in \Upsilon\mid f(o)\leq r\}$ is finite for every $r>0$. In particular, for every point $x\in X$ the set of values $\{f(x)\mid f\in \Upsilon\}$ is discrete and bounded from below. Let $x\in X$ and let $f_1,\ldots, f_m$ be the functions in $\Upsilon$ which realize the minimal value at $x$. The cell corresponding to $f_i$ with the minimal label $\ell(f_i)$ is the one that contains $x$. 
\end{proof}
 
By Lemma \ref{lem-PalmProductMarking}, the Palm version of product marking $\Pi\times \Upsilon\in \MM(G,Z)$ is given by $\Pi_\circ\times \Upsilon$ and the Palm equivalence relation $\mathcal R$ on $\MMo(G,Z)$ is given by
$$\omega\times\upsilon\sim_{\mathcal R}\omega'\times\upsilon' \text{ iff } \exists g\in\omega \text{ such that } \omega'=g^{-1}\omega \text{ and } \upsilon'=g^{-1}\upsilon,$$ for $\omega\in\MM(G)$ and $\upsilon\in Z$. We will define a sub-relation $\mathcal S\subset \mathcal R$ by adding the condition that $1$ and $g$ should be in the same cell of $\Vor(\upsilon)$ (see Figure \ref{fig:RestrictedRerooting}). 
Define $c(\upsilon)$ as the unique point $f\in \upsilon$ such that $1\in C_f^\upsilon$, if such point exists. By Lemma \ref{lem-TBVornoi}, $c\colon Z\to D$ is $\mathcal L(\Upsilon)$-almost everywhere well defined measurable map. The sub-relation $\mathcal S\subset \mathcal R$ is defined as 
\begin{equation}\label{eq-DefS}
 \omega\times\upsilon\sim_{\mathcal S}\omega'\times\upsilon' \text{ iff } \exists g\in\omega\cap C_{c(\upsilon)}^\upsilon \text{ such that } \omega'=g^{-1}\omega \text{ and } \upsilon'=g^{-1}\upsilon.
\end{equation}
The key feature of $\mathcal S$ that is essential to the proof of Theorem \ref{thm-AbstractCostOne} is that the equivalence classes of $\mathcal S$ inside an equivalence class of $\mathcal R$ are parameterized by the points of $\Upsilon$ -- a Poisson point process on an amenable, doubly recurrent action of $G.$
\begin{figure}
    \includegraphics[scale=0.8]{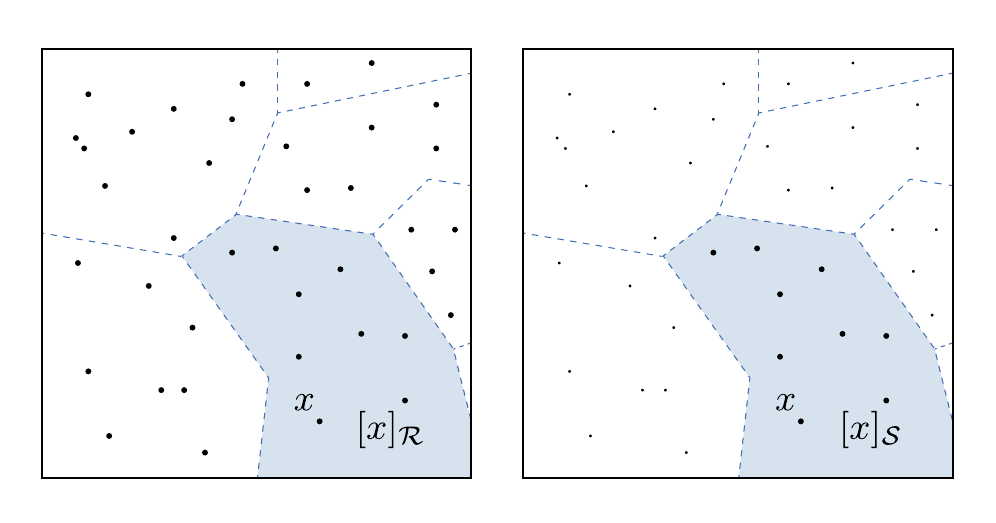}
    \caption{Equivalence class of $\mathcal R$ versus and equivalence class of $\mathcal S$, with $C^\upsilon_{c(\upsilon)}$ highlighted. Note points in other cells are still present, are no longer in the same $\mathcal{S}$-class.}     
    \label{fig:RestrictedRerooting}
\end{figure}

The remainder of this section is devoted to the proof that $\mathcal S$ is hyperfinite. To that end we will construct an auxiliary quasi-pmp hyperfinite equivalence relation $\mathcal S'$ and a class injective factor map $\psi$ from $\mathcal S$ to $\mathcal S'$. We will then show that one can ``lift'' the hyperfiniteness from $\mathcal S'$ to $\mathcal S$.

\begin{lemma}\label{quasipmpness}
    If $\Pi\in\MM(G)$ be a free invariant point process on $G$ with law $\mathcal L(\Pi)$ and let $G \acts (Y, \nu)$ be a (possibly infinite) measure preserving action, then $(\MMo \times Y, \Rel, \mathcal L(\Pi_\circ) \otimes \nu)$ is a measure preserving equivalence relation, where $\Rel$ denotes the restriction of the orbit equivalence relation of $G \acts \MM(G) \times Y$ to $\MMo(G) \times Y$
    $$(\omega,y)\sim_{\Rel} (\omega', y') \text{ iff } \exists g\in\omega \text{ such that } \omega'=g^{-1}\omega \text{ and } y'=g^{-1}y.$$ 

\end{lemma}

\begin{proof}
    The proof is a variant of  \cite[Prop. 3.10]{SamMiklos}. We need to show that for any non-negative measurable function $F\colon \Rel\to \mathbb R_{\geq 0}$ there is an equality 
    \begin{multline*}
    \int_{\MMo(G)}\int_Y \sum_{g\in \Pi_\circ} F((\Pi_\circ,y),(g^{-1}\Pi_\circ,g^{-1}y))d\nu(y)d\mathcal L(\Pi_\circ)\\=\int_{\MMo(G)}\int_Y \sum_{g\in \Pi_\circ} F((g^{-1}\Pi_\circ,g^{-1}y),(\Pi_\circ,y))d\nu(y)d\mathcal L(\Pi_\circ).
    \end{multline*}
    This identity is often called the mass transport principle and is a well known characterisation of measure preserving equivalence relations \cite[Example 9.9]{AR07}.    

Fix $U \subseteq G$ of unit volume. We will use $dh,d\gamma$ for the Haar measure on $G$.
\begin{align*}
    I &:= \int_{\MMo(G)}\int_Y \sum_{g \in \Pi_\circ} F((\Pi_\circ, y), (g^{-1}\Pi_\circ, g^{-1}y))d\nu(y)d\mathcal L(\Pi_\circ) \\
    &=\int_{\MMo(G)}\int_G \1[h \in U]\int_Y \sum_{g \in \Pi_\circ} F((\Pi_\circ, y), (g^{-1}\Pi_\circ, g^{-1}y))d\nu(y)dhd\mathcal L(\Pi_\circ) \\
    &=\frac{1}{\intensity(\Pi)}\int_{\MM(G)} \sum_{h \in \Pi} \1[h \in U]\int_Y \sum_{g \in h^{-1}\Pi} F((h^{-1}\Pi, y), (g^{-1}h^{-1}\Pi, g^{-1}y))d\nu(y)d\mathcal L(\Pi),
\end{align*}
where we have applied Theorem \ref{CLMM} with the function
\[f(\Pi_\circ, h) = \1[h \in U]\int_Y \sum_{g \in \Pi_\circ} F((\Pi_\circ, y), (g^{-1}\Pi_\circ, g^{-1}y)) d\nu(y).
\]
We now make the variable substitution $\gamma = hg$ eliminating $h$, and swap $g^{-1}$ for $g$:
\begin{align*}
    I &=\frac{1}{\intensity(\Pi)} \int_{\MM(G)}\sum_{\gamma \in \Pi} \int_Y \sum_{g^{-1} \in\gamma^{-1} \Pi}  \1[\gamma g^{-1} \in U] F((g\gamma^{-1}\Pi, y), (\gamma^{-1}\Pi, g^{-1}y))d\nu(y)d\mathcal L(\Pi)\\
    &=\frac{1}{\intensity(\Pi)} \int_{\MM(G)}\sum_{\gamma \in \Pi} \int_Y \sum_{g \in\gamma^{-1} \Pi}  \1[\gamma g \in U] F((g^{-1}\gamma^{-1}\Pi, y), (\gamma^{-1}\Pi, gy))d\nu(y)d\mathcal L(\Pi)\\
    &= \int_{\MM(G)}\int_G \int_Y \sum_{g \in \Pi_\circ}  \1[\gamma g \in U] F((g^{-1}\Pi_\circ, y), (\Pi_\circ, gy))d\nu(y)d\gamma d\mathcal L(\Pi_\circ),
\end{align*}
where we have again used Theorem \ref{CLMM}, this time with the function
\[
f(\gamma^{-1}\Pi, \gamma) = \int_Y \sum_{g \in \gamma^{-1}\Pi}  \1[\gamma g \in U] F((g\gamma^{-1}\Pi, y), (\gamma^{-1}\Pi, gy))d\nu(y).
\]
We now rearrange things back in place and use unimodularity of the Haar measure in the form $\lambda(Ug^{-1}) = 1$:
\begin{align*}
    I  &= \int_{\MM(G)} \sum_{g \in \Pi_\circ}\int_G \int_Y   \1[\gamma g \in U] F((g^{-1}\Pi_\circ, y), (\Pi_\circ, gy))d\nu(y)d\gamma d\mathcal L(\Pi_\circ) \\
    &= \int_{\MM(G)} \sum_{g \in \Pi_\circ}\int_Y F((g^{-1}\Pi_\circ, y), (\Pi_\circ, gy))\left(\int_G   \1[\gamma g \in U] d\gamma \right)d\nu(y) d\mathcal L(\Pi_\circ) \\
    &= \int_{\MM(G)} \sum_{g \in \Pi_\circ}\int_Y F((g^{-1}\Pi_\circ, y), (\Pi_\circ, gy))d\nu(y) d\mathcal L(\Pi_\circ) \\
    &= \int_{\MM(G)} \sum_{g \in \Pi_\circ} \int_YF((g^{-1}\Pi_\circ, g^{-1}y), (\Pi_\circ, y))d\nu(y)d\mathcal L(\Pi_\circ),
\end{align*}
where the last line follows from invariance of $\nu$. This is exactly the form we wanted.
\end{proof}

Let $\psi\colon \MMo(G,Z)\to \MMo(G)\times D$ be the map
$$\psi(\omega\times \upsilon):=(\omega, c(\upsilon)),\quad \omega\in \MMo(G), \upsilon\in Z,$$ where $c$ is the function returning the ``identity cell'' defined just above the equation (\ref{eq-DefS}). By Lemma \ref{lem-TBVornoi}, this map is $\mathcal L(\Pi_\circ)\times \mathcal L(\Upsilon)$ almost everywhere well defined and measurable. 

\begin{lemma}\label{lem-pushforwardmeasure}
    Let $\kappa=c_*\mathcal L(\Upsilon)$. Then, $\psi_*(\mathcal L(\Pi_\circ)\times \mathcal L(\Upsilon))=\mathcal L(\Pi_\circ)\times\kappa$ and $\kappa$ is absolutely continuous with respect to $\mu.$
\end{lemma}
\begin{proof}
    Since $\psi={\rm id}\times c,$ we have $\kappa=\mathcal L(\Pi_\circ)\times c_*\mathcal L(\Upsilon).$ To compute $c_* \mathcal L(\Upsilon)$ we will use the Mecke equation (\ref{Mecke}) for functions with values in $\MM(D)$ (recall that this is the space of locally finite Borel measures on $D$).  Define $K\colon \MM(D)\times D \to \MM(D)$ as
    $$K(\upsilon, f)=\begin{cases} \delta_{f} &\text{ if } f=c(\upsilon)\\ 0 & \text{otherwise.} \end{cases}$$
    Like many functions before, $K$ is only well defined for $\mathcal L(\Upsilon)$-almost every $\upsilon.$ This is not a problem for us since we are after an almost-everywhere statement. Note that $K$ is supported only on the pairs $(\upsilon, c(\upsilon)).$ 
    By definition of $K$, we have \begin{equation}\label{eq-PSF1}\int_{\MM(D)}\sum_{f\in \Upsilon}K(\Upsilon,f)d\mathcal L(\Upsilon)=\int_{\MM(D)}\delta_{c(\Upsilon)}d\mathcal L(\Upsilon)=c_*\mathcal L(\Upsilon).\end{equation}
    Let $E_t:=\{f\in D\mid f(o)\leq t\}.$ By (\ref{Mecke}), the left hand side of (\ref{eq-PSF1}) is equal to $$\int_{D}\int_{\MM(D)}K(\Upsilon+\delta_f,f)d\mathcal L(\Upsilon) d\mu(f)\geq \int_D \delta_f\mathbb P[\Upsilon\cap E_{f(o)}=\emptyset]d\mu(f)=e^{-\mu(E_{f(o)})}\mu(f).$$ The inequality comes from the fact that the event $\Upsilon\cap E_t=\emptyset$ guarantees that $c(\Upsilon+\delta_f)=f$. It can be strict because even without that event, it can still happen that $c(\Upsilon+\delta_f)=f$, as a consequence of the tie-breaking mechanism. Anyways, we have proved that $c_*\mathcal L(\Upsilon)$ is absolutely continuous with respect to $\mu$. 
\end{proof}
Let $\mathcal S'$ be the equivalence relation on $(\MMo(G)\times D, \mathcal L(\Pi_\circ)\times \kappa)$ defined by
\begin{equation}\label{eq-AuxRelationDef} (\omega,f)\sim_{\mathcal S'} (\omega',f') \text{ iff }\exists g \in \omega \text{ such that }\omega'=g^{-1}\omega \text{ and }f'=g^{-1}f.
\end{equation}
As $\mathcal L(\Pi_\circ)\times \kappa$ is the same measure class as $L(\Pi_\circ)\times \kappa$ we can use Lemma \ref{quasipmpness} to show that $\mathcal S'$ is a quasi pmp equivalence relation. We will now show that it is hyperfinite.
Recall that a quasi pmp equivalence relation is \emph{hyperfinite} if it can be expressed as an increasing union of equivalence sub-relations with almost all classes finite. This condition is equivalent to \emph{amenability} of the equivalence relation  (see \cite[Section II]{kechris2004topics} and \cite{CFW}). 

\begin{definition}[{\cite[II.9]{kechris2004topics}}]
    A quasi-pmp equivalence relation $(X, \Rel, \mu)$ is \emph{amenable} if there is a sequence $m_n : \Rel \to \RR$ of non-negative Borel functions such that
    \begin{itemize}
        \item $m^n_x \in \ell^1([x]_\Rel)$, where $m^n_x(y) = m^n(x,y)$ for $(x, y) \in \Rel$,
        \item $\norm{m^n_x}_1 = 1$, and
        \item there is a Borel $\Rel$-invariant set $A \subseteq X$ with $\mu(A) = 1$ and such that $\norm{m^n_x - m^n_y}_1 \to 0$ for all $x, y \in A$ with $(x, y) \in \Rel$.
    \end{itemize}
\end{definition}

We will also use the notion of an amenable actions. The original definition of Zimmer \cite{zimmer1984ergodic} is not well suited for our purposes so we will use an equivalent description.

\begin{definition}[{\cite[2B(2.1)]{kaimanovich2005amenability}, \cite[3.2.14]{ADR00}}]\label{def-amenability} 
    A quasi-pmp action $G \acts (X, \mu)$ is \emph{amenable} if there exists a sequence of maps $M^n$ from $X$ to $\mathcal{P}(G)$, the space of probability measures on $G$, which are \emph{asymptotically invariant} in the sense that
    \[
        \norm{gM^n_x - M^n_{gx}} \to 0 \text{ weakly.}
    \]
\end{definition}

\begin{lemma}\label{amenabilitylemma}
      Let $\Pi$ be a free invariant point process on $G$  and let $G \acts (Y, \nu)$ be quasi-pmp action. If the action $G \acts (\MM(G) \times Y, \mathcal L(\Pi) \times \nu)$ is amenable, then the quasi-pmp equivalence relation $(\MMo(G) \times Y, \Rel, \mathcal L(\Pi_\circ) \otimes \nu)$ given by
        $$(\omega,y)\sim_{\Rel} (\omega', y') \text{ iff } \exists g\in\omega \text{ such that } \omega'=g^{-1}\omega \text{ and } y'=g^{-1}y,$$ 
        is hyperfinite. 
\end{lemma}

\begin{proof}
    Let $M^n$ be an asymptotically invariant assignment of probability measures defined on $\MM(G) \times Y$, as in Definition \ref{def-amenability}. The idea of the proof is to use them to construct an asymptotically invariant assignment on $\MMo(G)\times Y$. We will make use of the fact that for a configuration $\omega \in \MMo(G)$ with trivial stabiliser we can identify the equivalence class $[(\omega,y)]$ with $\omega$ itself, by identifying $g\in \omega$ with $(g^{-1}\omega,g^{-1}y)$.
    
    For $g\in G$ let $C^\omega(g)$ be the Voronoi cell in the tessellation $\Vor(\omega)$ containing $g$. For any $g\in\omega$ put
    \[
        m^n_{\omega, y}(g^{-1}\omega) := \int_B  M^n_{h\omega, hy}(C^\omega(g)) dh,
    \]
    where $B \subseteq G$ is some fixed open set of finite measure and $dh$ is the Haar measure on $G$. 
 
    The first two conditions for amenability from Definition \ref{def-amenability} are clearly satisfied $\mathcal L(\Pi_\circ)\times \nu$-almost everywhere. We check the third. Fix $\alpha \in \omega$. Then
    \begin{align*}
        \norm{m^n_{\omega, y}(g^{-1} \omega) - m^n_{\alpha^{-1}\omega, \alpha^{-1}y}(\alpha^{-1}g^{-1}\omega)} &\leq \int_B \norm{M^n_{h\omega, hy}(C^\omega(g)) - M^n_{h\alpha^{-1}\omega, h\alpha^{-1}y}(C^{\alpha^{-1}\omega}(\alpha^{-1}g))} dh\\
        &=  \int_B \norm{M^n_{g\omega, gy}(C^\omega(g)) - M^n_{g\alpha^{-1}\omega, g\alpha^{-1}y}(\alpha^{-1}C^\omega(g))} dh \\
        &=  \int_B \norm{M^n_{g\omega, gy}(C^\omega(g)) - \alpha M^n_{g\alpha^{-1}\omega, g\alpha^{-1}y}(C^\omega(g))} dh,
    \end{align*}
    which converges to zero by the dominated convergence theorem, using that the integrand converges almost everywhere to zero by amenability of the action.
\end{proof}
\begin{lemma}\label{lem-AuxHyperfinite}
    The equivalence relation $(\MMo(G)\times D, \mathcal S', \mathcal L(\Pi_\circ)\times \kappa)$ is hyperfinite. 
\end{lemma}
\begin{proof}
    The action of $G$ on $(D,\mu)$ is amenable so the same holds true for $(D,\kappa)$, as $\kappa$ is absolutely continuous with respect to $\mu$ by Lemma \ref{lem-pushforwardmeasure}. By Lemma \ref{amenabilitylemma}, the relation  $(\MMo(G)\times D, \mathcal S', \mathcal L(\Pi_\circ)\times \kappa)$ is amenable, hence hyperfinite by \cite{CFW}.
\end{proof}

Finally we can prove that $\mathcal S$ is hyperfinite.
\begin{proposition}\label{prop-Shyperfinite}
    The restricted relation $\mathcal S\subset \mathcal R$ is hyperfinite.
\end{proposition}
\begin{proof}

We start by proving the following claim.

\textbf{ Claim.} Restriction of $\psi$ to $[\omega\times \upsilon]_{\mathcal S}$ is an injective map to $[\psi(\omega\times \upsilon)]_{\mathcal S'}$ for $\mathcal L(\Pi_\circ)\times \mathcal L(\Upsilon)$-almost every $\omega\times \upsilon.$

First we check that $\psi$ maps an equivalence class of $\mathcal S$ inside an equivalence class in $\mathcal S'$. Let $\omega\times \upsilon\in \MMo(G,[0,1]\times Z)$ and let $g\in \omega\cap C_{c(\upsilon)}^\upsilon$, that is, $g$ is an element of $\omega$ lying in the cell $C_{c(\upsilon)}^\upsilon$ of $\Vor(\upsilon)$ containing the identity. Then, almost surely, $1$ lies in the cell $g^{-1} C_{c(\upsilon)}^\upsilon$ which is the cell of $\Vor(g^{-1}\upsilon)$ containing the identity. It follows that $c(g^{-1}\upsilon)=g^{-1}c(\upsilon),$ so $$\psi(g^{-1}\omega\times g^{-1}\upsilon)=(g^{-1}\omega, g^{-1}c(\upsilon))\in [(\omega,c(\upsilon)]_{\mathcal S'}=[\psi(\omega\times \upsilon)]_{\mathcal S'}.$$ This proves that the image of almost every $\mathcal S$ class is contained in an $\mathcal S'$ class. 
For injectivity we argue as follows. Let $\pi_1\colon \MMo(G,[0,1]\times Z)\to \MMo(G)$ be the map which forgets the labeling and let $\pi_2\colon \MMo(G,[0,1])\times D\to\MMo(G)$ be the projection to the first factor followed by forgetting the IID label. 
A simple calculation using (\ref{eq-PalmClass}), shows that $\pi_1$ and $\pi_2$ both map classes of $\mathcal R$ and $\mathcal S$ (respectively) bjiectively to the classes of the Palm equivalence relation of $\Pi$ on $\MMo(\Pi)$, simply because $\Pi$ is a free point process. Since $\pi_2\circ \psi=\pi_1,$ and $\mathcal S\subset\mathcal R$, we deduce that $\psi$ must be injective on almost every equivalence class of $\mathcal S$ and the proof of the Claim is complete. 

Let $\mathcal E_n'\subset \mathcal S'$ be an increasing sequence of sub-relations such that $\mathcal S'=\bigcup_{n=1}^\infty \mathcal E_n'$ and $\mathcal E_n'$ is finite $\mathcal L(\Pi_\circ)\times \kappa$-almost everywhere. Define the sub-relation $\mathcal E_n\subset \mathcal S$ by $$\omega\times\upsilon\sim_{\mathcal E_n}\omega'\times\upsilon' \text{ iff } \omega\times\upsilon\sim_{\mathcal S}\omega'\times\upsilon' \text{ and } \psi(\omega\times\upsilon)\sim_{\mathcal E_n'}\psi(\omega'\times\upsilon').$$
By the Claim we just proved, the equivalence classes of $\mathcal E_n$ are finite $\mathcal L(\Pi_\circ)\times \mathcal L(\Upsilon)$-almost everywhere and $\mathcal S=\bigcup_{n=1}^\infty \mathcal E_n$, so $\mathcal S$ is indeed hyperfinite.
\end{proof}

\subsection{Construction of the cheap graphing}\label{sec-Graphing}

In this section we use the unbounded wall phenomenon (Theorem \ref{cellpairs}) or more generally the double recurrence condition (2) in Theorem \ref{thm-AbstractCostOne} to find a small cost graphing $\mathscr G$ on $\mathcal R$ that connects all pairs of the equivalence classes of $\mathcal S$ inside an equivalence class of $\mathcal R$. As a consequence, the equivalence relation generated by $\mathcal S$ and $\mathscr G$ is the whole $\mathcal R$. From the existence of such $\mathscr G$ and the hyperfiniteness of $\mathcal S$, it quickly follows that $\mathcal R$ is cost one (see Proposition \ref{prop-CheapGraphing}). Before we delve into the formal proof we present a sketch of the simplified construction of $\mathscr G$ in case of higher rank semisimple real Lie groups, where the argument has clear geometric intuition. In the construction we switch between the two equivalent perspectives of graphings and factor graphs; see Section \ref{sec-facgraph}.

\begin{proof}[Proof for semisimple Lie groups]As explained in (\ref{eq-PalmClass}), the equivalence class of $[\Pi_\circ\times\Upsilon]_\mathcal R$ can be identified with the set $\Pi_\circ$ with $1$ corresponding to the point $\Pi_\circ\times\Upsilon$. We choose a small $\delta>0$. To define a graphing on $\mathcal R$ we need to determine the set of outgoing arrows from $1$. In this case we construct the graphing $\mathscr G$ by connecting $1$ to all points $g\in\Pi_\circ$ such that $d(K,gK)\leq 1$ (this metric is the one defined in Section \ref{sec-SymSpace}) and the IID label of $g$ is less than $\delta$. Overall, this has the effect of adding ``stars'' of radius $1$ independently for each element of the equivalence class with probability $\delta$, as shown in Figure \ref{fig:AddedEdges1}.
\begin{figure}
    \includegraphics[trim=100 580 0 100,clip,scale=0.9]{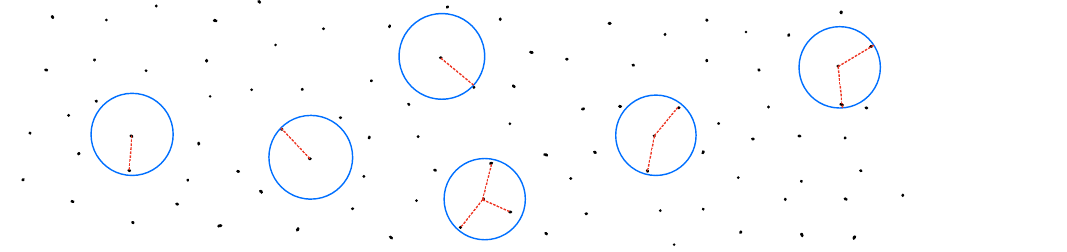}
    \caption{The graphing $\mathscr G_1$; centers of stars are chosen randomly via the IID labeling.}
    \label{fig:AddedEdges1}
\end{figure}

 We will use Theorem \ref{thm-cellpairsthick} to show that $\mathscr G$
 does indeed connect all the $\mathcal S$-equivalence classes inside $[\Pi\times \Upsilon]_\mathcal R.$ The equivalence classes in $\mathcal S$ are the intersections of $\Pi_\circ$ with the cells of the tie-breaking Voronoi tessellation $\Vor(\Upsilon)$. We therefore need to show $\mathscr G$ connects every pair of cells $C_{g_1U}^\Upsilon,C_{g_2U}^\Upsilon$. The unbounded wall property established in Theorem \ref{thm-cellpairsthick} implies that for every such pair of cells there is an unbounded set of points lying on the intersection of the boundaries of $C_{g_1U}^\Upsilon,C_{g_2U}^\Upsilon$ which do not witness any other cell within radius $1$.
\begin{figure}
    \includegraphics[scale=0.8]{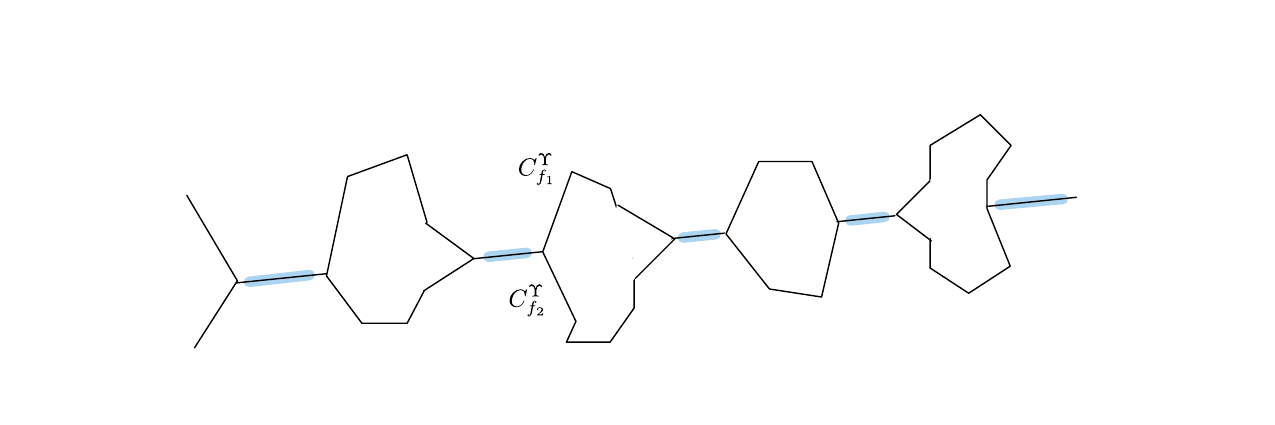}
    \caption{The thick wall set for a pair of cells, intersected with a flat.}
    \label{fig:ThickWall}
\end{figure}

By a simple Borel-Cantelli argument, it follows from Theorem \ref{thm-cellpairsthick} the process $\Pi_\circ$ contains infinitely many pairs of points $h_n\in C_{g_1U}^\Upsilon,h_n'\in C_{g_2U}^\Upsilon$ such that $d(h_nK,h_n'K)\leq 1.$ Using Borel-Cantelli once again, we see that infinitely many of these pairs must become connected by an edge in $\mathscr G_1$ (see Figure \ref{fig:AddedEdges2}). Taking $\delta$ to $0$ we can bring the cost of $\mathscr G_1$ arbitrary close to $0$. This concludes the proof in the case of symmetric spaces.
\end{proof}
\begin{figure}
    \includegraphics[scale=0.8]{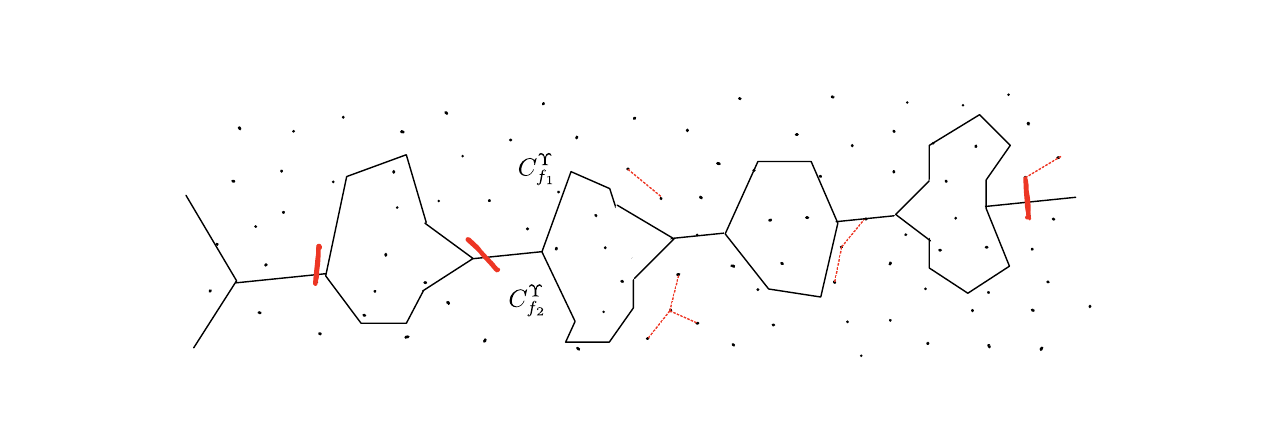}
    \caption{Bold edges of $\mathscr G_1$ connect the pair of cells.}
    \label{fig:AddedEdges2}
\end{figure}

We go back to the general setting of Theorem \ref{thm-AbstractCostOne}. Let $\omega$ be a (possibly marked) point configuration on $G$. For any $r>0$, and $g\in \omega$ we let $\bigstar_r(\omega,g)$ be the directed ``star graph'' on $\omega$ connecting $g$ to all elements $g'\in \omega$ such that $0<d(go,g'o)\leq r$. As before, let $\Pi$ be the IID Poisson point process on $G$ and let $\Upsilon$ be the IID Poisson on $D$ with mean measure $\mu$. Choose an $\varepsilon>0$ and a sequence $q_n>0,n\in\mathbb N$ such that $\sum_{n=0}^\infty q_n=1$ and $\sum_{n=1}^\infty q_n\vol(B(n))=\varepsilon.$ Using the IID labels on $\Pi$ we construct a new IID labeling $\xi\colon \Pi \to \mathbb N$, such that $\mathbb P[\xi(g)=n]=q_n$ for every $n\in\mathbb N,g\in \Pi$. As in (\ref{eq-PalmClass}) we identify the equivalence class $[\Pi_\circ\times\Upsilon]_\mathcal R$ with the set $\Pi_\circ$. 
We define the graphing $\mathscr G$ on the relation $\mathcal R$ by adding copies of $\bigstar_{\xi(g)}(\Pi_\circ,g)$ for every $g\in \Pi_\circ$ as shown in Figure \ref{fig:ConnectingG}.  In other words, $\mathscr G$ is given by 
$$\mathscr G(\Pi_\circ\times \Upsilon)=\{(g,g') \mid g,g\in \Pi_\circ \text{ and } 0<d(g,g')\leq \xi(g)\}=\bigcup_{g\in \Pi_\circ}\bigstar_{\xi(g)}(\Pi_\circ,g).$$
We note that the definition does not depend on the second factor, we only use the Poisson point process and the IID labeling.
\begin{figure}
    \includegraphics[trim=60 590 10 60,clip,scale=.8]{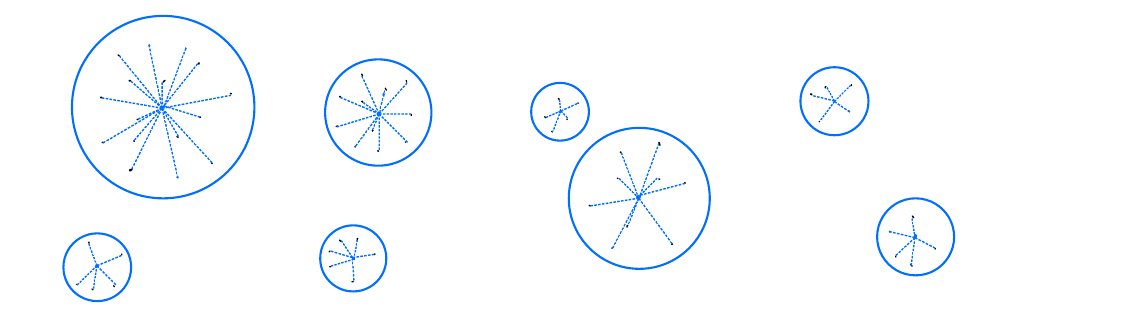}
    \caption{$\mathscr G$ is a union of star graphs for varying radii. Points of $\Pi$ omitted for clarity.}
    \label{fig:ConnectingG}
\end{figure}
The cost of $\mathscr G$ is given by 
$$\mathbb E[\deg_{\mathscr G(\Pi_\circ\times \Upsilon)}(1)]=\sum_{n=0}^\infty q_n \mathbb E[\deg \bigstar_n(\Pi_\circ,1)]=q_0\cdot 0+\sum_{n=1}^\infty q_n \vol(B(n))=\varepsilon.$$

\begin{definition}
    Say that $\omega\in \MM(G)$, or more generally a marked point process on $G$, is in a \emph{good position} if the following property holds almost surely for the underlying base process $\omega$. 
    Let $\Upsilon$ be the IID Poisson on $(D,\mu)$. For every $f_1,f_2\in \Upsilon$ such that $C_{f_i}^\Upsilon\cap \omega\neq \emptyset$ for $ i=1,2$, there is an $r>0$ such that the set $$\{g\in \omega\mid \bigstar_r(\omega, g)\text{ connects } C_{f_1}^\Upsilon\cap \omega \text{ and } C_{f_2}^\Upsilon\cap \omega\}$$ is infinite (see Figure \ref{fig:ConnectingG2}). 
\end{definition}
\begin{figure}
    \includegraphics[scale=0.8]{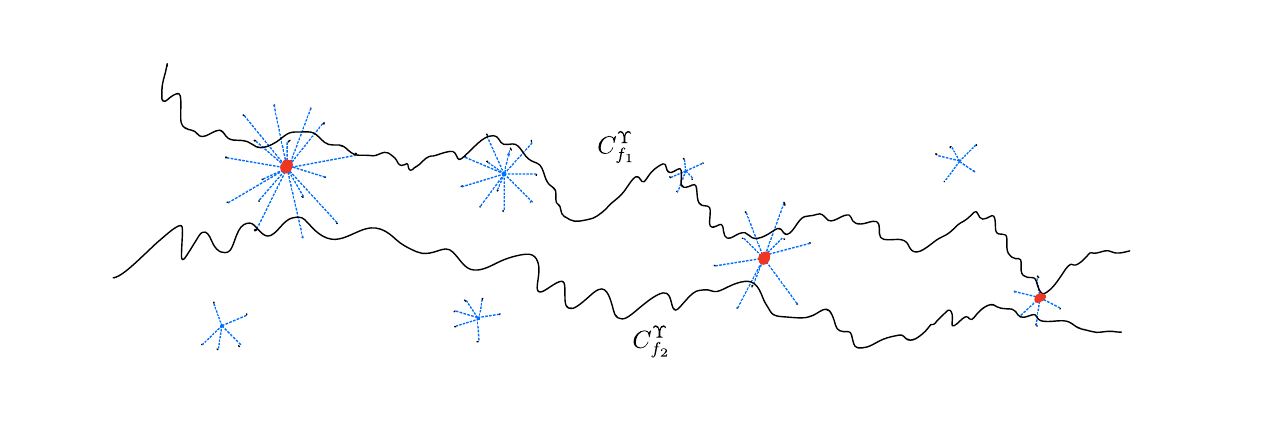}
    \caption{Points in $\omega$ such that $\bigstar_r(\omega, g)$ connects $C_{f_1}^\Upsilon\cap \omega$ and $ C_{f_2}^\Upsilon\cap \omega$.}
    \label{fig:ConnectingG2}
\end{figure}

\begin{lemma}\label{lem-PoissnGoodPosition}
    The Poisson point process $\Pi$ is in a good position almost surely. 
\end{lemma}

\begin{proof}
    Put $D':=D\times [0,1]$, $\mu':=\mu\times {\rm Leb}|_{[0,1]}$. Let $U_r\subset \MM(G)\times \MM(D,[0,1])\times D'\times D'$ be the set of quadruples $(\omega, \upsilon, f_1,f_2)$ such that there exists an $g_0\in B(r)\cap \omega$ such that $\bigstar_r(\omega, g_0)$ connects a pair of points in $\omega \cap C_{f_1}^{\upsilon\cup\{f_1,f_2\}}$ and $\omega\cap C_{f_2}^{\upsilon \cup\{f_1,f_2\}}$. In addition, define $U_\infty$ as the set of $(\omega, \upsilon, f_1,f_2)$ such that $C_{f_1}^{\upsilon\cup\{f_1,f_2\}}\cap \omega, C_{f_2}^{\upsilon\cup\{f_1,f_2\}}\cap\omega$ are non-empty. It is clear that $U_r$ is increasing in $r$ and $\limsup_{r\to\infty} U_r=U_\infty$, because any pair of points in $\omega \cap C_{f_1}^{\upsilon\cup\{f_1,f_2\}},\omega \cap C_{f_2}^{\upsilon\cup\{f_1,f_2\}}$ will become connected by some star graph $\bigstar_r(\omega,g_0), g_0\in\omega\cap B(r)$ once $r$ is big enough.
    The action of $G$ on $\MM(G)\times \MM(D,[0,1])\times D'\times D'$ is conservative by Lemma \ref{lem-ConservativeProduct}, so for $\Poisson(G,dg)\times\Poisson(D,\mu)^{IID}\times \mu'\times\mu'$-every point in $(\omega,\upsilon,f_1,f_2)\in U_r$ there is an unbounded set of $s\in G$ such that $s^{-1}(\omega,\upsilon,f_1,f_2)\in U_r$. 

    We note that $s^{-1}(\omega,\upsilon,f_1,f_2)\in U_r$ means that there is some point $g\in \omega$ such that $s^{-1}g\in B(r)$ and the star $\bigstar_r(s^{-1}\Pi,s^{-1}g)$ connects the sets $s^{-1}\omega \cap C_{s^{-1}f_1}^{s^{-1}\upsilon\cup\{s^{-1}f_1,s^{-1}f_2\}}$ and $s^{-1}\omega\cap C_{s^{-1}f_2}^{s^{-1}\upsilon \cup\{s^{-1}f_1,s^{-1}f_2\}}.$ Translating everything by $s$, we get that $\bigstar_r(\Pi,g)$ connects $\omega \cap C_{f_1}^{\upsilon\cup\{f_1,f_2\}}$ and $C_{f_2}^{\upsilon \cup\{f_1,f_2\}}$ and $g\in \omega\cap sB(r).$
    
    It follows that we can find a $2r$-separated sequence $s_n\in G$ such that for all $n\in\mathbb N$  the set $s_n B(r)\cap \Pi$ contains a point $g_n\in \omega$ where $\bigstar_r(\omega, g_n)$ connects $\omega \cap C_{f_1}^{\upsilon\cup\{f_1,f_2\}}$ and $\omega\cap C_{f_2}^{\upsilon\cup\{f_1,f_2\}}$. The separation condition guarantees that the points $g_n$ are pairwise distinct, so the set  $$\{g\in \omega\mid \bigstar_r(\omega, g)\text{ connects } C_{f_1}^{\upsilon\cup\{f_1,f_2\}}\cap \omega \text{ and } C_{f_2}^{\upsilon\cup\{f_1,f_2\}}\cap \omega\}$$ is infinite for almost every $(\omega,\upsilon,f_1,f_2)\in U_r$. Taking the union over a sequence of radii tending to infinity we get that for almost every $(\omega,\upsilon,f_1,f_2)\in U_\infty$ there is a radius $r$ such that the set $$\{g\in \omega\mid \bigstar_r(\omega, g)\text{ connects } C_{f_1}^{\upsilon\cup\{f_1,f_2\}}\cap \omega \text{ and } C_{f_2}^{\upsilon\cup\{f_1,f_2\}}\cap \omega\}$$ is infinite. 
    
    Finally, an application of the bivariate Mecke equation \ref{2Mecke} for the process $\Upsilon$ and the function $$F^\omega(\upsilon+\delta_{f_1}+\delta_{f_2},f_1,f_2)=1-\mathds 1_{U_\infty}(\omega,\upsilon,f_1,f_2),$$ means that for $\Poisson(G,dg)\times \Poisson(D,\mu)^{IID}$-almost every $\omega,\upsilon$ and every $f_1,f_2\in \upsilon$ which correspond to cells with non-empty intersection with $\omega$, the set 
    $$\{g\in \omega\mid \bigstar_r(\omega, g)\text{ connects } C_{f_1}^\upsilon\cap \omega \text{ and } C_{f_2}^\upsilon\cap \omega\}$$ is infinite.
\end{proof}
\begin{corollary}\label{cor-PalmGoodPosition}
    The Palm version $\Pi_\circ$ of the IID Poisson point process $\Pi$ is in good position almost surely. \end{corollary}
    \begin{proof}
    Let $A$ be the set of point configurations in good position. It is a $G$-invariant subset of $\MM(G)$, so $\mathbb P(\Pi_\circ\in A)=\mathbb P(\Pi\in A)=1$ by Lemma \ref{lem-PlamInvariant}.
    \end{proof}
\begin{lemma}\label{lem-ConnectingS}
    The graph $\mathscr G$ connects all the $\mathcal{S}$-equivalence classes inside the class $[\Pi_\circ\times\Upsilon]_{\mathcal R}$ almost surely.
\end{lemma}
\begin{proof}
    Since the classes of $\mathcal S$ contained in $[\Pi_\circ\times\Upsilon]_{\mathcal R}$ correspond to the cells in the Voronoi tessellation $\Vor(\Upsilon)$, we only need to show that for every pair of cells $C_{f_i}^\Upsilon, f_i\in \Upsilon, i=1,2$ such that $C_{f_i}^\Upsilon\cap \Pi_\circ\neq \emptyset, i=1,2$, there is a pair of points $g_1\in C_{f_1}^\Upsilon\cap \Pi_\circ, g_2\in C_{f_2}^\Upsilon\cap \Pi_\circ$ which are connected in $\mathscr G.$ By Corollary \ref{cor-PalmGoodPosition}, $\Pi_\circ$ is in a good position a.s., so for every $f_1,f_2\in \Upsilon$ such that $C_{f_i}^\Upsilon\cap \Pi_\circ\neq \emptyset, i=1,2$ there is an $r>0$ and infinitely many points $g\in \Pi_\circ$ such that star $\bigstar_r(\Pi, g)$ connects $C_{f_i}^\Upsilon\cap \Pi_\circ\neq \emptyset, i=1,2$. By Borel-Cantelli, infinitely many of these stars are contained in $\mathscr G$, so the lemma is proved.
\end{proof}

\begin{proposition}\label{prop-CheapGraphing}
The relation $\mathcal R$ admits a connecting graphing of cost $1+\varepsilon$ for every $\varepsilon>0.$
\end{proposition}
\begin{proof}
    The relation $\mathcal S$ is hyperfinite and aperiodic so it admits a graphing $\mathscr T$ of cost one. By Lemma \ref{lem-ConnectingS}, the union $\mathscr T\cup \mathscr G$ connects $\mathcal R$. The cost of $\mathscr T\cup \mathscr G$ is bounded by $\cost(\mathscr T)+\cost(\mathscr G)\leq 1+\varepsilon.$
\end{proof}

We expect that the following ergodic theoretic lemmas are well known to experts but we weren't able to locate a convenient reference.
\begin{lemma}\label{lem-ConservativeConditions}
Let $G\curvearrowright (Y,\nu)$ be measure preserving action of a unimodular lcsc group $G$. The following conditions are equivalent.
\begin{enumerate}
    \item The action is conservative.
    \item For every $U\in\mathcal B(Y)$ with $\nu(U)>0$ and almost every $y\in U$ we have 
    $\int_G \mathds{1}_{U}(gy) dg=+\infty.$
    \item For every non-negative function $f\in L^1(Y,\nu)$ with $\int_Yf(y)d\nu(y)>0$ we have 
    $$\int_{G}f(gy)dg=+\infty,$$ for almost every $y\in Y$ with $f(y)>0$.
\end{enumerate}
\end{lemma}  
\begin{proof}
    The implications (3) to (2) to (1) are clear. To deduce (3) from (2) one needs to apply (2) to the sets $U':=\{y\in Y\mid f(y)\geq \delta\}$ for all $\delta>0$. It remains to prove that (1) implies (2). 
    Let $U\subset Y$ be a positive measure subset and let $0<\varepsilon<1$. Since the action of $G$ on $L^2(Y,\nu)$ is continuous (see Appendix A.6 of \cite{kazhdanbook}), there exists a bounded symmetric open neighborhood $W$ of $1$ in $G$ such that $$\langle \mathds 1_U,\mathds 1_{gU}\rangle=\int_{Y}\mathds 1_U(y)\mathds 1_U(g^{-1}y)dy>(1-\varepsilon^2)\langle \mathds 1_U,\mathds 1_U\rangle=(1-\varepsilon^2)\nu(U),\text{ for every } g\in W.$$ It follows that 
    $$\int_W \langle \mathds 1_U,\mathds 1_{gU}\rangle dg=\int_U\left(\int_W\mathds 1_U(g^{-1}y)dg\right) d\nu(y)\geq (1-\varepsilon^2)m_G(W)\nu(U).$$
    Using Markov's inequality we deduce that there is a subset $U_\varepsilon \subset U$ such that $\int_W\mathds 1_U(g^{-1}y)dg\geq (1-\varepsilon) m_G(W)$ for all $y\in U_\varepsilon$ and $\nu(U_\varepsilon)\geq (1-\varepsilon)\nu(U)$. Using the fact that the action of $G$ on $(Y,\nu)$ is conservative, we get that for almost every $y\in U_\varepsilon$ the set $R_y:=\{g\in G\mid gy\in U_\varepsilon\}$ is unbounded. Choose a sequence $\{g_n\}\subset R_y$ such that the sets in $\{Wg_n\}$ are pairwise disjoint. We have 
    $$\int_G \mathds 1_U(gy)dg\geq \sum_{n=1}^\infty \int_{W}\mathds 1_U(gg_ny)dg\geq\sum_{n=1}^\infty (1-\varepsilon) m_G(W)=+\infty.$$
    This shows that for almost every $y\in U_\varepsilon$ the set of return times to $U$ is infinite measure. Since $U_\varepsilon\subset U$ and $\nu(U_\varepsilon)>(1-\varepsilon)\nu(U)$, we can finish the proof by letting $\varepsilon\to 0$ and using Fatou's lemma.
\end{proof}
\begin{lemma}\label{lem-ConservativeProduct}
    Let $G\curvearrowright (Y_1,\nu_1)$ be a pmp action of a unimodular lcsc group $G$ and let $G\curvearrowright (Y_2,\nu_2)$ be a conservative action. Then, the action of $G$ on $(Y_1\times Y_2, \nu_1\times \nu_2)$ is conservative.
\end{lemma}
\begin{proof}
Let $U\subset Y_1\times Y_2$ be a positive measure subset. By Lemma \ref{lem-ConservativeConditions} we need to prove that 
$$\int_G \mathds 1_U(gy_1,gy_2)dg=+\infty$$ for almost all $(y_1,y_2)\in U$. For the sake of contradiction, assume that there is a positive number $t$ and a positive measure subset $U'\subset U$ such that $\int_G \mathds 1_U(gy_1,gy_2)dg\leq t$ for all $(y_1,y_2)\in U'$. In particular, the function $F(y_1,y_2):=\int_G \mathds 1_{U'}(gy_1,gy_2)dg$ is almost everywhere bounded by $t$. We have
$$\int_{Y_1} F(y_1,y_2)d\nu_1(y_1)=\int_{Y_1}\int_G \mathds 1_{U'}(gy_1,gy_2)dgd\nu_1(y_1)=\int_G \int_{Y_1}\mathds 1_{U'}(y_1,gy_2)d\nu_1(y_1)dg,$$
by invariance of $\nu_1$. The leftmost integral is bounded by $t$, because $\nu_1$ is a probability measure and $F$ is pointwise bounded by $t$. On the other hand, by Lemma \ref{lem-ConservativeConditions} applied to function  $f(y_2):= \int_{Y_1}\mathds 1_{U'}(y_1,y_2)d\nu_1(y_1)$, the rightmost integral is infinite. This is a contradiction.
\end{proof}

\subsection[Proof of Theorem 7.1]{Proof of Theorem \ref{thm-AbstractCostOne}}\label{sec-AbstCostOneProof}

\begin{proof}By Corollary \ref{cor-FPOneReduction}, we need to show that the cost of a Poisson point process on $G$, of arbitrary intensity, is one. By Corollary \ref{WeakLimitTheorem} the cost of a Poisson point process is upper bounded by the cost of the product marking $\Pi\times \Upsilon$ (see Definition \ref{def-ProductMarking}) where $\Pi$ is the IID Poisson point process on $G$ of intensity $1$ and $\Upsilon$ is a weak factor of the IID Poisson point process. We wish to take $\Upsilon:=\Poisson(D,\mu)^{\rm IID}$, so we need to show it is a weak factor of $\Pi$. Recall that the measure $\mu$ in the corona action is the weak limit of measures $\mu_t$ on $D$ defined as $$\vol(B(t))^{-1}(\iota_t)_* \vol,$$ where $\vol$ is the pushforward of the Haar measure on $G$ and $\iota_t\colon X=G/K\to D$ is a sequence of proper maps $G\to D$ defined in Section \ref{sec-CoronaActions}. We define a sequence of factors $\Psi_t(\Pi)\in \MM(Z), Z:=\MM(D,[0,1])$. Put $\eta(t):=\vol(B(t))^{-1}$ and
$$\Psi_t(\Pi):=\{(\iota_t(gK),\ell(g)\eta(t)^{-1})\in D\times [0,1]\mid g\in \Pi, \ell(g)\leq \eta(t)\}\in \MM(D,[0,1]).$$
The process $\Psi_t(\Pi)$ is the IID labelling of the Poisson point process on $D$ with mean measure $\mu_t$. We warn the reader that when $X=G/K$ is countable, for example when $G$ is product of automorphism groups of trees, the base process of $\Psi_t(\Pi)$ is not simple, so it is a random multiset rather than a set. Still, by Lemma \ref{lem-PoissonContinuity} these processes weakly converge to $\Upsilon$ because the measures $\mu_t$ converge weakly-* to $\mu$.

By proposition \ref{prop-CheapGraphing}, the Palm equivalence relation $\mathcal R$ of $\Pi\times\Upsilon$ admits a generating graphing of cost $1+\varepsilon,$ for every $\varepsilon>0$. Thus $\cost(\Pi)\leq \cost(\Pi\times \Upsilon)=1$ and Theorem \ref{thm-AbstractCostOne} is proved.\end{proof}

\subsection{Comparison with countable groups}\label{sec-GaboriauComp}
Theorem \ref{thm-AbstractCostOne} should be compared with one of Gaboriau's criteria for fixed price $1$ in the context of countable groups.

\begin{theorem}[special case of {\cite[VI.24.(3)]{Gaboriau}}]\label{thm-GaboriauFP1}
Let $\Gamma$ be a finitely generated countable group containing a subgroup $\Lambda$ such that $\Lambda\cap \Lambda^\gamma$ is infinite for every $\gamma\in \Gamma$ and $\Lambda$ has fixed price one. Then $\Gamma$ has fixed price one.
\end{theorem}

Both Theorem \ref{thm-AbstractCostOne} and Theorem \ref{thm-GaboriauFP1} use a double recurrence assumption. Indeed, the condition that $\Lambda\cap \Lambda^\gamma$ is infinite for every $\gamma\in\Gamma$ is equivalent to the condition that the action $\Gamma\curvearrowright (\Gamma/\Lambda)^2$ with counting measure is conservative. Amenability for the action $\Gamma\curvearrowright \Gamma/\Lambda$ would translate to $\Lambda$ being amenable which is a stronger condition than fixed price one. Below we sketch how to adapt the proof of Theorem \ref{thm-AbstractCostOne} to reprove Theorem \ref{thm-GaboriauFP1}. 

 \begin{proof}By \cite{AbertWeiss}, the maximal cost of a free action of $\Gamma$ is realized by the the Bernoulli shift $[0,1]^\Gamma$. Let $\mathcal R$ be the orbit equivalence relation on $[0,1]^\Gamma$ and let $\mathcal S\subset \mathcal R$ be the orbit equivalence relation of $\Lambda$. We will use the letter $\xi$ to denote the labeling $\xi\in [0,1]^\Gamma$. Equip $\Gamma$ with any left-invariant word metric $d$ coming from a finite generating set and choose a function $s\colon [0,1]\to \mathbb N$ such that $\sum_{n=1}^\infty \mathbb P[s(\xi(g))=n]|B(n)|\leq \varepsilon$. Define the graphing $\mathscr G$ as 
$$\mathscr G(\xi)=\{(g,g')\mid g,g'\in \Gamma \text{ and } 0<d(g,g')\leq \xi(g)\}=\bigcup_{g\in \Gamma}\bigstar_{\xi(g)}(\Gamma,g).$$
We claim that any two equivalence classes of $\mathcal S$ inside the equivalence class of $\mathcal R$ are connected by $\mathscr G$. Since the equivalence classes of $\mathcal S$ in $[\xi]_{\mathcal R}$ are parameterized by points in $(\Gamma/\Lambda)$ and the action of $\Gamma$ on $(\Gamma/\Lambda)^2$ is conservative, one can repeat the proof of Lemma \ref{lem-PoissnGoodPosition} to show that almost surely for any pair of orbits $g_1\Lambda, g_2\Lambda\in G/\Lambda$ there is an $r>0$ and infinitely many $g\in \Gamma$ such that 
$\bigstar_r(\Gamma, g)$ connects $g_1\Lambda$ and $g_2\Lambda$. We note that this can be also verified directly using the condition that $\Lambda^{g_1}\cap \Lambda^{g_2}$ is infinite.  

To construct a cheap graphing for $\mathcal R$ we use the fixed price one property for $\mathcal S$ to find a graphing $\mathscr T$ of cost $1+\varepsilon$ which connects $\mathcal S$ and add the graphing $\mathscr G$ to get a generating graphing for $\mathcal R$ of cost at most $1+2\varepsilon.$
\end{proof}

\section{Further questions}\label{Sec-Questions}
\subsection{Weakly proper actions}
Our proof relies on the fact that a Poisson point process on $D$ with mean measure $\mu$ is a weak factor of a Poisson point process on $G$. This fact in turn follows from the weak-* convergence $\mu_t\to\mu$ where $\mu_t\in\MM(D)$ are suitably renormalized push-forwards of the Haar measure on $G$. In the case of real semisimple Lie group the pair $(D,\mu)$ turns out to be isomorphic to $(G/U,dg/du)$ where $U$ is  the kernel of the modular character of a minimal parabolic subgroup of $G$. This motivates the following definition. Let $G\curvearrowright(D,\mu)$ be a measure preserving action of $G$. 

\begin{definition}\label{def-WeaklyProper} We say that $G\curvearrowright (D,\mu)$ is \emph{weakly proper} if there exists a locally compact space $Y$ with a continuous $G$-action and a sequence of points $y_n$ and real numbers $\eta_n>0, n\in\mathbb N$ such that  
\begin{enumerate}
    \item the maps $\iota_n\colon G\to Y$ given by $\iota_n(g):=gy_n$ are proper,
    \item measures $\eta_n(\iota_n)_*dg$ converge weakly-* to a locally finite measure $\nu$,
    \item $(Y,\nu)\simeq (D,\mu)$ as measure preserving $G$-actions.
\end{enumerate}
\end{definition}
If $(D,\mu)$ is weakly proper, then the same argument as the one used in the proof of Theorem \ref{thm-AbstractCostOne} shows that the IID Poisson point process on $D$ with mean measure $\mu$ is a weak factor of the IID Poisson on $G$, so one immediately can generalize Theorem \ref{thm-AbstractCostOne} to weakly proper actions. 
\begin{question}
Are all amenable actions weakly proper? Are there any non-amenable weakly proper actions?  In particular, if $H\subset G$ is amenable closed subgroup is the action $G\curvearrowright G/H$ weakly proper?
\end{question}

One can also ask a similar question for factor maps.

\begin{question}
    Does there exist a factor map from a Poisson point process on $G$ to a Poisson point process on $G/H$ for an amenable closed subgroup $H\subset G$?
\end{question}

\subsection{Ideal dual graphs}

Let $G$ be semisimple real Lie group with the associated symmetric space $X$. Recall that the ideal Poisson-Voronoi tessellation is given by  $\Vor(\{\beta_{gU}\mid gU\in \Upsilon'\})$ (see Definition \ref{IPVTmodel}), where $\Upsilon'$ is the Poisson point process on $G/U$ with any $G$-invariant mean. It divides the symmetric space into countably many cells, indexed by a Poisson point process $\Upsilon'$ on the space $(G/U,dg/du).$
We define the \emph{ideal dual graph} $\mathcal G(X)$ as the random graph with vertex set $V(\mathcal G(X)):=\Upsilon'$ and edge set $$E(\mathcal G(X)):=\{(g_1U,g_2U)\in \Upsilon'\times \Upsilon'\,|,\ \overline{C_{g_1U}}\cap \overline{C_{g_2U}}\neq \emptyset \}.$$ In other words, the vertices corresponding to two cells are connected if the cells ``touch''. One can imagine analogous definition for any lcsc group $G$ acting properly isometrically on a metric space, with an associated corona action. We are not the first ones to propose the study of these graphs, similar direction of research was announced in \cite{ACELU23} in the case of $X=\mathbb H^n.$ 

If $G$ is higher-rank, then Theorem \ref{cellpairs} implies that the graph $\mathcal G(X)$ is the countable complete graph, which makes it a rather uninteresting object. For any infinite graph $G$, the Bernoulli bond percolation $G[p], p>0$ is the random graph where each edge is kept independently with probability $p$.  The critical probability $p_u$ is defined as 
$$p_u(G):=\inf\{p\in [0,1] \mid G[p] \text{ has unique infinite connected component} \}.$$
This definition is usually stated for bounded degree graphs but it makes perfect sense even when the degrees are allowed to be infinite. The graph $\mathcal G(X)$ is an infinite degree random graph canonically attached to $X$. It would be very interesting to estimate $p_u(\mathcal G(X))$ for rank one symmetric spaces.

\begin{question}\label{q-DualGraphCritical}
Are there rank one symmetric spaces $X$ with $p_u(\mathcal G(X))=0$ almost surely? What about the $p_u(\mathcal G(\mathcal T_g))$ for the mapping class group acting on the Teichm{\"u}ller space $\mathcal T_g$ of surface of genus $g\geq 2$?
\end{question}

If $X$ is higher rank then $p_u(\mathcal G(X))=0$, since the graph is complete. We believe that if the answer to Question \ref{q-DualGraphCritical} is positive, then the proof of Theorem \ref{thm-AbstractCostOne} still goes through, in the sense that if $\mathcal S$ and $\mathcal R$ are the equivalence relations defined in Section \ref{sec-Subrelation} and $\mathscr G$ is the graphing constructed in Section \ref{sec-Graphing}, then $\mathscr G$ together with any generating graphing of $\mathcal S$ still do generate $\mathcal R$ (although it may no longer be true that every pair of classes get connected by $\mathscr G$). 
\begin{question}
Suppose $p_u(\mathcal G(X))=0$ almost surely. Then $G$ has fixed price one, in the sense of Section \ref{sec-cost}.
\end{question}

\subsection{Fixed price for hyperbolic space}

\begin{question}
    Does $\Isom(\HH^3)$ have fixed price one?
\end{question}

In \cite{AN} it is shown that if a certain lattice in $\Isom(\HH^3)$ has fixed price one, then the rank versus Heegaard genus conjecture is false in a strong sense. The rank versus Heegaard genus conjecture is known to be false, as shown in \cite{li}. 

\subsection{Convergence of tessellations}
Although we defined the ideal Poisson-Voronoi tessellation on $X$, we do not show that $\Vor(\Pi_\eta)$ converge to the ideal Poisson-Voronoi tessellation in any precise way. In fact, it is a good question what should be the topology on the space of tessellations and whether one can use our computation to show the weak-* convergence of distributions of $\Vor(\Pi_\eta)$ in such topology. Can one say something quantitative about the speed of such convergence? A positive result in this direction might lead to a quantitative version of Theorem \ref{rankgradient} (compare with \cite{Fraczyk}, where an upper quantitative bound on $\dim_{\mathbb F_p} H_1(\Gamma, \mathbb F_p)$ is proved for higher rank lattices). It would be interesting to know whether one can match the bound on the rank with the known bounds on the mod-$2$ homology groups. 
\begin{question}
Let $G$ be a simple real Lie group of real rank $d\geq 2$. Is it true that any lattice $\Gamma<G$ with injectivity radius at least $R$ satisfies $d(\Gamma)\ll \vol(\Gamma\bs G) R^{\frac{1-d}{2}}?$
\end{question}

\bibliographystyle{amsalpha}
\bibliography{bibliography}

\end{document}